\documentclass{amsart}
\usepackage{amsmath,amssymb,amsfonts,amsthm}
\usepackage{fouridx}
\usepackage{a4wide}
\usepackage{verbatim}
\usepackage[all]{xy}
\usepackage[usenames,dvipsnames]{color}
\usepackage[verbose,colorlinks=true,linktocpage=true,linkcolor=blue,citecolor=blue]{hyperref}
\usepackage{pgfplots}
\usepackage{tikz}
\usepackage{mathtools}
\usepackage{mhchem}
\usepackage{tensor}
\usepackage{tikz}

\theoremstyle{definition}
\newtheorem{definition}{Definition}[section]

\theoremstyle{plain}
\newtheorem{theorem}[definition]{Theorem}
\newtheorem{proposition}[definition]{Proposition}
\newtheorem{lemma}[definition]{Lemma}
\newtheorem{corollary}[definition]{Corollary}
\newtheorem{example}[definition]{Example}
\theoremstyle{remark}
\newtheorem{remark}[definition]{Remark}

\numberwithin{equation}{section}

\newcommand{\bbe}{\mathbb{E}}
\newcommand{\bbf}{\mathbb{F}}
\newcommand{\bbk}{\mathbb{K}}
\newcommand{\bbi}{\mathbb{I}}

\newcommand{\RNum}[1]{\uppercase\expandafter{\romannumeral #1\relax}}
\allowdisplaybreaks[1]

\begin{document}

\title{On the Harish-Chandra homomorphism for  quantum superalgebras}
\date{}
\author{Yang Luo}
\address{School of Mathematics, University of Science and Technology of China, Hefei 230026, China.}
\email{yangluo@mail.ustc.edu.cn}
\author{Yongjie Wang}
\address{Department of Mathematics, Hefei University of Technology, Hefei 230009, China.}
\email{wyjie@mail.ustc.edu.cn}
\author{Yu Ye}
\address{School of Mathematics, University of Science and Technology of China, Hefei 230026, China.}
\email{yeyu@ustc.edu.cn}
\dedicatory{}
\maketitle
\begin{abstract}
In this paper, we introduce the Harish-Chandra homomorphism for the quantum superalgebra $\mathrm{U}_q(\mathfrak{g})$ associated with a simple basic Lie superalgebra $\mathfrak{g}$ and give an explicit description of its image. We use it to prove that the center of $\mathrm{U}_q(\mathfrak{g})$ is isomorphic to a subring of the ring  $J(\mathfrak{g})$ of exponential super-invariants in the sense of Sergeev and Veselov, establishing a Harish-Chandra type theorem for $\mathrm{U}_q(\mathfrak{g})$. As a byproduct, we obtain a basis of the center of $\mathrm{U}_q(\mathfrak{g})$ with the aid of quasi-$R$-matrix. 
\bigskip

\bigskip

\end{abstract}

\tableofcontents

\section{Introduction}
\label{sec:intr}

Harish-Chandra introduced a homomorphism, known as the \textit{Harish-Chandra homomorphism}, for semisimple Lie algebras in the study of unitary representations of semisimple Lie groups in 1951 \cite{HC}. Later on, the  Harish-Chandra homomorphism was developed for infinite dimensional Lie algebras \cite{Kac3,Naito}, Lie superalgebras \cite{Kac3, Sergeev0, Sergeev} and quantum groups \cite{BY,CK,JoLe,Rosso,Tanisaki2}. 

Knowledge about the invariants and the center of quantum superalgebras is not merely of mathematical interest but is also physically important.  On one hand, the study of the centralizer of a (quantized) universal enveloping (super)algebra is an indispensable part of its representation theory. On the other hand, the study of physical theories to a large extent involves the exploration of the invariants of the symmetry algebras, which usually correspond to certain physical observables.   The \textit{Harish-Chandra homomorphism} reveals many connections between the center of the enveloping (super)algebras or their quantization and the (super)symmetric polynomials as well as the highest weight representations of the corresponding algebras, and it has been one of the most inspiring themes in Lie theory.

Let $\mathfrak{g}$ be a semisimple Lie algebra (resp., a basic Lie superalgebra) over $\mathbb{C}$ with triangular decomposition $\mathfrak{g}=\mathfrak{n}^-\oplus\mathfrak{h}\oplus\mathfrak{n}^+$, where $\mathfrak{h}$ is a Cartan subalgebra and $\mathfrak{n}^+$ (resp., $\mathfrak{n}^-$) is the positive (resp., negative) part of $\mathfrak{g}$ corresponding to a positive root system $\Phi^+$. By the PBW Theorem, we have the decomposition $\mathrm{U}(\mathfrak{g})=\mathrm{U}(\mathfrak{h})\oplus \left(\mathfrak{n}^-\mathrm{U}(\mathfrak{g})+\mathrm{U}(\mathfrak{g})\mathfrak{n}^+\right)$. Let $\pi\colon\mathrm{U}(\mathfrak{g})\rightarrow\mathrm{U}(\mathfrak{h})=\mathrm{S}(\mathfrak{h})$ be the associated projection. The restriction  of $\pi$ to the center $\mathcal{Z}(\mathrm{U}(\mathfrak{g}))$ of $\mathrm{U}(\mathfrak{g})$ is an algebra homomorphism,  and the composite $\gamma_{-\rho}\circ\pi\colon\mathcal{Z}(\mathrm{U}(\mathfrak{g}))\rightarrow\mathrm{S}(\mathfrak{h})$ of  $\pi$  with a ``shift'' by the Weyl vector $\rho$ is called the \textit{Harish-Chandra homomorphism} of $\mathrm{U}(\mathfrak{g})$. The famous Harish-Chandra isomorphism theorem says that $\gamma_{-\rho}\circ\pi$ induces an isomorphism from $\mathcal{Z}(\mathrm{U}(\mathfrak{g}))$ to the algebra of $W$-invariant polynomials if $\mathfrak{g}$ is a semisimple Lie algebra or the algebra of $W$-invariant supersymmetric polynomials if $\mathfrak{g}$ is a classical Lie superalgebra. More details can be found in \cite[Chapter 11]{Carter} for classical Lie algebras,  and \cite[Section 2.2]{ChengWang}, \cite[Chapter 13]{Musson} for classical Lie superalgebras.

Quantum groups, first appearing in the theory of quantum integrable system, were formalized independently by Drinfeld and Jimbo as certain special Hopf algebras around 1984 \cite{Drinfeld, Jimbo}, including deformations of universal enveloping algebras of semisimple Lie algebras and coordinate algebras of the corresponding algebraic groups. In 1990, by the aid of the Universal $R$-matrix, Rosso \cite{Rosso} defined a significant ad-invariant bilinear form on $\mathrm{U}_q(\mathfrak{g})$ at a generic value $q$ of the parameter. The form, often referred to as the Rosso form or quantum Killing form, could also be obtained by using Drinfeld double construction. Tanisaki \cite{Tanisaki1,Tanisaki2} described this form by skew-Hopf pairing between the positive part and the negative part of the quantum algebra and obtained the quantum analogue of the Harish-Chandra isomorphism between $\mathcal{Z}(\mathrm{U}_q(\mathfrak{g}))$ and the subalgebra of $W$-invariant Laurent polynomials. As an application, the generators and the defining relations for $\mathcal{Z}(\mathrm{U}_q(\mathfrak{g}))$ have been obtained in \cite{Baumann,DZ,LXZ}.

Associated with the generalization of Lie algebras to Lie superalgebras,  many researchers have investigated the quantization of universal enveloping superalgebras in recent years. Drinfeld-Jimbo quantum superalgebras \cite{Yamane1, ZGB3} are a class of quasi-triangular Hopf superalgebras, depending on the choice of Borel subalgebras,  which were introduced in the early 1990s. As a supersymmetric version of quantum groups, quantum superalgebras have a natural connection with supersymmetric integrable lattice models and conformal field theories. They have been found applications in various areas, including in the study of the solution of quantum Yang-Baxter equation \cite{GZB}, construction of topological invariants of knots and 3-mainfolds \cite{Zhang0,ZGB1,ZGB2} and so on. Quantum superalgebras have been investigated extensively by many authors from different perspectives in aspects such as Serre relations, PBW basis, universal R-matrix \cite{Yamane1, Yamane2}, crystal bases \cite{Kwon1,Kwon2}, invariant theory \cite{LZZ}, highest weight representations \cite{Ge2, Zhang1, Zhang2} and so on. 

Comparing to Lie (super)algebras and quantum groups, the following questions for quantum superalgebras are natural and fundamental: What is the Harish-Chandra isomorphism for quantum superalgebras? How to determine the center of quantum superalgebras? The purpose of the present work is to answer these questions.

Let $\mathfrak{g}$ be a simple basic Lie superalgebra, except for $A(1,1)$, with root system $\Phi=\Phi_{\bar{0}}\cup\Phi_{\bar{1}}$,  and let $\mathrm{U}=\mathrm{U}_q(\mathfrak{g})$ be the associated quantum superalgebra over $k=K(q^{\frac{1}{2}})$, where $K$ is a field of characteristic 0 and $q$ is an indeterminate. The Weyl group and Weyl vector are defined by $W$ and $\rho$, respectively. Let $\Lambda=\left\{\lambda\in \mathfrak{h}^*\left |\frac{2(\lambda,\alpha)}{(\alpha,\alpha)}\in\mathbb{Z},~\forall\alpha\in\Phi_{\bar{0}}\right.\right\}$ be the integral weight lattice,  where $\mathfrak{h}^*$ is the dual space of the cartan subalgebra $\mathfrak{h}$. In this paper, all finite-dimensional $\mathrm{U}$-modules are of type {\bf 1} with the weights contained in $\frac{1}{2}\mathbb{Z}\Phi$.

The Cartan subalgebra $\mathrm{U}^0$ is the group ring of $\mathbb{Z}\Phi$ with basis $\left\{\left.\bbk_{\mu}\right |\mu\in \mathbb{Z}\Phi\right\}$ and multiplication $\bbk_{\mu}\bbk_{\nu}=\bbk_{\mu+\nu}$ for all $\mu,\nu\in \mathbb{Z}\Phi$. For each $\lambda\in \Lambda$, we define an automorphism $\gamma_{\lambda}\colon\mathrm{U}^0\rightarrow\mathrm{U}^0$ by $\gamma_{\lambda}(\bbk_{\mu})=q^{(\lambda,\mu)}\bbk_{\mu}$ for all $\mu\in \mathbb{Z}\Phi$.

Let $\Pi$ be the simple roots of distinguished borel subalgebra if $\mathfrak{g}=A(n,n)$ with $n\neq 1$, and let $\mathbb{Z}\tilde{\Phi}$ be the free abelian group with $\mathbb{Z}$-basis $\Pi$.  We 
set
$$Q=
\begin{cases}
\mathbb{Z}\tilde{\Phi}, &\text{ for } \mathfrak{g}= A(n,n),\\
\mathbb{Z}\Phi, & \text{ otherwise}.
\end{cases}
$$
Thus, the root system of $A(n,n)$ is $\mathbb{Z}{\Phi}=\mathbb{Z}\tilde{\Phi}/\mathbb{Z}\gamma$ for some $\gamma$. Define the standard partial order relation on $Q$ by $\lambda\leqslant\mu \Leftrightarrow \mu-\lambda\in \sum_{i\in\mathbb{I}}\mathbb{Z}_+\alpha_i$.
 
There is a triangular decomposition $\mathrm{U}=\mathrm{U}^-\mathrm{U}^0\mathrm{U}^+$, where $\mathrm{U}^-$ and $\mathrm{U}^+$ are the negative and positive parts of $\mathrm{U}$, respectively. Clearly $\mathrm{U}$, $\mathrm{U}^-$ and $\mathrm{U}^+$ are all $Q$-graded algebras. The triangular decomposition implies a direct sum decomposition
\begin{equation*}
  \mathrm{U}_0=\mathrm{U}^0\oplus\mathop{\bigoplus}\limits_{\nu>0}\mathrm{U}^-_{-\nu}\mathrm{U}^0\mathrm{U}^+_{\nu},
\end{equation*}
where $ \mathrm{U}_0$ is the component of degree $0$ of $\mathrm{U}$, and $\mathrm{U}^+_{\nu}$ (resp., $\mathrm{U}^-_{-\nu}$) is the component of degree $\nu$ (resp., $-\nu$) of $\mathrm{U}^+$ (resp., $\mathrm{U}^-$) for $\nu>0$. Note that the projection map $\pi\colon\mathrm{U}_0\rightarrow\mathrm{U}^0$ is an algebra homomorphism.  From now on, we do not make a distinction between the element in $\mathbb{Z}\Phi$ and $Q$ if no confusion emerges. 

We observe that the center  $\mathcal{Z}(\mathrm{U}_q(\mathfrak{g}))$ of $\mathrm{U}_q(\mathfrak{g})$ is contained in $\mathrm{U}_0$ by Corollary \ref{centerU_0}. Inspired by the quantum group case, we define the \textit{Harish-Chandra homomorphism} $\mathcal{HC}$ of $\mathrm{U}_q(\mathfrak{g})$ to be the composite 
\[\mathcal{HC}\colon\mathcal{Z}(\mathrm{U}_q(\mathfrak{g}))\hookrightarrow \mathrm{U}_0\xrightarrow{\pi} \mathrm{U}^0 \xrightarrow{\gamma_{-\rho}}  \mathrm{U}^0.
\]

To establish the Harish-Chandra type theorem for quantum superalgebras, we need to describe the image of $\mathcal{HC}$. Recall that a root $\alpha\in\Phi$ is isotropic if $(\alpha,\alpha)=0$, and the set of isotropic roots is denoted by ${\Phi}_{\mathrm{iso}}$. Set
\begin{equation*}
(\mathrm{U}^0_{\mathrm{ev}})^{W}_{\mathrm{sup}}=\Bigg\{\mathop{\sum}\limits_{\mu\in 2\Lambda\cap \mathbb{Z}\Phi} a_{\mu}\bbk_{\mu}\in\mathrm{U}^0\Bigg |a_{w\mu}=a_{\mu}, ~\forall w\in W; \mathop{\sum}\limits_{\mu \in A^{\alpha}_{\nu}}a_{\mu}=0,~\forall\alpha\in{\Phi}_{\mathrm{iso}}~\text{with} ~(\nu,\alpha)\neq0\Bigg\},
\end{equation*}
where $A^{\alpha}_{\nu}=\left\{\left.\nu+n\alpha~\right |n\in\mathbb{Z}\right\}$ for each $v\in\Lambda$ and $\alpha\in{\Phi}_{\mathrm{iso}}$.
The notation is consistent with the one in quantum groups \cite[Section 6.6]{Jan} and the one in classical Lie superalgebras \cite[Subsection 2.2.4]{ChengWang}. Then the image of $\mathcal{HC}$ is contained in $(\mathrm{U}^0_{\mathrm{ev}})^{W}_{\mathrm{sup}}$, which is essentially derived from character formulas of Verma modules and simple modules of $\mathrm{U}_q(\mathfrak{g})$, certain automorphisms of $\mathrm{U}_q(\mathfrak{g})$ and nontrivial homomorphisms between Verma modules; see Lemmas \ref{U0}, \ref{Uev}, \ref{Usup}.

Now we can state our main theorem.

\vspace{5pt}
\noindent {\bf Theorem A.} \emph{ The Harish-Chandra homomorphism $\mathcal{HC}$ for the quantum superalgebra $\mathrm{U}_q(\mathfrak{g})$ associated to a simple basic Lie superalgebra $\mathfrak{g}$ induces an  isomorphism from $\mathcal{Z}(\mathrm{U}_q(\mathfrak{g}))$ to $(\mathrm{U}^0_{\mathrm{ev}})^{W}_{\mathrm{sup}}$.}
\vspace{1.5pt}

The Lie superalgebra $\mathfrak{g}=A(1,1)$ is very special, and the image of $\mathcal{HC}$ is contained in $(\mathrm{U}^0_{\mathrm{ev}})^{W}_{\mathrm{sup}}$, we think $\mathcal{HC}$ is not an isomorphism; see Remark \ref{A(1,1)}.

We noticed that Batra and Yamane have introduced the generalized quantum group $U(\chi,\pi)$ associated with a bi-character $\chi$ and established a Harish-Chandra type theorem for describing its (skew) center in \cite{BY}.  While the quantum superalgebra $\mathrm{U}_q(\mathfrak{s})$ of a basic classical Lie superalgebra $\mathfrak{s}$ has been identified with a subalgebra of $\hat{U}^{\sigma}$ involving a new generator $\sigma$, so does the image of Harish-Chandra homomorphism (see \cite{BY}).  It is not 
known whether one can derive the Harish-Chandra type theorem for quantum superalgebra $\mathrm{U}_q(\mathfrak{s})$ from \cite{BY}.

As an application, we obtain  a basis of $\mathcal{Z}(\mathrm{U}_q(\mathfrak{g}))$ by using quasi-R-matrix.

\vspace{5pt}
\noindent {\bf Theorem B.}  \emph{ The center $\mathcal{Z}(\mathrm{U}_q(\mathfrak{g}))$ has a basis, which is constructed by using quasi-R-matrix and parametrized by $\left\{\left.\lambda\in \Lambda\cap\frac{1}{2}\mathbb{Z}\Phi\right |\mathrm{dim}L(\lambda)<\infty  \right\}$, where $L(\lambda)$ is an irreducible module of Lie superalgebra $\mathfrak{g}$ with the highest weight $\lambda$.}
\vspace{5pt}

To prove Theorem A, it suffices to prove $\mathcal{HC}$ is injective and the image $\mathcal{HC}$ is equal to $(\mathrm{U}^0_{\mathrm{ev}})^{W}_{\mathrm{sup}}$. For the injectivity, we establish a key Proposition \ref{3.3} by using the character formula of typical finite-dimensional modules of $\mathrm{U}_q(\mathfrak{g})$, which is a super version of Tanisaki's result for quantum algebras \cite[Section 3.2]{Tanisaki2}.
 
The difficulty is proving the image of $\mathcal{HC}$ is equal to $(\mathrm{U}^0_{\mathrm{ev}})^{W}_{\mathrm{sup}}$. With the help of the well-known classical Lie theory of semisimple Lie algebras, one can prove the surjectivity for quantum groups by using induction on the weights. However, the similar technique does not apply to quantum superalgebras, where one encounters two main obstacles: \\
\indent 1): There are infinite dominant weights less than a given dominant weight with respect to the standard partial order if $\mathfrak{g}$ is of type  \RNum1.\\
\indent  2): Besides the condition of the $\Phi_{\bar{0}}^+$-dominant integral, an extra condition for the finiteness of the dimension of an irreducible $\mathfrak{g}$-module $L(\lambda)$ is that $\lambda$ satisfies the hook partition if $\mathfrak{g}$ is of type  \RNum2. 

We notice that the close connection between $K(\mathfrak{g})$,  $J(\mathfrak{g})$ and $K(\mathrm{U}_q(\mathfrak{g}))$ will help us to overcome the obstacles, where $K(\mathfrak{g})$ and $K(\mathrm{U}_q(\mathfrak{g}))$ are the Grothendieck rings of $\mathfrak{g}$ and $\mathrm{U}_q(\mathfrak{g})$, respectively, and $J(\mathfrak{g})$ is the ring of Laurent supersymmetric polynomials (also called \textit{ring of exponential super-invariants} in \cite{SerVes}). Recall Sergeev and Veselov's isomorphism \cite{SerVes} $\mathrm{Sch}\colon K(\mathfrak{g})\xrightarrow{\sim} J(\mathfrak{g})$, where $\mathrm{Sch}$ is the supercharater map, and the injective algebra homomorphism $\jmath\colon K(\mathfrak{g})\hookrightarrow K(\mathrm{U}_q(\mathfrak{g}))$ is  induced by taking deformation,  which is implicitly given by Geer in \cite{Ge2}. The main ingredient of 
our proof can be illustrated in the following commutative diagram:
\begin{equation*}
\xymatrix@C=2pt{
&k\otimes_{\mathbb{Z}}K(\mathrm{U}_q(\mathfrak{g}))&&k\otimes_{\mathbb{Z}}K(\mathfrak{g})\ar@{_{(}->}[ll]_{k\otimes_{\mathbb{Z}}\jmath}\ar[dr]_{\cong}^(.6){k\otimes_{\mathbb{Z}}\mathrm{Sch}}   &\\
  {k\otimes_{\mathbb{Z}}}{K_{\mathrm{ev}}(\mathrm{U}_q(\mathfrak{g}))}\ar@{^{(}->}[ur]\ar@{-->}[dr]_(.5){\Psi_{\mathcal{R}}}&&k\otimes_{\mathbb{Z}}K_{\mathrm{ev}}(\mathfrak{g})\ar@{^{(}->}[ur]\ar@{_{(}->}[ll]\ar@{-->}[dr]_(.4){\cong} &&{k\otimes_{\mathbb{Z}}}J(\mathfrak{g})\\
&\mathcal{Z}(\mathrm{U}_q(\mathfrak{g}))\ar@{-->}[rr]^(.42){\mathcal{HC}}
&&(\mathrm{U}^0_{\mathrm{ev}})_{\mathrm{sup}}^W=k\otimes_{\mathbb{Z}}J_{\mathrm{ev}}(\mathfrak{g})\ar@{^{(}-->}[ur]_{\iota}&
}
\end{equation*}


First,  we identify $(\mathrm{U}^0_{\mathrm{ev}})_{\mathrm{sup}}^W$ with a subring of ${k\otimes_{\mathbb{Z}}}J(\mathfrak{g})$ by some $\iota$,  and the key idea is to reformulate $(\mathrm{U}^0_{\mathrm{ev}})_{\mathrm{sup}}^W$ as $k\otimes_{\mathbb{Z}} J_{\mathrm{ev}}(\mathfrak{g})$, which embeds into ${k\otimes_{\mathbb{Z}}}J(\mathfrak{g})$ in a natural way; see Equation \ref{J_q(g)} and Proposition \ref{J(U)}. One can prove that under the isomorphism $k\otimes_{\mathbb{Z}}\mathrm{Sch}$, the ring $(\mathrm{U}^0_{\mathrm{ev}})_{\mathrm{sup}}^W$ is isomorphic to ${k\otimes_{\mathbb{Z}}}{K_{\mathrm{ev}}(\mathfrak{g})}$, where $K_{\mathrm{ev}}(\mathfrak{g})$ is a subring of $K(\mathfrak{g})$ consisting of modules with all weights contained in $\Lambda\cap\frac{1}{2}\mathbb{Z}\Phi$.

Second,  $\jmath$ induces an injection  $k\otimes_{\mathbb{Z}}K_{\mathrm{ev}}(\mathfrak{g})\hookrightarrow k\otimes_{\mathbb{Z}}K_{\mathrm{ev}}(\mathrm{U}_q(\mathfrak{g}))$, where  $K_{\mathrm{ev}}(\mathrm{U}_q(\mathfrak{g}))$ is the subring of $K(\mathrm{U}_q(\mathfrak{g}))$ consisting of modules with  all weights contained in $\Lambda\cap\frac{1}{2}\mathbb{Z}\Phi$.

Third,  analogous to quantum groups \cite[Chapter 6]{Jan}, \cite{Rosso, Tanisaki1}, by using the Rosso form and the quantum supertrace for quantum superalgebras, we define a linear map $\Psi_\mathcal{R}\colon {k\otimes_{\mathbb{Z}}}{K_{\mathrm{ev}}(\mathrm{U}_q(\mathfrak{g}))} \to \mathcal{Z}(\mathrm{U}_q(\mathfrak{g}))$; see Proposition \ref{Psi_R}. This involves lengthy computations and some subtle constructions. We remark that $\Psi_{\mathcal{R}}$ is an algebra isomorphism, but not in an obvious way.

Now the surjectivity of $\mathcal{HC}$ follows from the commutative diagram easily. Moreover, we show that $\mathcal{HC}\circ \Psi_{\mathcal{R}}$ is injective, and combined with the injectivity of $\mathcal{HC}$, we can prove that homomorphisms occurring in the bottom left parallelogram are all isomorphisms of algebras. Consequently, the restriction $\jmath\colon K_{\mathrm{ev}}(\mathfrak{g})\to K_{\mathrm{ev}}(\mathrm{U}_q(\mathfrak{g}))$ is an isomorphism. 

By definition, $k\otimes_{\mathbb{Z}}K_{\mathrm{ev}}(\mathfrak{g})$ has a basis $\left\{\left.[L(\lambda)] \right | \lambda\in \Lambda\cap \frac{1}{2} \mathbb{Z}\Phi,~ \dim L(\Lambda)<\infty\right\}$ and $k\otimes_{\mathbb{Z}}K_{\mathrm{ev}}(\mathrm{U}_q(\mathfrak{g}))$ has a basis $\left\{\left.[L_q(\lambda)]\right | \lambda\in \Lambda\cap \frac{1}{2} \mathbb{Z}\Phi,~ \dim L_q(\Lambda)<\infty\right\}$, where $L(\Lambda)$ and $L_q(\Lambda)$  are the irreducible $\mathfrak{g}$-module and the irreducible $\mathrm{U}_q(\mathfrak{g})$-module with the highest weight $\lambda$, respectively. We remark that if $\lambda\in \Lambda\cap \frac{1}{2} \mathbb{Z}\Phi$, then $\dim L(\Lambda)<\infty$ if and only if $\dim L_q(\Lambda)<\infty$. Then the desired basis of $\mathcal{Z}(\mathrm{U}_q(\mathfrak{g}))$ in Theorem B is obtained by applying the isomorphism $\Psi_{\mathcal{R}}$, and here we rely heavily on an alternating construction of  $\Psi_{\mathcal{R}}$ by using quasi-R-matrix as in \cite{GZB0}. 

The paper is organized as follows: In Section \ref{Liequantum}, we review some basic facts related to contragredient Lie superalgebras and quantum superalgebras. In Section \ref{Representations}, we show several useful results on representations of quantum superalgebras, which seem to be well-known among experts. In particular, we give a super version of a Tanisaki's result for quantum superalgebras (see Proposition \ref{3.3}), which has been used to prove the injectivity of $\mathcal{HC}$. In Section \ref{Drinfelddouble}, we recall that the quantum superalgebra can be realized as a Drinfeld double. As a consequence, a non-degenerate ad-invariant bilinear form on $\mathrm{U}_q(\mathfrak{g})$ (Theorem \ref{adinvariant}) is obtained, which serves for proving the surjectivity of $\mathcal{HC}$. In Section \ref{HChomomorphism}, first we define the Harish-Chandra homomorphism for quantum superalgebras and prove its injectivitity. Then we prove that the image of $\mathcal{HC}$ is contained in $(\mathrm{U}^0_{\mathrm{ev}})_{\mathrm{sup}}^W$
and then explicitly describe its image $J_{\mathrm{ev}}(\mathfrak{g})$, which will be used to prove our main theorem for quantum superalgebras; see Theorem A. In  Section \ref{Centers}, we construct an explicit central element $C_{M}$ associated with each finite-dimensional $\mathrm{U}_q(\mathfrak{g})$-module $M$ by using the quasi-R-matrix of quantum superalgebras.  As an application of the Harish-Chandra theorem, we obtain a basis for the center of quantum superalgebras.

\medskip
\noindent{\bf Notations and terminologies:}

Throughout this paper, we will use the standard notations $\mathbb{Z}$, $\mathbb{Z}_+$ and $\mathbb{N}$ that represent the sets of integers, non-negative integers and positive integers, respectively. The Kronecker delta $\delta_{ij}$ is equal to $1$ if $i=j$ and $0$ otherwise.

We write $\mathbb{Z}_2=\left\{\bar{0},\bar{1}\right\}$. For a homogeneous element $x$ of an associative or Lie superalgebra, we use $|x|$ to denote the degree of $x$ with respect to the $\mathbb{Z}_2$-grading. Throughout the paper, when we write $|x|$ for an element $x$,  we will always assume that $x$ is a homogeneous element and automatically extend the relevant formulas by linearity (whenever applicable). All modules of Lie superalgebras and quantum superalgebras are assumed to be $\mathbb{Z}_2$-graded. The tensor product of two superalgebras $A$ and $B$ carries a superalgebra structure by 
$$(a_1 \otimes b_1)\cdot(a_2\otimes b_2) =(-1)^{|a_2||b_1|}a_1a_2\otimes b_1b_2.$$

\section{Lie superalgebras and quantum superalgebras}\label{Liequantum}
\subsection{Lie superalgebras}
\quad Let $\mathfrak{g}=\mathfrak{g}_{\bar{0}}\oplus\mathfrak{g}_{\bar{1}}$ be a finite-dimensional complex simple Lie superalgebra of type A-G such that $\mathfrak{g}_{\bar{1}}\ne 0$, and let $\Pi=\left\{\alpha_1, \alpha_2,\ldots \alpha_r\right\}$, with $r$ the rank of $\mathfrak{g}$, be the simple roots of $\mathfrak{g}$. Also let $(A,\tau)$ be the corresponding Cartan matrix, where $A=(a_{ij})$ is a $r\times r$ matrix and $\tau$ is a subset of $\bbi=\left\{1,2,\ldots,r\right\}$ determining the parity of the generators. Kac showed that the Lie superalgebra  $\mathfrak{g}(A,\tau)$ is characterized by its associated Dynkin diagrams (equivalent Cartan matrix $A$,  and a subset $\tau$); see \cite{Kac}. These Lie superalgebras are called basic.  For convenience (see remark \ref{nonconju}), we will restrict our attention to the simplest case and only consider root systems related to a special Borel sub-superalgebra with at most one odd root, called distinguished root system, denoted by $\mathfrak{g}(A,\{s\})$ or simply $\mathfrak{g}$ in no confusion.  The explicit description of root systems can be found in Appendix \ref{appendix}.  The Cartan matrix $A$ is symmetrizable,  that is,  there exist non-zero rational numbers $d_1,d_2,\ldots d_r$ such that $d_ia_{ij}=d_ja_{ji}$.  Without loss of  generality,  we assume $d_1=1$, since there exists a unique (up to constant factor) non-degenerate supersymmetric invariant bilinear form $(\text{-},\text{-})$ on $\mathfrak{g}$ and the restriction of this form to Cartan subalgebra $\mathfrak{h}$ is also non-degenerate. Let $\Phi$ be the root system of $\mathfrak{g}$, and define the sets of even and odd roots, respectively, to be $\Phi_{\bar{0}}$ and $\Phi_{\bar{1}}$.   
In order to define quantum superalgebra associated with a Lie superalgebra $\mathfrak{g}(A,\{s\})$,  we first review the generators-relations presentation of Lie superalgebra $\mathfrak{g}(A,\{s\})$ given by  Yamane \cite{Yamane2} and Zhang \cite{Zhang3}.  
\begin{definition} 
\cite[Theorem 3.4]{Zhang3} Let $(A,\{s\})$ be the Cartan matrix of a contragredient Lie superalgebra in the distinguished root system. Then $\mathrm{U}(\mathfrak{g}(A,\{s\}))$ (simplify for $\mathrm{U}(\mathfrak{g})$) is generated by $e_i,  f_i, h_i (i=1,2,\ldots r)$ over $\mathbb{C}$, where  $e_s$ and $f_s$ are odd and the rest are even,  subject to the quadratic relations:
\begin{align}
[h_i,  h_j]=0,\quad [h_i,  e_j]=a_{ij}e_j,\quad  [h_i,  f_j]=-a_{ij}f_j,\quad  [e_i,  f_j]=\delta_{ij}h_j,
\end{align}
and the additional linear relation $\mathop{\sum}\limits_{i=1}^rJ_ih_i=0$ if $\mathfrak{g}=A\left(\frac{r-1}{2},\frac{r-1}{2}\right)$ for odd $r$, where $J=\left(J_1,J_2,\cdots, J_r\right)$ such that $JA=0$ (more explicitly, $J=\left(1,2,\cdots,\frac{r+1}{2},-\frac{r-1}{2},-\frac{r-3}{2},\cdots,-1\right)$), 
and the \textit{standard Serre relations}
\begin{align*}
e_s^2=f_s^2=0,\quad\text{ if } (\alpha_s, \alpha_s)=0,\\
(\mathrm{ad} e_i)^{1-a_{ij}} e_j=(\mathrm{ad} f_i)^{1-a_{ij}} f_j=0,  &\text{ if } i\ne j,  \text{ with } a_{ii}\ne 0, ~\text{or} ~a_{ij}=0,
\end{align*}
and \textit{higher order Serre relations}
\begin{eqnarray}
\begin{aligned}
[e_s,  [e_{s-1},  [e_ s,  e_{s+1}]]]=0,\quad [f_s,  [f_{s-1},  [f_s,  f_{s+1}]]]=0,
\end{aligned}
\end{eqnarray}
if the Dynkin diagram of $A$ contains a full sub-diagram of the form
\begin{center}
\begin{picture}(85, 15)(0, 7)
\put(10, 10){\circle{10}} 
\put(0, -2){\tiny $s-1$} 
\put(15,10){\line(1, 0){20}}
\put(42, 10){\makebox(0,0)[c]{$\bigotimes$ }}
\put(38, -2){\tiny $s$} 
\put(45, 10){\line(1, 0){20}} 
\put(70,10){\circle{10}} 
\put(62, -2){\tiny $s+1$} 
\put(77, 7){,}
\end{picture}
\quad or \quad
\begin{picture}(85, 15)(0, 7)
\put(10, 10){\circle{10}} 
\put(0, -2){\tiny $s-1$} 
\put(15,10){\line(1, 0){20}} 
\put(42, 10){\makebox(0,0)[c]{$\bigotimes$ }}
\put(38, -2){\tiny $s$} 
\put(45, 9){\line(1, 0){20}} 
\put(45, 11){\line(1, 0){20}} 
\put(52, 7.3){$>$}
\put(70,10){\circle{10}} 
\put(62, -2){\tiny $s+1$} 
\put(77, 7){.}
\end{picture}
\end{center}
\end{definition}

We refer the reader to \cite{Zhang3} for undefined terminology and the presentation for each simple basic Lie superalgebra in an arbitrary root system. 

\subsection{Quantum superalgebras} 

Let $k=K(q^{\frac{1}{2}})$,  where $K$ is a field of characteristic 0 and $q$ is an indeterminate,  and we set $q_i=q^{d_i}$, then $q_{i}^{a_{ij}}=q_{j}^{a_{ji}}$ for all $i,j=1,2,\ldots, r$. Define
\begin{align*}
\begin{bmatrix} m\\ n\end{bmatrix}_q=\begin{cases}
\prod\limits_{i=1}^n\frac{(q^{m-i+1}-q^{i-m-1})}{(q^i-q^{-i})}, &\text{ if  }m>n>0,\\
1, &\text{if } n=m,  0.
\end{cases}
\end{align*}

\begin{definition}{\cite{Ge1,  LZZ,Yamane1}}\label{quantumsupergroups}
Let $(A,\{s\})$ be the Cartan matrix of a simple basic Lie superalgebra $\mathfrak{g}$ in the distinguished root system. The quantum superalgebra $\mathrm{U}_q(\mathfrak{g})$ is defined over $k$ in $q$ generated by $\mathbb{K}_i^{\pm1}, \mathbb{E}_i, \mathbb{F}_i, i\in\bbi$ (all generators are even except for $\bbe_s$ and $\bbf_s$, which are odd), subject to the following relations:
\begin{align}\label{def}
&\bbk_i\bbk_j=\bbk_j\bbk_i,\ \quad\quad\quad\quad \bbk_i\bbk_i^{-1}=\bbk_i^{-1}\bbk_i=1,\\\ 
&\bbk_i\bbe_j\bbk_i^{-1}=q^{(\alpha_i,\alpha_j)}\bbe_j,\, \quad \bbk_i\bbf_j\bbk_i^{-1}=q^{-(\alpha_i,\alpha_j)}\bbf_j, \\
&\bbe_i\bbf_j-(-1)^{|\bbe_i||\bbf_i|}\bbf_i\bbe_i=\delta_{ij}\frac{\bbk_i-\bbk_i^{-1}}{q_i-q_i^{-1}},\\
&\sum\limits_{k=0}^{1-a_{ij}}(-1)^k{\begin{bmatrix} 1-a_{ij}\\ k\end{bmatrix}}_{q_i}\bbe_i^{1-a_{ij}-k}\bbe_j\bbe_i^k=0,\text{ if } i\ne j,  \text{ with } a_{ii}\ne 0,  \text{ or }a_{ij}=0,\label{serreE1}\\
&\sum\limits_{k=0}^{1-a_{ij}}(-1)^k{\begin{bmatrix} 1-a_{ij}\\ k\end{bmatrix}}_{q_i}\bbf_i^{1-a_{ij}-k}\bbf_j\bbf_i^k=0, \text{ if } i\ne j,  \text{ with } a_{ii}\ne 0,  \text{ or }a_{ij}=0,\label{serreF1}\\
&(\bbe_s)^2=(\bbf_s)^2=0,\text{ if } a_{ss}=0,
\end{align}
\text{ and higher order quantum Serre relations, and} 
$$\mathop{\prod}\limits_{i=1}^r\bbk_i^{d_iJ_i}=1 \text{ if }\mathfrak{g}=A\left(\frac{r-1}{2}, \frac{r-1}{2}\right) \text{ for odd }r.$$
For the distinguished root data \cite[Appendix A.2.1]{Zhang3},  higher order Serre relations appear if the Dynkin diagram contains a sub-diagram of the following types:
\begin{enumerate}
\item
\begin{picture}(85, 15)(0, 7)
\put(10, 10){\circle{10}}
\put(0, -2){\tiny $s-1$}
\put(40,10){\makebox(0,0)[c]{$\bigotimes$}}
\put(15, 10){\line(1,0){20}} 
\put(38, -2){\tiny $s$} 
\put(45, 10){\line(1, 0){20}} 
\put(70,10){\circle{10}} 
\put(62, -2){\tiny $s+1$} 
\put(77, 7){,}
\end{picture}
the higher order quantum Serre relations are
\vspace{2mm}
\begin{eqnarray}
\begin{aligned}
\bbe_s\bbe_{s-1,s,s+1}+\bbe_{s-1,s,s+1}\bbe_s=0,\quad \bbf_s\bbf_{s-1,s,s+1}+\bbf_{s-1,s,s+1}\bbf_s=0;\label{serreE2}
\end{aligned}
\end{eqnarray}
\item \begin{picture}(85, 15)(0, 7)
\put(10, 10){\circle{10}}
\put(0, -2){\tiny $s-1$}
\put(15, 10){\line(1, 0){20}}
\put(42, 10){\makebox(0,0)[c]{$\bigotimes$ }}
\put(38, -2){\tiny $s$}
\put(45, 11){\line(1, 0){20}}
\put(45, 9){\line(1, 0){20}}
\put(53, 7.3){$>$}
\put(70, 10){\circle{10}}
\put(62, -2){\tiny $s+1$}
\put(77, 7){,}
\end{picture}
the higher order quantum Serre relations are
\vspace{2mm}
\begin{eqnarray}
\begin{aligned}
\bbe_s \bbe_{s-1; s; s+1} +  \bbe_{s-1; s; s+1}\bbe_s=0,\quad
\bbf_s \bbf_{s-1; s; s+1} +  \bbf_{s-1; s; s+1}\bbf_s=0;\label{serreE3}
\end{aligned}
\end{eqnarray}

\item \begin{picture}(100, 25)(0, 7)
\put(10, 10){\circle{10}}
\put(0, -2){\tiny $s-1$}
\put(15, 10){\line(1, 0){20}}
\put(42, 10){\makebox(0,0)[c]{$\bigotimes$ }}
\put(38, -2){\tiny $s$}
\put(45, 10){\line(2, 1){20}}
\put(45, 10){\line(2, -1){20}}
\put(70, 20){\circle{10}}
\put(77, 18){\tiny $s+1$} 
\put(70, 0){\circle{10}}
\put(77, -3){\tiny $s+2$}
\put(95, 7){,}
\end{picture}
the higher order quantum Serre relations are
\vspace{4mm}
\begin{eqnarray}
\begin{aligned}
\bbe_s \bbe_{s-1; s; s+1} +  \bbe_{s-1; s; s+1}\bbe_s=0,\quad
\bbf_s \bbf_{s-1; s; s+1} +  \bbf_{s-1; s; s+1}\bbf_s=0,\\
\bbe_s \bbe_{s-1; s; s+2} +  \bbe_{s-1; s; s+2}\bbe_s=0, \quad
\bbf_s \bbf_{s-1; s; s+2} +  \bbf_{s-1; s; s+2}\bbf_s=0;\label{serreE4}
\end{aligned}
\end{eqnarray}
where
\[
\begin{aligned}
\bbe_{s-1; s; j} =&\bbe_{s-1}\left(\bbe_s \bbe_{j}-q_{j}^{a_{j s}} \bbe_{j} \bbe_s\right)-
q_{s-1}^{a_{s-1, s}}\left(\bbe_s \bbe_{j}-q_{j}^{a_{j s}} \bbe_{j} \bbe_s\right) \bbe_{s-1},\\
\bbf_{s-1; s; j} =&\bbf_{s-1}\left(\bbf_s \bbf_{j}-q_{j}^{a_{j s}} \bbf_{j} \bbf_s\right)-
q_{s-1}^{a_{s-1, s}}\left(\bbf_s \bbf_{j}-q_{j}^{a_{j s}} \bbf_{j} \bbf_s\right) \bbf_{s-1}.
\end{aligned}
\]
\end{enumerate}
For the other root data of $\mathfrak{g}$,  the higher order quantum Serre relations vary considerably with the choice of the root datum; thus,  we will not spell them out explicitly here.
\end{definition}
\begin{remark}\label{nonconju}
The definition of the quantum superalgebra above is dependent on the choice of the Borel subalgebras. Although the quantum superalgebras defined by non-conjugacy Borel subalgebras of a Lie superalgebra are not isomorphic as Hopf superalgebras, they are isomorphic as superalgebras; see \cite{KT} or \cite[Proposition 7.4.1]{Yamane3}.
\end{remark}
There is a unique automorphism $\omega$ of $\mathrm{U}_q(\mathfrak{g})$ such that $\omega(\bbe_i)=(-1)^{|\bbe_i|}\bbf_i$, $\omega(\bbf_i)=\bbe_i$ and $\omega(\bbk_{i})=\bbk_{i}^{-1}$ for $i\in\bbi$. The quantum superalgebra $\mathrm{U}_q(\mathfrak{g})$ has the structure of a $\mathbb{Z}_2$-graded Hopf algebra.
The co-multiplication
$$\Delta\colon\mathrm{U}_q(\mathfrak{g})\rightarrow \mathrm{U}_q(\mathfrak{g})\otimes \mathrm{U}_q(\mathfrak{g})$$
is given by
\begin{align}
\Delta(\bbe_i)=\bbk_i\otimes \bbe_i+\bbe_i\otimes 1,\quad
\Delta(\bbf_i)=1\otimes\bbf_i+ \bbf_i\otimes \bbk_i^{-1},\quad
\Delta(\bbk_i^{\pm1})=\bbk_i^{\pm1}\otimes \bbk_i^{\pm1},
\end{align}
for $i\in\bbi$ and the co-unit $\varepsilon\colon\mathrm{U}_q(\mathfrak{g})\rightarrow k$ is defined by
$$\varepsilon(\bbe_i)=\varepsilon(\bbf_i)=0,\quad \varepsilon(\bbk_i^{\pm1})=1, \text{ for } i\in\bbi, $$
then the corresponding antipode $S\colon\mathrm{U}_q(\mathfrak{g})\rightarrow \mathrm{U}_q(\mathfrak{g})$ is given by
\begin{align}
S(\bbe_i)=-\bbk_i^{-1}\bbe_i,\quad
S(\bbf_i)=-\bbf_i\bbk_i,\quad
S(\bbk_i^{\pm1})&=\bbk_i^{\mp1},\text{ for } i\in\bbi,\end{align}
which is a $\mathbb{Z}_2$-graded algebra anti-automorphism,  i.e., $S(xy)=(-1)^{|x||y|}S(y)S(x)$.

Denote by $\mathrm{U}^{\geqslant0}$ (resp., $\mathrm{U}^{\leqslant0}$) the sub-superalgebra of $\mathrm{U}_q(\mathfrak{g})$ generated by all $\bbe_i,\bbk_i^{\pm1}$ (resp., $\bbf_i,\bbk_i^{\pm1}$),  set $\mathrm{U}^0$ equal to the sub-superalgebra of $\mathrm{U}_q(\mathfrak{g})$ generated by all $\bbk_i^{\pm1}$, and denote by $\mathrm{U}^+$ (resp., $\mathrm{U}^-$) the sub-superalgebra of $\mathrm{U}_q(\mathfrak{g})$ generated by all $\bbe_i$ (resp., $\bbf_i$), it is well-known that $\mathrm{U}^+\otimes \mathrm{U}^0\cong\mathrm{U}^{\geqslant0}$ (resp., $\mathrm{U}^-\otimes \mathrm{U}^0\cong\mathrm{U}^{\leqslant0}$) by the multiplication map. And the multiplication map $\mathrm{U}^-\otimes \mathrm{U}^0\otimes \mathrm{U}^+\rightarrow \mathrm{U}$ is an isomorphism as super vector spaces.
\begin{remark}\label{skew-primitive}
Analogous to the quantum group,  the quantum Serre relations and the higher order quantum Serre relations can be explained from the view of skew primitive elements in the quantum superalgebras.  For example,
\begin{align*}
&\Delta(u^+_{ij})=u^+_{ij}\otimes 1+\bbk_i^{1-a_{ij}}\bbk_j\otimes u^+_{ij},
&&\Delta(u^-_{ij})=u^-_{ij}\otimes\bbk_{i}^{a_{ij}-1}\bbk_j^{-1}+1\otimes u^-_{ij},\\
&\Delta(u^+_B)=u^+_B\otimes 1+\bbk_{m-1}\bbk_m^3\otimes u^+_B,
&&\Delta(u^-_B)=1\otimes u^-_B+u^-_B\otimes \bbk_{m-1}^{-1}\bbk_m^{-3},\\
&\Delta(u^+)=u^+\otimes 1+\bbk_{m-1}\bbk_m^2\bbk_{m+1}\otimes u^+,
&&\Delta(u^-)=1\otimes u^-+u^-\otimes \bbk_{m-1}^{-1}\bbk_m^{-2}\bbk_{m+1}^{-1}.
\end{align*}
where $u^{\pm}_{ij}$ (resp. $u_{B}^{\pm}$) is on the left side of equations (\ref{serreE1}) and (\ref{serreF1}) for $i\ne j$ and even $\alpha_i$ (resp., for non-isotropic odd root $\alpha_i$ with $a_{ij}\ne 0$ for $i\ne j$), and $u^{\pm}$ is on the left side of equations (\ref{serreE2})-(\ref{serreE4}).
\end{remark}
Define $\bbk_{\mu}=\mathop{\prod}\limits_{i=1}^r\bbk_i^{m_i}~~\text{if}~~ \mu=\mathop{\sum}\limits_{i=1}^rm_i\alpha_i\in \mathbb{Z}\Phi$. Thus, $\bbk_{\mu}\bbk_{\mu'}=\bbk_{\mu+\mu'}$ for all $\mu,\mu'\in \mathbb{Z}\Phi$. Therefore, $\{\bbk_{\mu}\}_{\mu\in \mathbb{Z}\Phi}$ spans $\mathrm{U}^0$ as a vector space,  and
\begin{align*}
  \bbk_{\mu}\bbe_i\bbk_{\mu}^{-1}=q^{(\mu,\alpha_i)}\bbe_i, \ \bbk_{\mu}\bbf_i\bbk_{\mu}^{-1}=q^{-(\mu,\alpha_i)}\bbf_i.
\end{align*}
The quantum superalgebra $\mathrm{U}_q(\mathfrak{g})$ is $\mathbb{Z}\Phi$-graded.  And the gradation is given by
\begin{align*}
\mathrm{deg}~\bbk_{\mu}=0, \ \mathrm{deg}~\bbe_i=\alpha_i, \ \mathrm{deg}~\bbf_i=-\alpha_i,
\end{align*}
for all $\mu\in \mathbb{Z}\Phi$ and $i\in\bbi$. We denote that $\mathrm{U}_{\nu}$ is the $\nu\in\mathbb{Z}\Phi$-component if $\mathfrak{g}\neq A(n,n)$. 

Note that if $\mathfrak{g}=A(n,n)$,  the  simple roots for distinguished Borel subalgebra are not linearly independent (that is, $\gamma=\mathop{\sum}\limits_{i=1}^{2n+1}d_iJ_i\alpha_i=0$). This causes some technical difficulties. However, the quantum superalgebra $\mathrm{U}_q(\mathfrak{g})$ is also $\mathbb{Z}\tilde{\Phi}$-graded, where $\mathbb{Z}\tilde{\Phi}$ is a free abelian group generated by all simple roots $\alpha_1,\alpha_2,\cdots,\alpha_{2n+1}$. Obviously, $\mathbb{Z}\Phi=\mathbb{Z}\tilde{\Phi}/\mathbb{Z}\gamma$. 

Define $\mathrm{U}|_{\mu}$ (resp. $\mathrm{U}_{\nu}$) as the $\mu$-component (resp. $\nu$-component) with respect to $\mathbb{Z}\Phi$-gradation (resp. $\mathbb{Z}\tilde{\Phi}$-gradation). From now on, we do not make a distinction between the elements in $\mathbb{Z}\Phi$ and $\mathbb{Z}\tilde{\Phi}$ if no confusion emerges. Hence,  $\mathrm{U}|_{\mu}=\mathop{\bigoplus}\limits_{k\in\mathbb{Z}}\mathrm{U}_{\mu+k\gamma}$. Set
$$Q=
\begin{cases}
\mathbb{Z}\tilde{\Phi}, &\text{ for } \mathfrak{g}= A(n,n),\\
\mathbb{Z}\Phi, & \text{ otherwise}.
\end{cases}
$$

Note that $\mathfrak{h}^*=\mathbb{C}\Phi$. If $\mathfrak{g}\neq A(n,n)$, define the standard partial order relation on $\mathfrak{h}^*$ by $\lambda\leqslant\mu \Leftrightarrow \mu-\lambda\in \sum_{i\in\mathbb{I}}\mathbb{Z}_+\alpha_i$. This breaks down if $\mathfrak{g}=A(n,n)$ because $\gamma=0$ and $d_iJ_i\in\mathbb{Z}_+$ for all $i\in\mathbb{I}$. However, we can define a similar partial order on $\mathbb{C}\tilde{\Phi}$. From now on, we will use the partial order on $\mathbb{C}\tilde{\Phi}$ if necessary for $\mathfrak{g}=A(n,n)$.
\begin{remark}

The Lie superalgebra $A(n,n)$ is rather special, and the restriction of the Harish-Chandra projection determined by the distinguish triangular decomposition to the zero weight space (with respect to $\mathbb{Z}\Phi$-gradation) is not an algebra homomorphism; for more details,  see \cite[Section 6.1.4]{Gorelik}. For this reason, we do not expect that the projection from $\mathrm{U}|_0$ to $\mathrm{U}^0$ is an algebra homomorphism. However, the projection $\pi\colon\mathrm{U}_0\to \mathrm{U}^0$ is an algebra homomorphism. Fortunately, we can prove that $\mathcal{Z}$ is contained in $\mathrm{U}_0$; see Corollary \ref{centerU_0}. Therefore, we can establish the Harish-Chandra homomorphism for $\mathfrak{g}=A(n,n)$. 
\end{remark}
\section{Representation of quantum superalgebras}\label{Representations}
\subsection{Representations}
We will recall some basic facts about the representation theory of the quantum superalgebra $\mathrm{U}_q(\mathfrak{g})$. The bilinear form $(\text{-},\text{-})$ on $\mathbb{Z}\Phi$ can be linearly extended to $\mathfrak{h}^*$. For any $\lambda,\mu\in \mathfrak{h}^*$, denote $\langle\lambda,\mu \rangle=\frac{2(\lambda,\mu)}{(\mu,\mu)}$. Let $\Lambda=\left\{\left. \lambda\in\mathfrak{h}^*\right |\langle \lambda,\alpha\rangle\in\mathbb{Z},~ \forall\alpha\in\Phi_{\bar{0}}\right\}$ be the integral weight lattice,
 and denote $\Lambda^+=\left\{\left.\lambda\in\mathfrak{h}^*\right|\langle \lambda,\alpha  \rangle\in\mathbb{Z}_+,~\forall\alpha\in\Phi_{\bar{0}}^+\right\}$ to be the set of dominant integral weights. In this paper, all modules are of type {\bf 1} with the weights contained in $\frac{1}{2}\mathbb{Z}\Phi$.

A $\mathrm{U}_q(\mathfrak{g})$-module $M$ is called \textit{a weight module} if it admits a weight space decomposition
\begin{align}
M=\mathop{\bigoplus}\limits_{\lambda\in\mathfrak{h}^*}M_{\lambda},\quad \text{where } M_{\lambda}=\left\{\left.u\in M\right | \bbk_iu=q^{(\lambda,\alpha_i)}u, ~\forall i\in\bbi\right\}.
\end{align}
Define by $\mathrm{wt}(M)$ the set of weights of the finite-dimensional $\mathrm{U}_q(\mathfrak{g})$-module $M$. A weight module $M$ is called a \textit{highest weight module} with the highest weight $\lambda$ if there exists a unique non-zero vector $v_{\lambda}\in M$, which is called a \textit{highest weight vector} such that $\bbk_iv_{\lambda}=q^{(\lambda,\alpha_i)},~ \bbe_iv_{\lambda}=0$ for all $i\in\bbi$ and $M=\mathrm{U}_q(\mathfrak{g})v_{\lambda}$.

Let $J_{\lambda}=\mathop{\sum}\limits_{i=1}^r\mathrm{U}_q(\mathfrak{g})\bbe_i+\mathop{\sum}\limits_{i=1}^r\mathrm{U}_q(\mathfrak{g})(\bbk_i-q^{(\lambda,\alpha_i)})$ for $\lambda\in\Lambda$,  and set $\Delta_q(\lambda)=\mathrm{U}_q(\mathfrak{g})/J_{\lambda}$. This is a $\mathrm{U}_q(\mathfrak{g})$-module generated by the coset of 1; also denote this coset by $v_{\lambda}$. Obviously, $\bbe_iv_{\lambda}=0$ and $\bbk_iv_{\lambda}=q^{(\lambda,\alpha_i)}v_{\lambda}~\text{for } i\in\bbi$.  We call $\Delta_q(\lambda)$ the \textit{Verma module} of the highest weight $\lambda$. It has the following universal property: If $M$ is an $\mathrm{U}_q(\mathfrak{g})$-module with the highest weight vector $v\in M_{\lambda}$, then there is a unique homomorphism of $\mathrm{U}_q(\mathfrak{g})$-modules $\varphi: \Delta_q(\lambda)\rightarrow M$ with $\varphi(v_{\lambda})=v$.
The Verma module $\Delta_q(\lambda)$ has a unique maximal submodule, thus,  $\Delta_q(\lambda)$ admits a unique simple quotient $\mathrm{U}_q(\mathfrak{g})$-module $L_q(\lambda)$.
\begin{lemma}\label{embeding}
Let $\lambda\in \Lambda$ with $(\lambda,\alpha_s)=0$. Then there is a homomorphism of $\mathrm{U}_q(\mathfrak{g})$-modules $\varphi\colon\Delta_q(\lambda-\alpha_s)\rightarrow\Delta_q(\lambda)$ with $\varphi(v_{\lambda-\alpha_s})=\bbf_sv_{\lambda}$.
\end{lemma}
\begin{proof}
We have $\bbf_sv_{\lambda}\in\Delta_q(\lambda)_{\lambda-\alpha_s}$. Therefore, the universal property of $\Delta_q(\lambda-\alpha_s)$ implies that it is enough to show that $\bbe_j\bbf_s v_{\lambda}=0$ for all $j\in\bbi$. This is obvious for $j\neq s$ because $\bbe_j$ and $\bbf_s$ commute. For $j=s$,  we have $\bbe_s\bbf_sv_{\lambda}=[\bbe_s,\bbf_s]v_{\lambda}-\bbf_s\bbe_sv_{\lambda}=\frac{\bbk_s-\bbk_s^{-1}}{q_s-q_s^{-1}}v_{\lambda}-0=0.$
\end{proof}

The finite-dimensional irreducible representations of $\mathrm{U}_q(\mathfrak{g})$ can be classified into two types: typical and atypical.  The representation theory of $\mathrm{U}_q(\mathfrak{g})$ at generic $q$ is rather similar to the Lie superalgebra $\mathfrak{g}$, as well. Geer proved the theorem that each irreducible highest weight module of a Lie superalgebra of Type A-G can be deformed to an irreducible highest weight module over the corresponding Drinfeld-Jimbo algebra; see \cite[Theorem 1.2]{Ge2}. We also refer to \cite[Proposition 3]{Zhang1}, \cite[Proposition 1]{Zhang2} and \cite[Theorem 4.2]{Kwon1} for quantum superalgebras of type $\mathrm{U}_q(\mathfrak{gl}_{m|n})$, $\mathrm{U}_q(\mathfrak{osp}_{2|2n})$ and $\mathrm{U}_q(\mathfrak{osp}_{m|2n})$, respectively.
\begin{theorem}\label{repquantum}
For $\lambda\in\mathfrak{h}^*$,  let $L(\lambda)$ be the irreducible highest weight module over $\mathfrak{g}$ of highest weight $\lambda$. Then there exists an irreducible highest weight module $L_q(\lambda)$ of highest weight $\lambda$ which is a deformation of $L(\lambda)$. Moreover, the classical limit of $L_q(\lambda)$ is $L(\lambda)$, and their (super)characters are equal \footnote{However, the inverse of the theorem is not true in general \cite{AYY}.  For example, there are many finite-dimensional irreducible modules (spinorial modules) of quantum superalgebras of type $\mathrm{U}_q(\mathfrak{osp}_{1|2})$ without classical limit; see \cite{Zhang} for more details.}.  
\end{theorem}
\subsection{Grothendieck ring}
Let $A$-{\bf mod} be the category of finite-dimensional modules of a superalgebra $A$. There is a parity reversing functor on this category. For an $A$-module $M=M_{\bar{0}}\oplus M_{\bar{1}}$, define
\begin{equation*}
\Pi(M)=\Pi(M)_{\bar{0}}\oplus \Pi(M)_{\bar{1}},\quad \Pi(M)_{i}=\Pi(M)_{i+\bar{1}},~\forall i\in\mathbb{Z}_2.
\end{equation*}
Then $\Pi(M)$ is also an $A$-module with the action $a m=(-1)^{|a|}m$. Let $\Pi$ be a 1-dimensional odd vector space with basis $\{\pi\}$, then $\Pi(M)\cong \Pi\otimes M$. Clearly, $\Pi^2=1$. Define the Grothendieck group $K(A)$ of $A$-{\bf mod} to be the abelian group generated by all objects in $A$-{\bf mod} subject to the following two relations: (i) $[M]=[L]+[N]$; (ii) $[\Pi(M)]=-[M]$, for all $A$-modules $L,M,N$ which satisfying a short exact sequence $0\to L\to M\to N\to 0$ with even morphisms.

It is easy to see that the Grothendieck group $K(A)$ is a free $\mathbb{Z}$-module
with the basis corresponding to the classes of the irreducible modules. Furthermore, if $A$ is a Hopf superalgebra, then for any $A$-modules $M$ and $N$, one can define the $A$-module structure on $M\otimes N$. Using this, we define the product on $K(A)$ by the formula
\begin{equation*}
[M][N]=[M\otimes N].
\end{equation*}

Since all modules are finite-dimensional, this multiplication is well-defined on
the Grothendieck group $K(A)$ and introduces the ring structure on it. The corresponding ring is called the \textit{Grothendieck ring} of $A$. The Grothendieck ring of $\mathrm{U}(\mathfrak{g})$ is denoted by $K(\mathfrak{g})$.
Let $K_{\mathrm{ev}}(\mathfrak{g})$ (resp. $K_{\mathrm{ev}}(\mathrm{U}_q(\mathfrak{g}))$) be the subring of $K(\mathfrak{g})$ (resp. $K(\mathrm{U}_q(\mathfrak{g}))$) generated by all objects in $\mathrm{U}(\mathfrak{g})$-{\bf mod} (resp. $\mathrm{U}_q(\mathfrak{g})$-{\bf mod}), whose weights are contained in $\Lambda\cap \frac{1}{2}\mathbb{Z}\Phi$. 

Let $M$ be a finite-dimensional representation of $\mathfrak{g}$ or $\mathrm{U}_q(\mathfrak{g})$. We define the \textit{character map} and the \textit{supercharacter map} as:
\begin{equation*}
\mathrm{ch}(M)=\mathop{\sum}\limits_{\lambda}\mathrm{dim}M_{\lambda}e^{\lambda},\quad\quad\mathrm{Sch}(M)=\mathop{\sum}\limits_{\lambda}\mathrm{sdim}M_{\lambda}e^{\lambda},
\end{equation*}
where $\mathrm{sdim}$ is the superdimension defined for any $\mathbb{Z}_2$-graded vector space $W=
W_0\oplus W_1$ as the difference of usual dimensions of graded components:
$\mathrm{sdim}W = \mathrm{dim}W_0-\mathrm{dim} W_1$.
\begin{proposition}\label{injgrothendieck}
There is an injective ring homomorphism $\jmath\colon K(\mathfrak{g})\to K(\mathrm{U}_q(\mathfrak{g}))$, which preserves (super)characters.
\end{proposition}
\begin{proof}
By Theorem \ref{repquantum}, we can define $\jmath([L(\lambda)])=[L_q(\lambda)]$ for all finite-dimensional irreducible $\mathfrak{g}$-modules $L(\lambda)$. This then induces an abelian group homomorphism from $K(\mathfrak{g})$ to  $K(\mathrm{U}_q(\mathfrak{g}))$. The map preserves (super)characters, so $\jmath$ is a ring homomorphism. Suppose there exist nonzero $a_i\in\mathbb{Z}$ and distinct $\lambda_i\in\mathfrak{h}^*$ for $i=1,2\cdots,n$ such that $\jmath(\mathop{\sum}\limits_{i=1}^na_i[L(\lambda_i)])=0$. Then $\mathrm{Sch}(\mathop{\sum}\limits_{i=1}^na_i[L(\lambda_i)])=0$. Choose $\lambda_j$ maximal in $\{\lambda_i\in\mathfrak{h}^*|i=1,2,\cdots,n\}$ for some $j$, then $a_j=0$ since $\mathrm{dim}(L(\lambda_i))_{\lambda_j}=\delta_{ij}$. This contradicts $a_j\neq0$. Thus, $\mathop{\sum}\limits_{i=1}^na_i[L(\lambda_i)]=0$.
\end{proof}

Sergeev and Veselov proved that the Grothendieck ring $K(\mathfrak{g})$ is isomorphic to the \textit{ring of exponential super-invariants} $J(\mathfrak{g})=\left\{\left. f\in\mathbb{Z}[P_0]^{W_0}\right|D_{\alpha}f\in(e^{\alpha}-1) \text{ for any isotropic root }\alpha \right\}$ for $\mathfrak{g}\neq A(1,1)$, where $D_{\alpha}(e^{\lambda})=(\lambda,\alpha)e^{\lambda}$, $\left\{\left. e^{\lambda}\right|\lambda\in P_0\right\}$ is a $\mathbb{Z}$-free basis of $\mathbb{Z}[P_0]$,  and here $P_0=\Lambda$ and $W_0=W$, more details could be found in \cite{SerVes}.

Define 
\begin{align}\label{J_q(g)}
J_{\mathrm{ev}}(\mathfrak{g})=\Bigg\{\mathop{\sum}\limits_{\mu\in 2\Lambda\cap \mathbb{Z}\Phi} a_{\mu}\bbk_{\mu}\in\mathrm{U}^0\Bigg|a_{w\mu}=a_{\mu}, ~\forall w\in W;~ D_{\alpha}(u)\in (\bbk_{\alpha}^2-1),~ \forall\alpha\in{\Phi}_{\mathrm{iso}}\Bigg\},
\end{align} 
where $D_{\alpha}(\bbk_{\mu})=(\mu,\alpha)\bbk_{\mu}$.

Obviously, there is an injective algebra homomorphism $\iota\colon J_{\mathrm{ev}}(\mathfrak{g})\rightarrow k\otimes_{\mathbb{Z}}J(\mathfrak{g})$ with $\iota(\bbk_{\mu})=e^{-\mu/2}$. This induces an isomorphism from $K_{\mathrm{ev}}(\mathfrak{g})$ to $ J_{\mathrm{ev}}(\mathfrak{g})$, hence we have the following commutative diagram:
\begin{equation*}
\xymatrix@C=50pt{
K(\mathrm{U}_q(\mathfrak{g}))&K(\mathfrak{g})\ar@{_{(}->}[l]_{\jmath}\ar[r]^{\cong}_{\mathrm{Sch}} &J(\mathfrak{g})  \\
 {K_{\mathrm{ev}}(\mathrm{U}_q(\mathfrak{g}))}\ar@{^{(}->}[u]&K_{\mathrm{ev}}(\mathfrak{g})\ar@{^{(}->}[u]\ar@{_{(}->}[l]\ar@{-->}[r]^{\cong} &J_{\mathrm{ev}}(\mathfrak{g})\ar@{^{(}-->}[u]_{\iota}
}
\end{equation*}
We remark that the above diagram is not true for $\mathfrak{g}=A(1,1)$. In Appendix \ref{appendixB}, we describe $J_{\mathrm{ev}}(\mathfrak{g})$ in sense of Sergeev and Veselov \cite{SerVes} and illustrate why $K_{\mathrm{ev}}(\mathfrak{g})\ncong J_{\mathrm{ev}}(\mathfrak{g})$ if $\mathfrak{g}=A(1,1)$.

\subsection{Some important propositions}
In this subsection, we investigate some important propositions, which show that the center of $\mathrm{U}_q(A(n,n))$ is contained in $\mathrm{U}^0$ and will be used to prove the injectivity of $\mathcal{HC}$. 

If $\mathfrak{g}$ is of type \RNum2, there exists a unique $\delta\in\Phi_{\bar{0}}^+$ such that $(\Pi\setminus\{\alpha_s\})\cup\{\delta\}$ is a simple root system of $\Phi_{\bar{0}}^+$. By writing $\delta=\mathop{\sum}\limits_{i=1}^rc_i\alpha_i$, we can get $c_s=2$. The following proposition is a super version of \cite[Section 3.2]{Tanisaki2} for quantum superalgebra $\mathrm{U}_q(\mathfrak{g})$ associated with a simple basic Lie superalgebra.
\begin{proposition}\label{3.3}
Set $\beta=\mathop{\sum}\limits_{i=1}^rm_i\alpha_i\in \mathbb{Z}_+\Pi$, and let $L_q(\lambda)$ be a typical finite-dimensional irreducible module. Suppose $\lambda$ satisfies\\
(i)  $\langle\lambda,\alpha_i\rangle\geqslant m_i$ for all $i\neq s$;\\
(ii) an  extra condition $2\langle\lambda+\rho,\delta\rangle \geqslant m_s+1$ when $\mathfrak{g}$ is of type \RNum2,\\
then $\mathrm{U}^-_{-\beta}\rightarrow L_q(\lambda)_{\lambda-\beta}$ with $u\mapsto uv_{\lambda}$ is bijective.
\end{proposition}
\begin{proof}
In the proof of this proposition, we choose $\lambda\in \mathbb{C}\otimes_{\mathbb{Z}} Q$ since the Verma module and simple module can be viewed as $Q$-graded modules. Notice that the partial order is well-defined on $Q$.

The canonical map from $\Delta_q(\lambda)$ to $L_q(\lambda)$ is surjective, which follows that every finite-dimensional irreducible module is a quotient of a Verma module. So we only need to prove $\mathrm{dim}\Delta_q(\lambda)_{\lambda-\beta}=\mathrm{dim}L_q(\lambda)_{\lambda-\beta}$, since $\mathrm{dim}\mathrm{U}^-_{-\beta}=\mathrm{dim}\Delta_q(\lambda)_{\lambda-\beta}$. The $\mathrm{dim}\Delta_q(\lambda)_{\lambda-\beta}$ is the coefficient of $e^{\lambda-\beta}$ in $\mathrm{ch} \Delta_q(\lambda)$, and $\mathrm{dim}L_q(\lambda)_{\lambda-\beta}$ is the coefficient of $e^{\lambda-\beta}$ in $\mathrm{ch} L_q(\lambda)$. 

The following character formulas of a Verma module and a typical finite-dimensional irreducible $\mathrm{U}_q(\mathfrak{g})$-module with the highest weight $\lambda$ are given by \cite[Theorem 1]{Kac2} and Theorem \ref{repquantum}:
\begin{align*}
\mathrm{ch} \Delta_q(\lambda)=\frac{\mathop{\Pi}_{\alpha\in\Phi^+_{\bar{1}}}(1+e^{-\alpha})}{\mathop{\Pi}_{\beta\in\Phi^+_{\bar{0}}}(1-e^{-\beta})}e^{\lambda},
\quad\quad
\mathrm{ch} L_q(\lambda)=\frac{\mathop{\Pi}_{\alpha\in\Phi^+_{\bar{1}}}(1+e^{-\alpha})}{\mathop{\Pi}_{\beta\in\Phi^+_{\bar{0}}}(1-e^{-\beta})}\mathop{\sum}\limits_{w\in W}(-1)^{l(w)}e^{w(\lambda+\rho)-\rho}.
\end{align*} 

Hence, it is sufficient to show $w(\lambda+\rho)-\rho-(\lambda-\beta)\notin \mathbb{Z}_+\Pi $ for all $w\neq 1$. Let us prove it by induction on $l(w)$.

If $\mathfrak{g}$ is of type \RNum1 and $l(w)=1$, then we have $w=s_i$ for some $i\neq s$, and hence
\begin{align*}
w(\lambda+\rho)-\rho-(\lambda-\beta)=-(\langle \lambda,\alpha_i\rangle+1)\alpha_i+\beta\notin \mathbb{Z}_+\Pi.
\end{align*}
Assume that $l(w)\geqslant 2$. There exist some $j\neq s$ and $w'\in W$ such that $w=s_jw'$ with $l(w')=l(w)-1$, and then it is known that $w'^{-1}(\alpha_j)\in\Phi^+_{\bar{0}}$. We have
\begin{align*}
w(\lambda+\rho)-\rho-(\lambda-\beta)=w'(\lambda+\rho)-\rho-(\lambda-\beta)-\langle\lambda+\rho,w'^{-1}(\alpha_j)\rangle \alpha_j,
\end{align*}
$w'(\lambda+\rho)-\rho-(\lambda-\beta)\notin \mathbb{Z}_+\Pi$ by induction and $\langle\lambda+\rho,w'^{-1}(\alpha_j)\rangle\geqslant0$ since $\lambda+\rho$ is dominant, so $w(\lambda+\rho)-\rho-(\lambda-\beta)\notin\mathbb{Z}_+\Pi$ for all $w\neq 1$.

If $\mathfrak{g}$ is of type \RNum2 and $l(w)=1$, then we have $w=s_i$ for some $i\neq s$ or $w=s_{\delta}$. By the same argument as above,  we only need to consider $w=s_{\delta}$.  Indeed,
\begin{align*}
w(\lambda+\rho)-\rho-(\lambda-\beta)=-\langle\lambda+\rho,\delta\rangle\delta+\beta\notin \mathbb{Z}_+\Pi.
\end{align*}
Assume $l(w)\geqslant2$. There exist some $j\neq s$ and $w'\in W$ such that $w=s_jw'$ or $w=s_{\delta}w'$ with $l(w')=l(w)-1$.  Then it is known that $w'^{-1}(\alpha_j)$ or $w'^{-1}(\delta)$ belongs to $\Phi_{\bar{0}}^+$. The proof is similar to type \RNum1 when $w=s_jw'$,  so we omit it here.
If $w=s_{\delta}w'$, then
\begin{align*}
w(\lambda+\rho)-\rho-(\lambda-\beta)=w'(\lambda+\rho)-\rho-(\lambda-\beta)-\langle\lambda+\rho,w'^{-1}(\delta)\rangle \delta.
\end{align*}
Once again, $w'(\lambda+\rho)-\rho-(\lambda-\beta)\notin \mathbb{Z}_+\Pi$ by induction and $\langle\lambda+\rho,w'^{-1}(\alpha_j)\rangle\geqslant0$ since $\lambda+\rho$ is dominant, so $w(\lambda+\rho)-\rho-(\lambda-\beta)\notin\mathbb{Z}_+\Pi$ for all $w\neq 1$.
\end{proof}
Let $\lambda\in\Lambda$ be a typical weight such that $L_q(\lambda)$ is finite-dimensional, then we can define a twisted action on $L_q(\lambda)$ via the automorphism $\omega$ of $\mathrm{U}_q(\mathfrak{g})$, denoted by $L^{\omega}_q(\lambda)$. Set $v_{\lambda}$ by $v'_{\lambda}$ when considered as an element of $L^{\omega}_q(\lambda)$. We then have $\bbk_{\mu}v'_{\lambda}=q^{-(\mu,\lambda)}v'_{\lambda}$ for all $\mu\in \mathbb{Z}\Phi$. Furthermore, we have $\bbf_iv'_{\lambda}=0$ for all $i\in\mathbb{I}$, and $x\mapsto xv'_{\lambda}$ maps each $\mathrm{U}^+_{\nu}$ onto $L_q^{\omega}(\lambda)_{-\lambda+\nu}$.

Similarly, if $\langle \lambda,\alpha_i\rangle\geqslant m_i,~\forall i\neq s$ and $\lambda$ satisfies an extra condition $2\langle\lambda+\rho,\delta\rangle\geqslant m_s+1$ for $\mathfrak{g}$ is of type \RNum2, then the map $\mathrm{U}_{\nu}^+\to L_q^{\omega}(\lambda)_{-\lambda+\nu}$ with $x\mapsto xv'_{\lambda}$ is bijective.
\begin{theorem}\label{annihilate}
Let $u\in\mathrm{U}$. If $u$ annihilates all finite-dimensional $\mathrm{U}$-modules, then $u=0$.
\end{theorem}
\begin{proof}
For any typical weights $\lambda,\lambda'\in\Lambda$ such that $L_q(\lambda)$ and $L^{\omega}_q(\lambda')$ are finite-dimensional, the tensor product $L_q(\lambda)\otimes L^{\omega}_q(\lambda')$ is also a finite-dimensional $\mathrm{U}_q(\mathfrak{g})$-module. Suppose that $u\in\mathrm{U}_q(\mathfrak{g})$ annihilates all these tensor products, in particular $u(v_{\lambda}\otimes v'_{\lambda'})=0$ for all $\lambda$ and $\lambda’$. We show that this implies $u=0$.

Choose bases $\{x_i\}$ of $\mathrm{U}^+$ and $\{y_j\}$ of $\mathrm{U}^-$ consisting of homogeneous weight vectors. This means $x_i\in\mathrm{U}^+_{\nu(i)}$ and $y_j\in\mathrm{U}^-_{-\nu'(j)}$ with $\mu(i)$ and $\nu'(j)$ in $\mathbb{Z}_+\Pi$. Write
\begin{align*}
u=\mathop{\sum}\limits_j\mathop{\sum}\limits_{\mu}\mathop{\sum}\limits_i a_{j,\mu,i}y_j\bbk_{\mu}x_i
\end{align*}
with $a_{j,\mu,i}\in k$, which is a finite sum. Suppose that $u\neq0$. Let $\nu_0\in\mathbb{Z}_+\Pi$ be maximal among the weights $\nu$ such that there exist $i,\mu,j$ with $a_{j,\mu,i}\neq0$ and $\nu=\nu(i)$.

So we have\begin{align*}
\bbk_{\mu}x_i(v_{\lambda}\otimes v'_{\lambda'})=q^{(\nu(i),\lambda)+(\mu,\lambda-\lambda'+\nu(i))}v_{\lambda}\otimes x_iv'_{\lambda'}.
\end{align*}
Each $\Delta(y_j)$ is equal to $y_j\otimes \bbk^{-1}_{\nu'(j)}$ plus a sum of terms in $\mathrm{U}^-\otimes \mathrm{U}^0\mathrm{U}^-_{<0}$. This implies that
\begin{align*}
y_j\bbk_{\mu}x_i(v_{\lambda}\otimes v'_{\lambda'})=q^{(\nu(i),\lambda)+(\mu,\lambda-\lambda'+\nu(i))-(\nu'(j),-\lambda'+\nu(i))}y_jv_{\lambda}\otimes x_iv'_{\lambda'}+(\ast),
\end{align*}
where $(\ast)$ is a sum of terms from a certain $L_q(\lambda)\otimes L^{\omega}_q(\lambda')_{-\lambda'+\nu}$ with $\nu\neq \nu(i)$.

The maximality of $\nu_0$ implies that $y_j\bbk_{\mu}x_i(v_{\lambda}\otimes v'_{\lambda'})$ has a component in $L_q(\lambda)\otimes L^{\omega}_q(\lambda')_{-\lambda'+\nu_0}$ only for $\nu(i)=\nu_0$. Therefore,  the projection of $u(v_{\lambda}\otimes v'_{\lambda'})$ onto $L_q(\lambda)\otimes  L^{\omega}_q(\lambda')_{-\lambda'+\nu_0}$ is equal to
\begin{align}
\mathop{\sum}\limits_{j,\mu,i;\nu(i)=\nu_0} a_{j,\mu,i}  q^{(\nu_0,\lambda)(\mu,\lambda-\lambda'+\nu_0)-(\nu'(j),-\lambda'+\nu_0)}y_jv_{\lambda}\otimes x_iv'_{\lambda'},\label{5.1}
\end{align}
since we assume that $u(v_{\lambda}\otimes v'_{\lambda'})=0$, this projection is also equal to $0$.

We can find an integer $N>0$ such that\begin{align*}
\nu_0<\mathop{\sum}\limits_{\alpha\in\Pi}N\alpha \quad\quad \text{and }\quad\quad \nu'(j)<\mathop{\sum}\limits_{\alpha\in \Pi} N\alpha 
\end{align*}
for all $j$. Set
\begin{align*}
\Lambda^+_N=\left\{ \lambda\in \Lambda\left|\begin{array}{c}
~\lambda ~\text{is typical,}~ L_q(\lambda) ~\text{is finite-dimensional,}~ \langle\lambda,\alpha_i\rangle>N ~\text{for all} ~i\neq s\\
\text{and plus an extra condition}~  2\langle\lambda+\rho,\delta\rangle>N+1~\text{if}~\mathfrak{g}~\text{is of type \RNum2}
\end{array}
\right.\right\}.
\end{align*}
By the same argument before the proposition, we know that the map $\mathrm{U}^+_{\nu_0}\rightarrow L^{\omega}_q(\lambda')_{\lambda'-\nu_0}, x\mapsto xv'_{\lambda'}$ is bijective for all $\lambda'\in\Lambda^+_N$.  Thus, the elements $x_iv'_{\lambda'}$ with $\nu(i)=\nu_0$ are linearly independent. Therefore,  the vanishing of the sum in (\ref{5.1}) implies (for all $\lambda'\in\Lambda^+_N$)\begin{align}
\mathop{\sum}\limits_{j,\mu}a_{j,\mu,i}q^{(\nu_0,\lambda)+(\mu,\lambda-\lambda'+\nu_0)-(\nu'(j),-\lambda'+\nu_0)}y_jv_{\lambda}=0,\label{5.2}
\end{align}
for all $i$ with $\nu(i)=\nu_0$.

The statement before this theorem implies that all $y_jv_{\lambda}$ with nonzero coefficients $a_{j,\mu,i}$ occuring in (\ref{5.2}) are linearly independent for all $\lambda\in\Lambda^+_N$. So we get from (\ref{5.2})
\begin{align}
\mathop{\sum}\limits_{\mu}a_{j,\mu,i}q^{(\nu_0,\lambda)+(\mu,\lambda-\lambda'+\nu_0)-(\nu'(j),-\lambda'+\nu_0)}=0, \label{5.3}
\end{align} 
for all $i,j$ with $\nu(i)=\nu_0$. We can cancel the (nonzero) factor $q^{(\nu_0,\lambda)-(\nu'(j),-\lambda'+\nu_0)}$ in (\ref{5.3}), which does not depend on $\mu$,  and get
\begin{align}
\mathop{\sum}\limits_{\mu}a_{j,\mu,i}q^{(\mu,\nu_0-\lambda')}q^{(\mu,\lambda)}=0,\label{5.4}
\end{align}
for all $i,j$ with $\nu(i)=\nu_0$ and all $\lambda,\lambda'\in\Lambda^+_N$. Now,  fix $\lambda'$ and notice that $(\text{-},\text{-})$ on $\mathbb{Z}\Phi\times \Lambda_N^+$ is non-degenerate in the first component for all $N$, thus the coefficients $a_{j,\mu,i}q^{(\mu,\nu_0-\lambda')}$
in (\ref{5.4}) are all equal to $0$. This implies that $a_{j,\mu,i}=0$ for all $i,j,\mu$ with $\nu(i)=\nu_0$, contradicting the choice of $\nu_0$. Therefore, $u=0$.
\end{proof}
One can check Proposition \ref{3.3} and Theorem \ref{annihilate} hold if $\mathfrak{g}=A(n,n)$ since $\mathbb{Z}\tilde{\Phi}$ has a partial order. Next, we strengthen Theorem \ref{annihilate} for $\mathfrak{g}=A(n,n)$.
\begin{theorem}\label{annihilate A(n,n)}
Let $u\in \mathrm{U}_q(A(n,n))$. If $u$ annihilates all typical finite-dimensional irreducible $\mathrm{U}_q(A(n,n))$-modules, then $u=0$.
\end{theorem}
\begin{proof}
It is known that if a typical irreducible module $L_q(\lambda)$ is a composition factor of a finite-dimensional module $M$, then $L_q(\lambda)$ is a direct summand of $M$. By the proof of Theorem \ref{annihilate}, we only need to prove the following claim.
 
For all $N>n$, there exists $\lambda\in \Lambda_N^+$ such that the set
\begin{equation*}
\left\{\left.\lambda'\in\Lambda_N^+ \right | L_q(\lambda)\otimes L^{\omega}_q(\lambda')\text{ is completely reducible} \right\}
\end{equation*}
could linearly span $\mathfrak{h}^*$.

If it is true, then $L_q(\lambda)\otimes L^{\omega}_q(\lambda')$ is completely reducible if all weights in $\lambda+\mathrm{wt}\left(L^{\omega}_q(\lambda')\right)$ are typical. Because the composition factors of $L_q(\lambda)\otimes L^{\omega}_q(\lambda')$ are the form of $L_q(\bar{\lambda})$ with $\bar{\lambda}\in \lambda+\mathrm{wt}\left(L^{\omega}_q(\lambda')\right)$ \cite[Corollary 5.2]{Serganova}.\\
%
\textbf{\textit{Proof of the claim:}} Let $\tilde{\lambda}=\mathop{\sum}\limits_{i=1}^{n+1}\big((n+1-i)(N+2)+2\big)\varepsilon_i-\mathop{\sum}\limits_{j=1}^n(j-1)(N+2)\delta_j-(nN+4n+2)\delta_{n+1}\in \Lambda_{N+1}^+$. Then $\tilde{\lambda}+\alpha_i\in\Lambda_N^+$ for all $i\in\mathbb{I}$. There exists a positive integer $\kappa$ such that it is bigger than  $\pm(\mu,\varepsilon_j)$ and $\pm(\mu,\delta_k)$
for any $\mu\in\mathrm{wt}\big(L^{\omega}_q(\tilde{\lambda}+\alpha_i)\big)$ with $i\in\mathbb{I},~j,k=1,2,\cdots,n+1$. Let $a=8\kappa$ and $\lambda=\mathop{\sum}\limits_{i=1}^{n+1}(n+\frac{5}{2}-i)a\varepsilon_i-\mathop{\sum}\limits_{j=1}^nja\delta_j-\frac{3(n+1)}{2}a\delta_{n+1}\in \Lambda$. Then $\lambda\in\Lambda_N^+$ and $\lambda+\mu$ are typical weights for all $\mu\in\mathrm{wt}\big(L^{\omega}_q(\tilde{\lambda}+\alpha_i)\big)$ with $i\in\mathbb{I}$. So $L_q(\lambda)\otimes L^{\omega}_q(\tilde{\lambda}+\alpha_i)$ are completely reducible for all $i\in\mathbb{I}$. Since $\{\tilde{\lambda}+\alpha_i|i\in\mathbb{I}\}$ could linearly span $\mathfrak{h}^*$, the claim holds.
\end{proof}
\begin{corollary}\label{centerU_0}
The Center $\mathcal{Z}(\mathrm{U}_q(\mathfrak{g}))$ is contained in $\mathrm{U}_0$.
\end{corollary}
\begin{proof}
If $\mathfrak{g}\neq A(n,n)$, note that $\mathcal{Z}(\mathrm{U}_q(\mathfrak{g}))$ is $\mathbb{Z}\Phi$-graded since $\mathrm{U}_q(\mathfrak{g})$ is $\mathbb{Z}\Phi$-graded. 
Assuming that  $\mathcal{Z}(\mathrm{U}_q(\mathfrak{g})) \cap \mathrm{U}_q(\mathfrak{g})_\nu\ne 0$ for some $\nu\in\mathbb{Z}\Phi$, we will show that $\nu=0$. Pick $0\ne z\in\mathcal{Z}(\mathrm{U}_q(\mathfrak{g})) \cap \mathrm{U}_q(\mathfrak{g})_\nu$. Then $z=\bbk_iz\bbk_i^{-1}=q^{(\nu,\alpha_i)}z$ for all $i\in\bbi$; hence $(\nu,\alpha_i)=0$ for all $i\in\bbi$, and $\nu=0$ since $(\text{-},\text{-})$ is non-degenerate.

For $\mathfrak{g}=A(n,n)$, the quantum superalgebra $\mathrm{U}_q(\mathfrak{g})$ is $\mathbb{Z}\tilde{\Phi}$-graded.  Similar to the argument above, if $\mathcal{Z}(\mathrm{U}_q(\mathfrak{g})) \cap \mathrm{U}_q(\mathfrak{g})_{\nu}\neq0$ with $\nu\in\mathbb{Z}\tilde{\Phi}$, then $\nu$ is contained in the radical of $(\text{-},\text{-})$. Thus, $\nu=k\gamma$ for some $k\in\mathbb{Z}$. We need to prove $k=0$. Otherwise assume $k\neq0$. Let $M$ be an arbitrary finite-dimensional irreducible module with the highest weight $\lambda$ and highest weight vector $v_{\lambda}$ and lowest weight $\lambda'$ and lowest weight vector $v_{\lambda'}$. Then $zv_{\lambda}\in M_{\lambda+k\gamma}=0$ if $k>0$ since $k\gamma>0$. Furthermore, $zv_{\lambda'}\in M_{\lambda'+k\gamma}=0$ if $k<0$ since $k\gamma<0$. Thus $zM=0$ and hence $z=0$ by Theorem \ref{annihilate A(n,n)}, which contradicts the choice of $z$.
\end{proof}

\section{Drinfeld double and ad-invariant bilinear form}\label{Drinfelddouble}
\subsection{The Drinfeld double}
In order to establish the Harish-Chandra homomorphism for quantum superalgebras,  we need to construct the quantum Killing form or Rosso form for quantum superalgebras.  Our approach to obtaining this takes advantage of the  Drinfeld double for $\mathbb{Z}_2$-graded Hopf algebras \cite{GZB}.  

\begin{definition}
A bilinear mapping $(\:,\:)\colon\mathcal{B}\times \mathcal{A}\mapsto k$ is called a \textit{skew-pairing} of the $\mathbb{Z}_2$-graded Hopf algebras $\mathcal{A}$ and $\mathcal{B}$ over $k$ if for all $a, a'\in\mathcal{A}$ and $b, b'\in\mathcal{B}$ we have
\begin{align}\label{skew-pairing}
(b,1)=\varepsilon(b)&,\quad  (1,a)=\varepsilon(a),\nonumber\\
(bb',a)=(-1)^{|b'||a_{(1)}|}\sum(b,a_{(1)})(b',a_{(2)})&,\quad (b,aa')=\sum(b_{(1)},a')(b_{(2)},a).
\end{align}
\end{definition}

\begin{proposition}
(\cite[Proposition 4]{GZB}) Let $\mathcal{A}$ and $\mathcal{B}$ be $\mathbb{Z}_2$-graded Hopf algebras equipped with a skew-pairing $(\:,\:)\colon\mathcal{B}\times \mathcal{A}\mapsto k$.  Then the vector space $\mathcal{A}\otimes\mathcal{B}$ becomes a superalgebra with multiplication defined by
\begin{equation}
(a\otimes b)(a'\otimes b')=\sum (-1)^{(|a'_{(1)}|+|a'_{(2)}|)(|b_{(2)}|+|b_{(3)}|)}(S(b_{(1)}),a'_{(1)}) (b_{(3)},a'_{(3)})aa'_{(2)}\otimes b_{(2)}b',
\end{equation}
for $a, a'\in\mathcal{A}$ and $b, b'\in\mathcal{B}$.  With the tensor product co-algebra and antipode $S(a\otimes b)=(-1)^{|a||b|}(1\otimes S(b))(S(a)\otimes 1)$ structure of $\mathcal{A}\otimes \mathcal{B}$, this superalgebra is also a $\mathbb{Z}_2$-graded Hopf algebra, called the Drinfeld double of $\mathcal{A}$ and $\mathcal{B}$ and denoted it by $\mathcal{D}(\mathcal{A},\mathcal{B})$. 
\end{proposition}
The existence of a dual pairing of $\mathrm{U}^{\geqslant0}$ and $(\mathrm{U}^{\leqslant0})^{\mathrm{op}}$ was observed by Drinfeld \cite{Drinfeld}. In our exposition, we followed Tanisaki \cite[Proposition 2.1.1]{Tanisaki1} for quantum groups and  Lehrer, Zhang, Zhang \cite[Section 3]{LZZ} for quantum superalgebra $\mathrm{U}_q(\mathfrak{gl}_{m|n})$. We have the following proposition.

\begin{proposition}
There is a unique non-degenerate skew-pairing between the $\mathbb{Z}_2$-graded Hopf algebras $\mathrm{U}^{\geqslant0}$ and $\mathrm{U}^{\leqslant0}$ with
\begin{align}\label{double}
(\bbk_i,\bbk_j)=q^{-(\alpha_i,\alpha_j)},\quad(\bbf_i,\bbe_j)=-\delta_{ij}\frac{1}{q_i-q_i^{-1}} \text{ and }
(\bbk_i,\bbe_j)=0,\quad      (\bbf_i,\bbk_j)=0.
\end{align} 
\end{proposition}
\begin{proof}
The well-defineness follows from \cite{Ge1} or Remark \ref{skew-primitive}, and the non-degeneracy of skew-pairing can be obtained from the following: for $\mu\in\mathbb{Z}\Phi$ with $\mu>0$ and $u\in\mathrm{U}^-_{-\mu}$ with $[\bbe_i,  u]=0$ for all $i\in\bbi$,  then $u=0$.  Similarly, if $u\in\mathrm{U}^+_{\mu}$ with $[\bbf_i, u]=0$ for all $i\in\bbi$,  then $u=0$. The fact can be proven in a similar way to Lemma \ref{injective}, which we omit here.
\end{proof}
\begin{remark}
Geer \cite{Ge1} extend Lusztig's \cite{Lu} results to the  Etingof-Kazhdan quantization of Lie superalgebras $\mathrm{U}_h^{DJ}(\mathfrak{g})$ and check directly that the extra quantum Serre-type relations are in the radical of the bilinear form.  Indeed,  the radical of the bilinear form is generated by the extra quantum Serre-type relations and higher order Serre relations.
\end{remark}

\begin{corollary}
As a superalgebra, 
$\mathcal{D}(\mathrm{U}^{\geqslant0},\mathrm{U}^{\leqslant0})$ is generated by elements $\bbe_i,\bbk_i,\bbk_i^{-1}, \bbf_i,  \bbk^{\prime}_i$, $\bbk_i^{\prime-1}$. The defining relations are the relations for the generators $\bbe_i,\bbk_i, \bbk_i^{-1},$  (resp. ,  $\bbf_i, \bbk^{\prime}_i, \bbk_i^{\prime-1}$) of the superalgebra $\mathrm{U}^{\geqslant0}$ (resp. $\mathrm{U}^{\leqslant0}$), and the following cross relations:\begin{align}
\bbk^{\prime}_i\bbe_j\bbk_i^{\prime-1}=q^{(\alpha_i,\alpha_j)}\bbe_j,&\quad \bbk_i\bbf_j\bbk_i^{-1}=q^{-(\alpha_i,\alpha_j)}\bbf_j,\\
\bbk_i\bbk^{\prime}_j=\bbk^{\prime}_j\bbk_i,\quad \bbe_i\bbf_j&-(-1)^{|\bbe_i||\bbf_j|}\bbf_j\bbe_i=\delta_{ij}\frac{\bbk_i-\bbk_i^{\prime-1}}{q_i-q_i^{-1}}.
\end{align}
\end{corollary}
It is known \cite{Ge1, GZB} that the sub-superalgebras $\mathrm{U}^{\geqslant0}$ and $\mathrm{U}^{\leqslant0}$ of the quantum superalgebras $\mathrm{U}_q(\mathfrak{g})$ form a skew-pairing,  and $\mathrm{U}_q(\mathfrak{g})$ is a quotient of quantum double of $\mathcal{D}(\mathrm{U}^{\geqslant0},\mathrm{U}^{\leqslant0})$.  More precisely,  we set $\mathcal{I}$ to be the two-sided ideal generated by the elements $\bbk_i-\bbk_i^{\prime-1}$,  which is also a $\mathbb{Z}_2$-graded Hopf ideal, and we have canonical isomorphism $\mathcal{D}(\mathrm{U}^{\geqslant 0},\mathrm{U}^{\leqslant0})/\mathcal{I}\cong \mathrm{U}_q(\mathfrak{g})$ as $\mathbb{Z}_2$-graded Hopf algebras. 
 Recently,  Drinfeld doubles have been studied by various authors as a useful tool to recover the quantum groups (see, e.g., \cite{BGH,FX,GHZ,HPR,HRZ,HZ}). 

\subsection{Rosso form}
Now we can define an ad-invariant and non-degenerate bilinear form on quantum superalgebras by using skew-pairing between $\mathrm{U}^{\geqslant 0}$ and $\mathrm{U}^{\leqslant0}$.  
\begin{theorem}\label{adinvariant}
Define a bilinear form $\langle~,~\rangle\colon\mathrm{U}_q(\mathfrak{g})\times\mathrm{U}_q(\mathfrak{g})\rightarrow k$ with\begin{equation}\label{comp}
\langle (y\bbk_{\nu})\bbk_{\lambda}x ,(y'\bbk_{\nu'})\bbk_{\lambda'}x'\rangle=(-1)^{|y|}(y',x)(y,x')q^{(2\rho,\nu)q^{-(\lambda,\lambda')/2}},
\end{equation}
for $x\in\mathrm{U}^+_{\mu},x'\in\mathrm{U}^+_{\mu'}$,  $y\in\mathrm{U}^-_{-\nu},y'\in\mathrm{U^-_{-\nu'}}, \lambda,\lambda'\in \mathbb{Z}\Phi$ and $\mu,\mu',\nu,\nu'\in Q$. The bilinear form is ad-invariant, i.e., $\langle \mathrm{ad}(u)v,v'\rangle=(-1)^{|u||v|}\langle v,\mathrm{ad}(S(u))v'\rangle$.  
\end{theorem}
By the use of the duality pairing, Tanisaki \cite{Tanisaki1} describes the Killing form of the quantum algebra, which is first constructed by Rosso \cite{Rosso}, then uses it to investigate the center of quantum algebra. Similar techniques could be applied in the case when $\mathfrak{g}$ is a Lie superalgebra of type A-G. Perhaps the proof of this theorem is known by several specialists, but it seems difficult to find in the existing literature. It is fundamental to prove the surjectivity of Harish-Chandra throughout this paper, so we write down the details to make the paper more accessible. Here we need some tedious computations,  which are also essential for Section \ref{Centers}.  

For $x\in \mathrm{U}^+_{\mu}$  and $y\in \mathrm{U}^-_{-\mu}$, we know $\Delta(x)\in\mathop{\bigoplus}\limits_{0\leqslant\nu\leqslant\mu}\mathrm{U}^+_{\mu-\nu}\bbk_{\nu}\otimes \mathrm{U}^+_{\nu}$ and $\Delta(y)\in\mathop{\bigoplus}\limits_{0\leqslant\nu\leqslant\mu}\mathrm{U}^-_{-\nu}\otimes \mathrm{U}^-_{-(\mu-\nu)}\bbk_{\nu}^{-1}$,  thus for each $\alpha_i\in \Pi$, we can define elements $r_i(x),  r'_i(x)$  in $\mathrm{U}^+_{\mu-\alpha}$ and  $r_i(y),  r'_i(y)$ in $\mathrm{U}^-_{-(\mu-\alpha)}$ to satisfy the following equations:
\begin{eqnarray*}
\begin{array}{llll}
&\Delta(x)=x\otimes 1+\mathop{\sum}\limits_{i=1}^rr_i(x)\bbk_i\otimes \bbe_i+\cdots=\bbk_{\mu}\otimes x+\mathop{\sum}\limits_{i=1}^r\bbe_i\bbk_{\mu-\alpha_i}\otimes r'_i(x)+\cdots, \text{ and }\\
&\Delta(y)=y\otimes \bbk_{\mu}^{-1}+\mathop{\sum}\limits_{i=1}^rr_i(y)\otimes \bbf_i\bbk_{\mu-\alpha_i}^{-1}+\cdots
=1\otimes y+\mathop{\sum}\limits_{i=1}^r\bbf_i\otimes r'_i(y) \bbk_{\alpha_i}^{-1}+\cdots.
\end{array}
\end{eqnarray*}
Then for all $x\in\mathrm{U}^+_{\mu},x'\in\mathrm{U}^+_{\mu'}$ and $y\in\mathrm{U}^-$,  we have
\begin{eqnarray*}
\begin{array}{llll}
&r_i(xx')=xr_i(x')+(-1)^{|\bbe_i||x'|}q^{(\mu',\alpha_i)}r_i(x)x',\quad &r'_i(xx')=(-1)^{|x||\bbe_i|}q^{(\mu,\alpha_i)}xr_i'(x')+r_i'(x)x',\\
&(\bbf_iy,x)=(-1)^{|r'_i(x)||\bbe_i|}(\bbf_i,\bbe_i)(y,r_i'(x)),\quad &(y\bbf_i,x)=(-1)^{|\bbf_i||r_i(x)|}(\bbf_i,\bbe_i)(y,r_i(x)).
\end{array}
\end{eqnarray*}
Similarly,  for all $y\in\mathrm{U}^-_{-\mu},y'\in\mathrm{U}^-_{-\mu'}$ and $x\in\mathrm{U}^+$, we have 
\begin{eqnarray*}
\begin{array}{llll}
&r_i(yy')=q^{(\mu,\alpha_i)}yr_i(y')+(-1)^{|\bbf_i||y'|}r_i(y)y',\quad &r'_i(yy')=(-1)^{|y||\bbf_i|}yr_i'(y')+q^{(\mu',\alpha_i)}r_i'(y)y',\\
&(y,\bbe_ix)=(\bbf_i,\bbe_i)(r_i(y),x),\quad &(y,x\bbe_i)=(\bbf_i,\bbe_i)(r'_i(y),x).
\end{array}
\end{eqnarray*}
Thus, we have the following lemma.
\begin{lemma} 
For all $x\in\mathrm{U}^+_{\mu}$ and $y\in\mathrm{U}^-_{-\mu}$, we have
\begin{align}
&[x,\bbf_i]=x\bbf_i-(-1)^{|x||\bbf_i|}\bbf_ix=(q_i-q_i^{-1})^{-1}\big(r_i(x)\bbk_i-(-1)^{|r'_i(x)||\bbf_i|}\bbk_i^{-1}r_i'(x)\big),\label{xfrelation}\\
&[\bbe_i,y]=\bbe_iy-(-1)^{|y||\bbe_i|}y\bbe_i=(q_i-q_i^{-1})^{-1}\big((-1)^{|\bbe_i||r_i(y)|}\bbk_ir_i(y)-r_i'(y)\bbk_i^{-1}\big).\label{eyrelation}
\end{align}
\end{lemma}
\begin{proof}
We only prove Equation (\ref{eyrelation}), and Equation (\ref{xfrelation}) is similar.
For $y=1$ and $y=\bbf_i$ the formula follows from definition, so it is enough to show that if Equation (\ref{eyrelation}) holds for $y\in \mathrm{U}^-_{-\mu}$ and $y'\in\mathrm{U}^-_{-\mu'}$, then Equation (\ref{eyrelation}) holds for $yy'$.  This can be derived as follows.
\begin{align*}
&(q_i-q_i^{-1})[\bbe_i,yy']=(q_i-q_i^{-1})\big([\bbe_i,y]y'+(-1)^{|\bbe_i||y|}y[\bbe_i,y']\big)\\
=&(-1)^{|\bbe_i||r_i(y)|}\big(\bbk_ir_i(y)-r_i'(y)\bbk_i^{-1}\big)y'+(-1)^{|\bbe_i||y|}y\big((-1)^{|\bbe_i||r_i(y')|}\bbk_ir_i(y')-r_i'(y')\big)\bbk_i^{-1}\\=&(-1)^{|\bbe_i||r_i(yy')|}\bbk_i\big((-1)^{|\bbe_i||y'|}r_i(y)y'+q^{(\mu,\alpha_i)}yr_i(y')\big)-\big(q^{(\mu',\alpha_i)}r'_i(y)y'+(-1)^{|\bbe_i||y|}yr'_i(y')\big)\bbk_i^{-1}\\=&(-1)^{|\bbe_i||r_i(yy')|}\bbk_ir_i(yy')-r_i'(yy')\bbk_i^{-1}.
\end{align*}
\end{proof}
Combining the above lemma, we get the following equations, which  are very useful when proofing Theorem \ref{adinvariant}.
\begin{align*}
\mathrm{ad}(\bbe_i)(y\bbk_{\lambda}x)=&\bbe_iy\bbk_{\lambda}x-(-1)^{|\bbe_i|(|x|+|y|)}\bbk_iy\bbk_{\lambda}x\bbk_i^{-1}\bbe_i\\
=&[\bbe_i,y]\bbk_{\lambda}x+(-1)^{|y||\bbe_i|}y\bbe_i\bbk_{\lambda}x-(-1)^{|\bbe_i|(|x|+|y|)}\bbk_iy\bbk_{\lambda}x\bbk_i^{-1}\bbe_i\\
=&(q_i-q_i^{-1})^{-1}\big((-1)^{|\bbe_i||r_i(y)|}\bbk_ir_i(y)\bbk_{\lambda}x-r_i'(y)\bbk_i^{-1}\bbk_{\lambda}x\big)
\\&+(-1)^{|y||\bbe_i|}q^{(\lambda,-\alpha_i)}y\bbk_{\lambda}\bbe_ix-(-1)^{|\bbe_i|(|x|+|y|)}q^{(\mu-\nu,\alpha_i)}y\bbk_{\lambda}x\bbe_i\\
=&(q_i-q_i^{-1})^{-1}\big((-1)^{|\bbe_i||r_i(y)|}q^{(\nu-\alpha_i,-\alpha_i)}r_i(y)\bbk_{\lambda+\alpha_i}x-r_i'(y)\bbk_{\lambda-\alpha_i}x\big)
\\&+(-1)^{|y||\bbe_i|}q^{(\lambda,-\alpha_i)}y\bbk_{\lambda}\bbe_ix-(-1)^{|\bbe_i|(|x|+|y|)}q^{(\mu-\nu,\alpha_i)}y\bbk_{\lambda}x\bbe_i.
\end{align*}
Now,  we are ready to prove Theorem \ref{adinvariant}.

\textit{Proof of Theorem \ref{adinvariant}:}
It is enough to take $u$ to be generators, i.e., $\bbe_i,\bbf_i$ and $\bbk_i$. Furthermore, we may assume that\begin{equation*}
v=(y\bbk_{\nu})\bbk_{\lambda}x\quad \text{and}\quad v'=(y'\bbk_{\nu'})\bbk_{\lambda'}x',
\end{equation*}
with $\lambda,\lambda'\in\mathbb{Z}\Phi$ and $x\in\mathrm{U}^+_{\mu},x'\in\mathrm{U}^+_{\mu'},y\in\mathrm{U}^-_{-\nu},y'\in\mathrm{U}^-_{-\nu'}$ with weights $\mu,\mu',\nu,\nu'\in Q$.

It is obvious for $u=\bbk_i$.  
For $u=\bbe_i$, then \begin{align*}
&\mathrm{ad}(\bbe_i)(v)
=(q_i-q_i^{-1})^{-1}\big((-1)^{|\bbe_i||r_i(y)|}q^{(\nu-\alpha_i,-\alpha_i)}r_i(y)\bbk_{\lambda+\nu+\alpha_i}x-r_i'(y)\bbk_{\lambda+\nu-\alpha_i}x\big)
\\&+(-1)^{|y||\bbe_i|}q^{(\lambda+\nu,-\alpha_i)}y\bbk_{\lambda+\nu}\bbe_ix-(-1)^{|\bbe_i|(|x|+|y|)}q^{(\mu-\nu,\alpha_i)}y\bbk_{\lambda+\nu}x\bbe_i,\text{ and }\\
&\mathrm{ad}(S(\bbe_i))(v')
= -\mathrm{ad}(\bbk_i^{-1})\mathrm{ad}(\bbe_i)(v')=-q^{(\mu'+\alpha_i-\nu',-\alpha_i)}\mathrm{ad}(\bbe_i)(v')             \\
=&(q_i-q_i^{-1})^{-1}\big(-(-1)^{|\bbe_i||r_i(y')|}q^{(\mu',-\alpha_i)}r_i(y')\bbk_{\lambda'+\nu'+\alpha_i}x'+q^{(\mu'+\alpha_i-\nu',-\alpha_i)}r_i'(y')\bbk_{\lambda'+\nu'-\alpha_i}x'\big)
\\&-(-1)^{|y'||\bbe_i|}q^{(\lambda'+\mu'+\alpha_i,-\alpha_i)}y'\bbk_{\lambda'+\nu'}\bbe_ix'+(-1)^{|\bbe_i|(|x'|+|y'|)}q^{(\alpha_i,-\alpha_i)}y'\bbk_{\lambda'+\nu'}x'\bbe_i.
\end{align*}
Now the problem can be split into two cases.  First, if $\mu=\nu'$ and $\mu'+\alpha_i=\nu$, then
\begin{align*}
\langle \mathrm{ad}(\bbe_i&)v,v'\rangle=(-1)^{|r_i(y)|}(q_i-q_i^{-1})^{-1}(y',x)q^{(2\rho,v-\alpha_i)},\\ 
&\cdot\big( (-1)^{|\bbe_i||r_i(y)|}q^{(v-\alpha_i,-\alpha_i)-1/2(\lambda+2\alpha_i,\lambda')}(r_i(y),x')-   q^{-1/2(\lambda,\lambda')}(r'_i(y),x')  \big),
\end{align*}
and
\begin{align*}
\langle v, \mathrm{ad}(S(\bbe_i))v'\rangle
=&(-1)^{|y|}(y',x)q^{(2\rho,\nu)}\big(-(-1)^{|y'||\bbe_i|}q^{(\lambda'+\mu'+\alpha_i,-\alpha_i)
-1/2(\lambda,\lambda')}(y,\bbe_ix')\\&+(-1)^{|\bbe_i|(|x'|+|y'|)}q^{(\alpha_i,-\alpha_i)-1/2(\lambda,\lambda')}(y,x'\bbe_i)\big).
\end{align*}
Therefore, $\langle \mathrm{ad}(\bbe_i)v,v'\rangle=(-1)^{|\bbe_i||v|}\langle v,\mathrm{ad}(S(\bbe_i))v'\rangle $.

Second,  if $\mu+\alpha_i=\nu'$ and $\mu'=\nu$, then\begin{align*}
\langle \mathrm{ad}(\bbe_i)v, v'\rangle=&(-1)^{|y|}q^{(2\rho,\nu)}(y,x')\cdot\big(   (-1)^{|y||\bbe_i|}q^{(\lambda+\nu,-\alpha_i)-1/2(\lambda,\lambda')}\\
&\cdot(y',\bbe_ix)-(-1)^{|\bbe_i|(|x|+|y|)} q^{(\mu-\nu,\alpha_i)-1/2(\lambda,\lambda')}(y',x\bbe_i)\big),
\end{align*}
and \begin{align*}
\langle v,\mathrm{ad}(S(\bbe_i))v'\rangle=&(-1)^{|y|}(q_i-q_i^{-1})^{-1}q^{(2\rho,\nu)}(y,x')
\cdot\big(-(-1)^{|\bbe_i||r_i(y')|}q^{(\mu',-\alpha_i)-1/2(\lambda,\lambda'+2\alpha_i)}\\
&\cdot(r_i(y'),x)+q^{(\mu'+\alpha_i-\nu',-\alpha_i)-1/2(\lambda,\lambda')}(r'_i(y'),x) \big).
\end{align*}
Therefore, $\langle \mathrm{ad}(\bbe_i)v,v'\rangle=(-1)^{|\bbe_i||v|}\langle v,\mathrm{ad}(S(\bbe_i))v'\rangle $. Using a similar procedure, we can check for $u=\bbf_i$. Thus, we prove the ad-invariance of the bilinear form.

\begin{proposition}\label{propnondeg}
Let $u\in \mathrm{U}_q(\mathfrak{g})$. If $\langle v,u\rangle=0$ for all $v\in\mathrm{U}_q(\mathfrak{g})$, then $u=0$.
\end{proposition}
\begin{proof}
Notice that $\mathrm{U}_q(\mathfrak{g})$ is the direct sum of all $\mathrm{U}^-_{-\nu}\mathrm{U}^0\mathrm{U}^+_{\mu}=\mathrm{U}^-_{-\nu}\bbk_{\nu}\mathrm{U}^0\mathrm{U}^+_{\mu}$ as vector space. Therefore, it is sufficient to show that if $u\in\mathrm{U}^-_{-\nu}\mathrm{U}^0\mathrm{U}^+_{\mu}$ with $\langle v,u\rangle=0$ for all $v\in\mathrm{U}^-_{-\mu}\mathrm{U}^0\mathrm{U}^+_{\nu}$, then $u=0$.\\
Since the skew-pairing between $\mathrm{U}^-$ and $\mathrm{U}^+$ is non-degenerate,  we can choose an arbitrary basis $u^{\mu}_1,u^{\mu}_2,\cdots,u^{\mu}_{r(\mu)}$ of $\mathrm{U}^+_{\mu}$ and dual basis  $v^{\mu}_1,v^{\mu}_2,\cdots,c^{\mu}_{r(\mu)}$ of $\mathrm{U}^-_{-\mu}$ for any $\mu\in Q$ with respect to skew-pairing, i.e.,  
$(v^{\mu}_i,u^{\mu}_j)=\delta_{ij}$ for all $1\leqslant i,j\leqslant r(\mu)$,  where $r(\mu)=\text{dim}\mathrm{U}_{\mu}^+$.

For any $\mu,\nu\in Q$,  we know that $\left\{\left.(v^{\nu}_i\bbk_{\nu})\bbk_{\lambda}u^{\mu}_j\right| \text{ for all }\lambda\in \mathbb{Z}\Phi \text{ and }1\leqslant i\leqslant r(\nu), 1\leqslant j\leqslant r(\mu)\right\}$ is a basis of $\mathrm{U}^-_{-\nu}\mathrm{U}^0\mathrm{U}^+_{\mu}$. From equation (\ref{comp}),  we have 
\begin{equation}
\langle(v^{\mu}_h\bbk_{\mu})\bbk_{\lambda'}u^{\nu}_l,(v^{\nu}_i\bbk_{\nu})\bbk_{\lambda}u^{\mu}_j\rangle=\delta_{hj}\delta_{li}(-1)^{|\mu|}(q^{1/2})^{-(\lambda,\lambda')}q^{(2\rho,\mu)}.
\end{equation}
Write $u=\mathop{\sum}\limits_{i,j,\lambda}a_{ij\lambda}(v^{\nu}_i\bbk_{\nu})\bbk_{\lambda}u^{\mu}_j$. The assumption $\langle v,u\rangle=0$ for all $v$ yields
\begin{equation}
\mathop{\sum}\limits_{\lambda\in \mathbb{Z}\Phi}(-1)^{|\nu|}a_{ij\lambda}(q^{1/2})^{-(\lambda,\lambda')}=0,\quad \text{for all}~ i,j,\lambda'.
\end{equation}
Thus, each $a_{ij\lambda}=0$; hence, $u=0$ as well.
\end{proof}
\subsection{Quantum supertrace}
Let $(A,\Delta, \varepsilon, S)$ be a $\mathbb{Z}_2$-graded Hopf algebra over field $k$ and $M,N$ be two $A$-modules. Then $M^*$ is an $A$-module with the action $(af)(m)=(-1)^{|a||f|}f(S(a)m)$ for all $m\in M,a\in A,f\in M^{*}$. $M\otimes N$ is an $A$-module with the action $a(m\otimes n)=\sum(-1)^{|a_{(2)}||m|}a_{(1)}m\otimes a_{(2)}n$ for all $a\in A,m\in M,n\in N$ where $\Delta(a)=\sum a_{(1)}\otimes a_{(2)}$. $\mathrm{Hom}_{k}(M,N)$ is an $A$-module with the action $(af)(m)=\sum (-1)^{|a_{(2)}||f|}a_{(1)}f(S(a_{(2)})m)$ for all $a\in A, m\in M, f\in \mathrm{Hom}_k(M,N)$. Supposing that $M$ is finite-dimensional,  we take $\{m_i\}$ to be a homogeneous basis of $M$ and $\{f_i\}$ to be the dual basis with respect to $\{m_i\}$. Then we have $|m_i|=|f_i|$ for all $i$ and the following isomorphism of $A$-modules:
\begin{align}
\Phi_{M,N}\colon N\otimes M^*\rightarrow\mathrm{Hom}(M,N),\quad n\otimes f\mapsto \varphi_{f,n},
\end{align}
with inverse homomorphism $\Psi_{M,N}\colon g\mapsto\sum g(m_i)\otimes f_i$, where $\varphi_{f,n}(m)=f(m)n$ for all $f\in M^*,g\in\mathrm{Hom}(M,N),m\in M,n\in N$. We also have a homomorphism of $A$-modules $\varepsilon_{M}\colon M^*\otimes M\rightarrow k$ with $\varepsilon_M(f\otimes m)=f(m)$ for all $f\in M^*,m\in M$.

In particular,  $A$ is the quantum superalgebra $\mathrm{U}_q(\mathfrak{g})$. Then we have $S^2(u)=\bbk_{2\rho}^{-1}u\bbk_{2\rho}$ since $(\rho,\alpha_i)=2(\alpha_i,\alpha_i)$ for all $i\in\bbi$.  We obtain a homomorphism of $A$-modules $\psi_{M}\colon M\rightarrow (M^*)^*$ with
\begin{equation}
\big(\psi_{M}(m)\big)(f)=(-1)^{|f||m|}f(\bbk_{2\rho}^{-1}m).
\end{equation}
Combined with the previous statements, we have the following homomorphisms of $A$-modules\begin{equation}
\xymatrix{
\mathrm{Str}_q^M\colon\mathrm{End}(M)\ar[r]^(.6){\Psi_{M,M}} &M\otimes M^*\ar[rr]^(.45){\psi_{M}\otimes 1_{M^*}}&& (M^*)^*\otimes M^*\ar[r]^(.7){\varepsilon_{M^*}} & k.
} 
\end{equation}
This composition is the so-called \textit{quantum supertrace} (we simply replace $\mathrm{Str}^M_q$ with $\mathrm{Str}_q$  if no confusion appears). More precisely,  if $ g\in\mathrm{End}(M)$, then
\begin{align*}
\mathrm{Str}_q(g)=&\varepsilon_{M^*}\circ(\psi_{M}\otimes 1_{M^*})\circ\Psi_{M,M}(g)=(-1)^{|g(m_i)||f_i|}\mathop{\sum}\limits_i f_i\big(\bbk_{2\rho}^{-1}g(m_i)\big)\nonumber\\
=&(-1)^{m_i}\mathop{\sum}\limits_i f_i\big(g(\bbk_{2\rho}^{-1}m_i)\big).
\end{align*} 
Let $A$ be a $\mathbb{Z}_2$-graded Hopf algebra and define the adjoint representation of $A$ as follows: $\mathrm{ad}(a)(b)=\sum (-1)^{|b||a_{(2)}|}a_{(1)}bS(a_{(2)})$. The map $\mathrm{ad}_M\colon A\rightarrow\mathrm{End}(M)$, which takes $a\in A$ to the action of $a$ on $M$, is a homomorphism of $A$-modules and we have
\begin{equation}\label{strad}
\mathrm{Str}_q\circ\mathrm{ad}_M(u)=(-1)^{|m_i|}\mathop{\sum}\limits_i f_i\big(u(\bbk_{2\rho}^{-1}m_i)\big).
\end{equation}
Indeed, this is the supertrace of $u\bbk_{2\rho}^{-1}$ acting on $M$.
In particular, we have
\begin{align}
\mathrm{ad}(\bbe_i)u&=\bbe_iu-(-1)^{|u||e_i|}\bbk_iu\bbk_i^{-1}\bbe_i,\label{ade}\\
\mathrm{ad}(\bbf_i)u&=(\bbf_iu-(-1)^{|u||f_i|}u\bbf_i)\bbk_i,\label{adf}\\
\mathrm{ad}(\bbk_i)u&=\bbk_iu\bbk_i^{-1}.\label{adk}
\end{align}
\subsection{Construct central elements}
In this subsection, we construct central elements for certain finite-dimensional $\mathrm{U}_q(\mathfrak{g})$-modules following Jantzen's book \cite{Jan}.

Let $\varphi:\mathrm{U}^-_{-\mu}\times \mathrm{U}^+_{\nu}\rightarrow k$ be a bilinear map and $\lambda\in \mathbb{Z}\Phi$. There is a unique element $u\in(\mathrm{U}^-_{-\nu}\bbk_{\nu})\bbk_{\lambda}\mathrm{U}^+_{\mu}=\mathrm{U}^-_{-\nu}\bbk_{\nu+\lambda}\mathrm{U}^+_{\mu}$ such that for all $x\in\mathrm{U}^+_{\nu},y\in\mathrm{U}^-_{-\nu},\lambda'\in \mathbb{Z}\Phi$
\begin{equation}
\langle(y\bbk_{\nu})\bbk_{\lambda'}x,u\rangle=\varphi(y,x)(q^{1/2})^{-(\lambda,\lambda')}.
\end{equation}
Indeed,  $u=\sum (-1)^{|y|}\varphi(v^{\mu}_j,u^{\nu}_i)q^{-(2\rho,\mu)}(v^{\nu}_i\bbk_{\nu}\bbk_{\lambda}u^{\mu}_j)$ will work and be unique according to Proposition \ref{propnondeg}.
\begin{lemma}\label{exists}
Let $M$ be a finite-dimensional $\mathrm{U}_q(\mathfrak{g})$-module such that all weights $\lambda$ of $M$ satisfy $2\lambda\in \mathbb{Z}\Phi$. Then there is for each  $m\in M$ and $f\in M^*$ a unique element $u\in\mathrm{U}_q(\mathfrak{g})$ such that $f(vm)=\langle v,u\rangle$ for all $v\in\mathrm{U}_q(\mathfrak{g})$.
\end{lemma}
\begin{proof}
The uniqueness follows from Proposition \ref{propnondeg}. To prove the existence of $u$, we may assume that $f$ and $m$ are weight vectors, since $f(\cdot m)$ depends linearly on $f$ and $m$. Suppose that there are two weights $\lambda$ and $\lambda'$ of $M$ with $m\in M_{\lambda}$ and $f\in(M^*)_{\lambda'}$; i.e., with $f(M_{\lambda''})=0$ for all $\lambda''\neq\lambda$. We have $\mathrm{U}^+_{\nu}m\in M_{\lambda+\nu}$ for all $\nu$. As $M$ has only finitely many weights, there are only finitely many $\nu$ with $\mathrm{U}^+_{\nu}m\neq0$. Since $\mathrm{U}^-_{-\mu}\mathrm{U}^0\mathrm{U}^+_{\nu}m\subseteq M_{\lambda+\nu-\mu}$ for all $\mu$ and $\nu$, we get $f(\mathrm{U}^-_{-\mu}\mathrm{U}^0\mathrm{U}^+_{\nu}m)=0$ unless $\lambda'=\lambda+\nu-\mu$. This shows that there are only finitely many pairs $(\mu,\nu)$ with $f(\mathrm{U}^-_{-\mu}\mathrm{U}^0\mathrm{U}^+_{\nu}m)\neq0$. For all $x\in\mathrm{U^+_{\nu}},y\in\mathrm{U}^-_{-\mu}$ and $\eta\in \mathbb{Z}\Phi$,\begin{equation}
f(y\bbk_{\mu}\bbk_{\eta}xm)=q^{(\eta,\lambda+\nu)}f(y\bbk_{\mu} xm)=(q^{1/2})^{(\eta,2\lambda+2\nu)}f(y\bbk_{\mu}xm).
\end{equation}
For all $\mu$ and $\nu$, the function $(y,x)\mapsto f(y\bbk_{\mu}xm)$ is bilinear. We now use that $2(\lambda+\nu)\in \mathbb{Z}\Phi$. We get an element $u_{\nu\mu}\in\mathrm{U}^-_{-\mu}\mathrm{U}^0\mathrm{U}^+_{\nu}$ with $\langle v,u_{\nu\mu}\rangle=f(vm)$ for all $v\in\mathrm{U}^-_{-\mu}\mathrm{U}^0\mathrm{U}^+_{\nu}$. Then $u=\sum u_{\nu\mu}$ will satisfy our claim. 
\end{proof}
\begin{lemma}\label{zlambda}
Let $M$ be a finite-dimensional $\mathrm{U}_q(\mathfrak{g})$-module such that all weights $\lambda$ of $M$ satisfy $2\lambda\in \mathbb{Z}\Phi$. Then there is a unique element $z_{M}\in\mathrm{U}_q(\mathfrak{g})$ such that $\langle u,z_{M}\rangle$ is equal to the supertrace of $u\bbk_{2\rho}^{-1}$ acting on $M$ for all $u\in\mathrm{U}_q(\mathfrak{g})$. The element $z_{M}$ is contained in the center $\mathcal{Z}(\mathrm{U}_q(\mathfrak{g}))$ of $\mathrm{U}_q(\mathfrak{g})$.
\end{lemma}
\begin{proof}
Let $\{m_1,m_2,\cdots,m_r\}$ be a homogeneous basis of $M$ and $\{f_1,f_2,\cdots,f_r\}$ is the dual basis of $M^*$, then the supertrace of $u\bbk_{2\rho}^{-1}$ acting on $M$ is equal to $\mathop{\sum}\limits_{i=1}^r(-1)^{|m_i|}f_i(u\bbk_{2\rho}^{-1}m_i)=\langle u,z_{M}\rangle$.
In this way, the existence and uniqueness of $z_{M}$ follows from Lemma \ref{exists}. Recall that the map $\mathrm{ad}_{M}\colon \mathrm{U}_q(\mathfrak{g})\rightarrow \mathrm{End}(M)$ is a homomorphism of $\mathrm{U}_q(\mathfrak{g})$-modules. We notice that $\mathrm{Str}_q^{M}\circ\mathrm{ad}_{M}(u)$ is the supertrace of $u\bbk_{2\rho}^{-1}$ acting on $M$ for all $u\in\mathrm{U}_q(\mathfrak{g})$; i.e., $\mathrm{Str}_q^{M}\circ\mathrm{ad}_{M}(u)=\langle u,z_{M}\rangle$ for all $u\in\mathrm{U}_q(\mathfrak{g})$ by (\ref{strad}). This means that for all $u,v\in\mathrm{U}_q(\mathfrak{g})$,
\begin{equation}
\varepsilon(v)\langle u,z_{M}\rangle=v\cdot (\mathrm{Str}_q^{M}\circ\mathrm{ad}_{M}(u)) =\langle \mathrm{ad}(v)u,z_{M}\rangle=(-1)^{|v||u|}\langle u,\mathrm{ad}(S(v))z_{M}\rangle.
\end{equation}
Hence, $\varepsilon(v)z_{M}=(-1)^{|v|(|v|+|z_{M}|)}\mathrm{ad}(S(v))z_{M}=(-1)^{|v|}\mathrm{ad}(S(v))z_{M}$ for all $v\in\mathrm{U}_q(\mathfrak{g})$ by Proposition \ref{propnondeg}. We also have $(-1)^{|v|}\mathrm{ad}(v)z_{M}=\varepsilon(v)z_{M}$ by $\varepsilon\circ S=\varepsilon$. Now (\ref{ade})-(\ref{adk}) easily yield  that $z_{M}$ commutes with all $\bbk_i,\bbe_i,\bbf_i$ and  is therefore central in $\mathrm{U}_q(\mathfrak{g})$.
\end{proof}

\section{Harish-Chandra homomorphism of quantum superalgebras}\label{HChomomorphism}
\subsection{The Harish-Chandra homomorphism}
In the previous section, we used the Drinfeld double to construct an ad-invariant bilinear form in Theorem \ref{adinvariant}, which was also non-degenerate (see Proposition \ref{propnondeg}). By using this form and quantum supertrace, we can construct the central elements of $\mathrm{U}_q(\mathfrak{g})$, which contributed to establish the Harish-Chandra isomorphism  for quantum superalgebras $\mathrm{U}_q(\mathfrak{g})$. Now we are ready to define the Harish-Chandra homomorphism.

For each $\lambda\in \Lambda$, there is an algebra homomorphism, also denoted by $\lambda\colon\mathrm{U}^0\rightarrow\mathbb{C}$, $\lambda(\bbk_{\mu})=q^{(\lambda,\mu)}$ for all $\mu\in \mathbb{Z}\Phi$. Obviously, $(\lambda+\lambda')(h)=\lambda(h)\lambda'(h)$ for $h\in\mathrm{U}^0$ and $\lambda, \lambda'\in\Lambda$.

The triangular decomposition of quantum superalgebra $\mathrm{U}_q(\mathfrak{g})$ implies a direct sum decomposition as follows:
\begin{equation*}
  \mathrm{U}_0=\mathrm{U}^0\oplus\mathop{\bigoplus}\limits_{\nu>0}\mathrm{U}^-_{-\nu}\mathrm{U}^0\mathrm{U}^+_{\nu}.
\end{equation*}
Let $\pi\colon\mathrm{U}_0\rightarrow\mathrm{U}^0$ be the projection with respect to this decomposition. One can check that $\mathop{\bigoplus}\limits_{\nu>0}\mathrm{U}^-_{-\nu}\mathrm{U}^0\mathrm{U}^+_{\nu}$ is a two-sided ideal of $\mathrm{U}_0$. Thus, $\pi$ is an algebra homomorphism.
Denoting the center of $\mathrm{U}_q(\mathfrak{g})$ by $\mathcal{Z}(\mathrm{U}_q(\mathfrak{g}))$\footnote{In general, the center of the Lie superalgebra and quantum superalgebra is $\mathbb{Z}_2$-graded \cite[Section 2.2]{ChengWang}.  Similar to the basic Lie superalgebra case,  the center of $\mathrm{U}_q(\mathfrak{g})$ consists of only even elements.  However, the center contains odd part is also interesting in some aspects; e.g., the skew center of generalized quantum groups \cite{BY}.}, 
we have $\mathcal{Z}(\mathrm{U}_q(\mathfrak{g}))\subseteq \mathrm{U}_0$ by Proposition \ref{centerU_0}. Let $z\in \mathcal{Z}(\mathrm{U}_q(\mathfrak{g}))$ and write $z=\mathop{\sum}\limits_{\nu\geqslant0}z_{\nu}$ where each $z_{\nu}\in\mathrm{U}^-_{-\nu}\mathrm{U}^0\mathrm{U}^+_{\nu}$,  thus $\pi(z)=z_0$.  If we take $v_{\lambda}\in \Delta_q(\lambda)_{\lambda}$, then $zv_{\lambda}=z_0v_{\lambda}=\lambda(z_0)v_{\lambda}$. Since $z$ is the center element of $\mathrm{U}_q(\mathfrak{g})$, this implies $zv=\lambda(z_0)v,~\forall~ v\in \Delta_q(\lambda)$, so it acts as scalar $\lambda(z_0)=\lambda(\pi(z))$ on $\Delta_q(\lambda)$. We set $\chi_{\lambda}\colon\mathcal{Z}(\mathrm{U}_q(\mathfrak{g}))\rightarrow k$ by $\chi_{\lambda}(z)=\lambda(\pi(z)).$

For $\lambda\in \Lambda$, we define an algebra automorphism
\begin{equation*}
  \gamma_{\lambda}\colon\mathrm{U}^0\rightarrow \mathrm{U}^0\quad  \text{by}\quad \gamma_{\lambda}(h)=\lambda(h)h,\quad  \text{ for all } h\in\mathrm{U}^0.
\end{equation*}
Then
\begin{equation*}
  \gamma_{\lambda}(\bbk_{\mu})=q^{(\lambda,\mu)}\bbk_{\mu},\quad \text{ for all }  \lambda\in\Lambda,~ \mu\in \mathbb{Z}\Phi.
\end{equation*}
Obviously, $\gamma_0$ is the identity map, and
\begin{equation*}
  \gamma_{\lambda}\circ\gamma_{\lambda'}=\gamma_{\lambda+\lambda'}\quad \text{and}\quad \lambda'(\gamma_{\lambda}(h))=(\lambda+\lambda')(h),\text{ for all }\lambda,\lambda'\in \Lambda,~h\in \mathrm{U}^0.
\end{equation*}
Inspired by the quantum group case, we define the \textit{Harish-Chandra homomorphism} $\mathcal{HC}$ of $\mathrm{U}_q(\mathfrak{g})$ to be the composite 
\[\mathcal{HC}\colon\mathcal{Z}(\mathrm{U}_q(\mathfrak{g}))\hookrightarrow \mathrm{U}_0\xrightarrow{\pi} \mathrm{U}^0 \xrightarrow{\gamma_{-\rho}}  \mathrm{U}^0.
\]
Assume that $h=\mathcal{HC}(z)=\gamma_{-\rho}\circ\pi(z)$, we have $\chi_{\lambda}(z)=\lambda(\pi(z))=\lambda(\gamma_{\rho}(h))=(\lambda+\rho)(h)$ for all $\lambda\in\Lambda$.
\begin{lemma}\label{injective}
The Harish-Chandra homomorphism $\mathcal{HC}$ is injective.
\end{lemma}
\begin{proof}
Suppose $z=\mathop{\sum}\limits_{\mu\geqslant0}z_{\mu}\in\mathcal{Z}(\mathrm{U}_q(\mathfrak{g}))$ with $\mathcal{HC}(z)=\gamma_{-\rho}\circ\pi(z)=0$ where $z_{\mu}\in\mathrm{U}^-_{-\mu}\mathrm{U}^0\mathrm{U}^+_{\mu}$, then $z_0=\pi(z)=0$ since $\gamma_{-\rho}$ is an algebra automorphism. If we assume $z\neq 0$, then there exists $z_{\mu}\neq 0$ for some $\mu\in Q$. Let $\beta\in Q$ be a minimal element satisfying $\beta>0$ and $z_{\beta}\neq0$. Let $\{y_i\}$ and $\{x_k\}$ be sets of bases of $\mathrm{U}^-_{-\beta}$ and $\mathrm{U}^+_{\beta}$, respectively, and write
\begin{align*}
z_{\beta}=\mathop{\sum}\limits_{j,k}y_jh_{jk}x_k,\quad h_{jk}\in\mathrm{U}^0.
\end{align*}
For all $x\in\mathrm{U}^+_{\gamma},h\in\mathrm{U}^0,y\in\mathrm{U}^-_{-\gamma}$ we have $[\bbe_i,yhx]=[\bbe_i,y]hx+(-1)^{|y||\bbe_i|}y[\bbe_i,hx]$ with $[\bbe_i,y]hx\in\mathrm{U}^-_{-(\gamma-\alpha_i)}\mathrm{U}^0\mathrm{U}^+_{\gamma}$ and $y[\bbe_i,hx]\in\mathrm{U}^-_{-\gamma}\mathrm{U}^0\mathrm{U}^+_{\gamma+\alpha_i}$ by equation (\ref{eyrelation}). Since $[\bbe_i,z]=0$, we have $\mathop{\sum}\limits_{j,k}[\bbe_i,y_j]h_{jk}x_k=0$ by the minimality of $\beta$. Hence $\mathop{\sum}\limits_{j}[\bbe_i,y_j]h_{jk}=0$ for any $k$. Write $\beta=\mathop{\sum}\limits_{i=1}^rm_i\alpha_i$, and let $L_q(\lambda)$ be a finite-dimensional module with the highest weight vector $v_{\lambda}$. Then we have
\begin{align*}
\bbe_i\Big(\mathop{\sum}\limits_j\lambda(h_{jk})y_jv_{\lambda}\Big)=\mathop{\sum}\limits_j[\bbe_i,y_j]h_{jk}v_{\lambda}=0,
\end{align*}
for all $i\in\mathbb{I}$. So $\mathop{\sum}\limits_j\lambda(h_{jk})y_jv_{\lambda}$ generates a proper submodule of $L_q(\lambda)$, and we get $\mathop{\sum}\limits_j\lambda(h_{jk})y_jv_{\lambda}=0$. The linear map $\mathrm{U}^-_{-\beta}\rightarrow L_q(\lambda)$ given by $y\mapsto yv_{\lambda}$ is bijective if $\lambda$ satisfies the condition of Proposition \ref{3.3}. Hence, $\mathop{\sum}\limits_j\lambda(h_{jk})y_j=0$. Therefore, $h_{jk}=0$ for any $j,k$, and $z_{\beta}=0$. This contradicts the choice of $\beta$ with $z_{\beta}\neq0$. Thus, $z=0$ and $\mathcal{HC}$ is injective.
\end{proof}
\subsection{Description of the image of the $\mathcal{HC}$}
The image of the $\mathcal{HC}$ is much more complicated. We split it into the following three lemmas. Recall that 
the Weyl group $W$ acts naturally on $\mathrm{U}^0$ as $w(\bbk_{\mu})=\bbk_{w\mu}$ for all $w\in W$ and $\mu\in \mathbb{Z}\Phi$. We have $(w\lambda)(w h)=\lambda(h)$ for all $w\in W, \lambda\in \Lambda,$ and $h\in\mathrm{U}^0.$
\begin{lemma} \label{U0}
The restriction of the image of Harish-Chandra homomorphism on the center of quantum superalgebra $\mathrm{U}_q(\mathfrak{g})$ is contained in the $W$-invariant of $\mathrm{U}^0$; i.e., $\mathcal{HC}\big(\mathcal{Z}(\mathrm{U}_q(\mathfrak{g}))\big)\subset(\mathrm{U}^0)^{W}$.
\end{lemma}
\begin{proof}
The character of the Verma module $\Delta_q(\lambda)$ with the highest weight $\lambda\in\Lambda$ is given by $\mathrm{ch}\Delta_q(\lambda)=\frac{1}{D}e^{\mu+\rho}$ where $D=\mathop{\prod}\limits_{\beta\in\Phi_{\bar{1}}^+}(e^{\beta/2}-e^{-\beta/2}) / \mathop{\prod}\limits_{\alpha\in\Phi_{\bar{0}}^+}(e^{\alpha/2}-e^{-\alpha/2})$ owing to \cite[Theorem 1]{Kac2} and Theorem \ref{repquantum}.

Since the character of a module is equal to the sum of the characters of its composition factors, we have
\begin{align*}
\mathrm{ch}\Delta_q(\lambda)=\mathop{\sum}\limits_{\mu}b_{\lambda\mu}\mathrm{ch}L_q(\lambda)
\end{align*}
where $b_{\lambda\mu}\in\mathbb{Z}_+$ and $b_{\lambda\lambda}=1$. Since $\Delta_q(\lambda)$ is a highest weight module, $b_{\lambda\mu}\neq0\Rightarrow\lambda-\mu\in\mathop{\sum}\limits_{i}\mathbb{Z}_+\alpha_i$ and also $\chi_{\lambda}=\chi_{\mu}$. Hence, we have
\begin{align*}
\mathrm{ch}L_q(\lambda)=\mathop{\sum}\limits_{\mu}a_{\lambda\mu}\mathrm{ch}\Delta_q(\lambda)\quad \text{ and } \quad D\mathrm{ch}L_q(\lambda)=\mathop{\sum}\limits_{\mu}a_{\lambda\mu}e^{\mu+\rho}
\end{align*}
where $a_{\lambda\mu}\in\mathbb{Z}$ with $a_{\lambda\lambda}=1$, and $a_{\lambda\mu}=0 \text{ unless } \lambda-\mu\in\mathop{\sum}\limits_{i}\mathbb{Z}_+\alpha \text{ and } \chi_{\lambda}=\chi_{\mu}$.

Assume for now that $L(\lambda)$ is finite-dimensional. Then $L_q(\lambda)$ is a semisimple $\mathfrak{g}_{\bar{0}}$-module, and $\mathrm{ch}L_q(\lambda)$ is $W$-invariant as a result. On the other hand, $w(D)=(-1)^{l(w)}D$ for all $w\in W$, and hence $D\mathrm{ch}L_q(\lambda)$ can be written as 
\begin{align*}
\mathop{\sum}\limits_{\mu\in X}a_{\lambda\mu}\mathop{\sum}\limits_{w\in W}(-1)^{l(w)}e^{w(\mu+\rho)},
\end{align*}
where $X$ consists of $\Phi_{\bar{0}}^+$-dominant integral weights such that $a_{\lambda\mu}\neq0$. Moreover, $a_{\lambda,w(\lambda+\rho)-\rho}=(-1)^{l(w)}a_{\lambda\lambda}=(-1)^{l(w)}$. Hence, we have $\chi_{\lambda}=\chi_{w(\lambda+\rho)-\rho}$ for all $w\in W,\lambda\in \Lambda_{f.d.}$, where $\Lambda_{f.d.}=\{\lambda\in\Lambda| \mathrm{dim}L_q(\lambda)<\infty \}$.

%

For $z\in \mathcal{Z}(\mathrm{U}_q(\mathfrak{g}))$, we set $h=\mathcal{HC}(z)$. Assuming that $\lambda\in\Lambda$ and $L_q(\lambda)$ is finite-dimensional, we get $(\lambda+\rho)(h)=\chi_{\lambda}(z)=\chi_{w(\lambda+\rho)-\rho}(z)=(w(\lambda+\rho))(h)=(\lambda+\rho)(w h)$. Hence $\lambda(w h-h)=0$ for all $w\in W$. Fix $w$ and write $w h-h=\mathop{\sum}\limits_{\mu}a_{\mu}\bbk_{\mu}$. Then $\lambda\big(\mathop{\sum}\limits_{\mu}a_{\mu}\bbk_{\mu}\big)=\mathop{\sum}\limits_{\mu}a_{\mu}q^{(\lambda,\mu)}=0$  for all $\lambda\in \Lambda_{f.d.}$. Thus, $w h-h=0$ and $h\in(\mathrm{U}^0)^{W}$ because the bilinear form on $\Lambda_{f.d.}\times \mathbb{Z}\Phi$ is non-degenerate in the second component.
\end{proof}

Set
\begin{equation}
(\mathrm{U}^0_{\mathrm{ev}})^{W}=\Bigg\{\mathop{\sum}\limits_{\mu} a_{\mu}\bbk_{\mu}\Bigg|\mu\in 2\Lambda\cap \mathbb{Z}\Phi~\text{and}~a_{\mu}=a_{w
 \mu},~\forall w\in W\Bigg\}.
\end{equation}
\begin{lemma}\label{Uev}
The Harish-Chandra homomorphism $\mathcal{HC}$ maps $\mathcal{Z}(\mathrm{U}_q(\mathfrak{g}))$ to $(\mathrm{U}^0_{\mathrm{ev}})^{W}$.
\end{lemma}
\begin{proof}
Take an arbitrary $z\in \mathcal{Z}(\mathrm{U}_q(\mathfrak{g}))$, we can write $\mathcal{HC}(z)=\mathop{\sum}\limits_{\mu}a_{\mu}\bbk_{\mu}$ with $a_{w\mu}=a_{\mu}$ for any $w\in{W}$.  We only need to prove $\langle\mu,\alpha\rangle\in 2\mathbb{Z}$ for all $\mu\in\mathbb{Z}\Phi$ with $a_{\mu}\neq 0$,  $\alpha\in\Phi_{\bar{0}}$.

For each group homomorphism $\sigma\colon\mathbb{Z}\Phi\rightarrow\{\pm 1\}$,  we can define an automorphism $\tilde{\sigma}$ of $\mathrm{U}_q(\mathfrak{g})$ by
\begin{align*}
\tilde{\sigma}(\bbk_{\mu})=\sigma(\mu)\bbk_{\mu},\quad \tilde{\sigma}(\bbe_{i})=\bbe_{i},\quad \tilde{\sigma}(\bbf_{i})=\sigma(\alpha_i)\bbf_{i}.
\end{align*}
Obviously, $\tilde{\sigma}$ maps the center $\mathcal{Z}(\mathrm{U}_q(\mathfrak{g}))$ to itself. One can check that $\mathcal{HC}=\gamma_{-\rho}\circ\pi$ commutes with $\tilde{\sigma}$. We already have $\mathcal{HC}(\tilde{\sigma}(z))=\tilde{\sigma}\big(\mathop{\sum}\limits_{\mu}a_{\mu}\bbk_{\mu}\big)=
\mathop{\sum}\limits_{\mu}a_{\mu}\sigma(\mu)\bbk_{\mu}$. Since $\tilde{\sigma}(z)$ is central, the sum is in $(\mathrm{U}^0)^{W}$; so we have $a_{\mu}\sigma(\mu)=a_{w\mu}\sigma(w\mu)=a_{\mu}\sigma(w\mu)$ for all $w \in W$. This means: if $a_{\mu}\neq0$, then $\sigma(\mu)=\sigma(w\mu)$ for all $w\in W$.  Thus, $\sigma(\mu-s_{\alpha}\mu)=1$ for all $\alpha\in\Phi_{\bar{0}}^+, \mu\in \mathbb{Z}\Phi$. For each $\alpha$, we can choose $\sigma$ such that $\sigma(\alpha)=-1$. Therefore, $(-1)^{\langle\mu,\alpha\rangle}=1$ and $\langle\mu,\alpha\rangle\in2\mathbb{Z}$.
\end{proof}
For $v\in\Lambda$ and $\alpha\in\Phi_{\mathrm{iso}}$,  we set $A^{\alpha}_{\nu}=\{\nu+n\alpha|n\in\mathbb{Z}\}$. Clearly, $\Lambda=\mathop{\bigcup}\limits_{\nu\in \Lambda} A^{\alpha}_{\nu}$.  Let 
\begin{equation}\label{Wsup}
(\mathrm{U}^0_{\mathrm{ev}})^{W}_{\mathrm{sup}}=\Bigg\{\mathop{\sum}\limits_{\mu} a_{\mu}\bbk_{\mu}\in(\mathrm{U}^0_{\mathrm{ev}})^{W}\Bigg| \mathop{\sum}\limits_{\mu \in A^{\alpha}_{\nu}}a_{\mu}=0,~\forall\alpha\in\Phi_{\mathrm{iso}} \text{ with } (\nu,\alpha)\neq0 \Bigg\}.
\end{equation}
\begin{lemma}\label{Usup}
The Harish-Chandra homomorphism $\mathcal{HC}$ maps $\mathcal{Z}(\mathrm{U}_q(\mathfrak{g}))$ to $(\mathrm{U}^0_{\mathrm{ev}})^{W}_{\mathrm{sup}}$.
\end{lemma}
\begin{proof}
We claim that if  $\alpha\in\Phi_{\mathrm{iso}}$ and $(\lambda+\rho, \alpha)=0, $ then $\chi_{\lambda}=\chi_{\lambda-k\alpha}$ for any $k\in\mathbb{Z}$.
Indeed, if $\alpha=\alpha_s$  and $(\lambda,\alpha_s)=0$, then we get a non-trivial homomorphism  $\varphi\colon\Delta_q(\lambda-\alpha_s)\rightarrow\Delta_q(\lambda)$ according to Lemma \ref{embeding}. In this way, $z\in \mathcal{Z}(\mathrm{U}_q(\mathfrak{g}))$ acts by the same constant on both modules; i.e., $\chi_{\lambda}(z)=(\lambda+\rho)(h)=(\lambda-\alpha_{s}+\rho)(h)=\chi_{\lambda-\alpha_{s}}(z)$ where $h=\mathcal{HC}(z)=\gamma_{-\rho}\circ\pi(z)$. Thus, $\chi_{\lambda}=\chi_{\lambda-\alpha_{s}}$.

For any $\alpha\in\Phi_{\mathrm{iso}}$, if $(\lambda+\rho,\alpha)=0$, then there exists $w\in W$ such that $w(\alpha)=\alpha_s$. Based on the $W$-invariance of $(\cdot, \cdot)$, we have $\big(w(\lambda+\rho),w(\alpha)\big)=(\lambda+\rho,\alpha)=0$, so 
$$\chi_{\lambda}=\chi_{w(\lambda+\rho)-\rho}=\chi_{w(\lambda+\rho)-\alpha-\rho}=\chi_{\lambda-\alpha}.$$

This implies $\chi_{\lambda}=\chi_{\lambda-\alpha}$, so we conclude that $\chi_{\lambda}=\chi_{\lambda-k\alpha}$ for all $k\in\mathbb{Z}$. 

Now suppose $h=\gamma_{-\rho}\circ\pi(z)=\mathop{\sum}\limits_{\mu} a_{\mu}\bbk_{\mu}$ for some $z\in \mathcal{Z}(\mathrm{U}_q(\mathfrak{g}))$ and $\alpha\in\Phi_{\mathrm{iso}}$, by $\chi_{\lambda}(z)=(\lambda+\rho)\Big(\mathop{\sum}\limits_{\mu} a_{\mu}\bbk_{\mu}\Big)$ and $\chi_{\lambda}=\chi_{\lambda-\alpha}$ for all $(\lambda+\rho,\alpha)=0$. We know 
\begin{equation}
(\lambda+\rho+\alpha)\bigg(\mathop{\sum}\limits_{\mu} a_{\mu}\bbk_{\mu}\bigg)=(\lambda+\rho)\bigg(\mathop{\sum}\limits_{\mu } a_{\mu}\bbk_{\mu}\bigg),
\end{equation} for all $(\lambda+\rho,\alpha)=0$, hence\begin{equation}\label{rewrite}
\mathop{\sum}\limits_{\mu } a_{\mu}q^{(\lambda+\rho,\mu)}\left(q^{(\mu,\alpha)}-1\right)=0.
\end{equation}
Notice that $(\lambda+\rho,\nu)=(\lambda+\rho,\nu')$ and $(\mu,\alpha)=(\nu',\alpha)$ if $A_{\nu}^{\alpha}=A_{\nu'}^{\alpha}$. There is a finite subset $X\subset \Lambda$ such that $A_{\nu}\neq A_{\nu'}$ for all $\nu,\nu'\in X$ and $\mathrm{Supp}(h)\subset\mathop{\bigcup}\limits_{\nu\in X}A_{\nu}^{\alpha}$, where $\mathrm{Supp}(h)$ is the set of $\mu\in 2\Lambda\cap \mathbb{Z}\Phi$ for which $a_\mu$ is nonzero. This means we can rewrite Equation (\ref{rewrite}) as
\begin{align*}
\mathop{\sum}\limits_{\nu\in X} \bigg(\mathop{\sum}\limits_{\mu\in A_{\nu}^{\alpha}} a_{\mu}\bigg)\left(q^{(\nu,\alpha)}-1\right) q^{(\lambda+\rho,\nu)}=0.
\end{align*}

Since $(\alpha,\alpha)=0$ and the bilinear form on $\Lambda\times \mathbb{Z}\Phi$ is non-degenerate, there is a $\lambda+\rho\in\Lambda$ such that $(\lambda+\rho,\alpha)=0$  and $(\lambda+\rho,\nu)\neq(\lambda+\rho,\nu')$ for all $\nu\neq \nu'$ with $\nu,\nu'\in X$. Let $\{n_{\nu}\}_{\nu\in X}$ be positive integers that are distinct. We get
\begin{align*}
\mathop{\sum}\limits_{\nu'\in X} \bigg(\mathop{\sum}\limits_{\mu\in A_{\nu'}^{\alpha}} a_{\mu}\bigg)\left(q^{(\nu',\alpha)}-1\right) q^{\left(n_{\nu}(\lambda+\rho),\nu'\right)}=0,
\end{align*}
for all $\nu\in X$ and the square matrix $\left(q^{(n_{\nu}(\lambda+\rho),\nu')}\right)_{\nu\nu'}$ is invertible. Therefore,
\begin{equation}
\bigg(\mathop{\sum}\limits_{\mu \in A^{\alpha}_{\nu}}a_{\mu}\bigg)\left(q^{(\nu,\alpha)}-1\right)=0,
\end{equation}
for all $\nu\in X$, and $\mathop{\sum}\limits_{\mu \in  A^{\alpha}_{\nu}}a_{\mu}=0$ if $(\nu,\alpha)\neq0$.
\end{proof}
\begin{example}
We give some explicit elements in $(\mathrm{U}^0_{\mathrm{ev}})^{W}_{\mathrm{sup}}$ when $\mathfrak{g}$ is of small rank.
\begin{enumerate}
\item Let $\mathfrak{g}=A(1,0)$. In such a case, $\Phi_{\bar{1}}^+=\{\alpha_2,\alpha_1+\alpha_2\}$ and $2\Lambda\cap \mathbb{Z}\Phi=\mathbb{Z}\alpha_1+2\mathbb{Z}\alpha_2$.  If $\lambda=k_1\alpha_1+2k_2\alpha_2$ is a dominant weight,  then we have   $k_1\geqslant k_2$ and $k_1,k_2\in\mathbb{Z}$. Furthermore, $W\lambda=\{\lambda,\lambda-2(k_1-k_2)\alpha_1\}$. Thus $k_{\lambda}=\bbk_{\lambda}-\bbk_{\lambda-2\alpha_2}-\bbk_{\lambda-2\alpha_1-2\alpha_2}+\bbk_{\lambda-2\alpha_1-4\alpha_2}+\bbk_{\lambda-2(k_1-k_2)\alpha_1}-\bbk_{\lambda-2(k_1-k_2)\alpha_1-2\alpha_2}-\bbk_{\lambda-2(k_1-k_2)\alpha_1-2\alpha_1-2\alpha_2}+\bbk_{\lambda-2(k_1-k_2)\alpha_1-2\alpha_1-4\alpha_2}\in(\mathrm{U}^0_{\mathrm{ev}})^{W}_{\mathrm{sup}} $.
\item  Let $\mathfrak{g}=C(2)$. As a result, $\Phi_{\bar{1}}^+=\{\alpha_1,\alpha_1+\alpha_2\}$ and $2\Lambda\cap \mathbb{Z}\Phi=2\mathbb{Z}\alpha_1+\mathbb{Z}\alpha_2$. If $\lambda=2k_1\alpha_1+k_2\alpha_2$ is a dominant weight, then we have $k_2\geqslant k_1$ and $k_1,k_2\in\mathbb{Z}$. Furthermore,  $W\lambda=\{\lambda,\lambda-2(k_2-k_1)\alpha_2\}$. Thus $k_{\lambda}=\bbk_{\lambda}-\bbk_{\lambda-2\alpha_1}-\bbk_{\lambda-2\alpha_1-2\alpha_2}+\bbk_{\lambda-4\alpha_1-2\alpha_2}+\bbk_{\lambda-2(k_2-k_1)\alpha_2}-\bbk_{\lambda-2(k_2-k_1)\alpha_2-2\alpha_1}-\bbk_{\lambda-2(k_2-k_1)\alpha_2-2\alpha_1-2\alpha_2}+\bbk_{\lambda-2(k_2-k_1)\alpha_2-4\alpha_1-2\alpha_2}\in(\mathrm{U}^0_{\mathrm{ev}})^{W}_{\mathrm{sup}}$.
\item Let $\mathfrak{g}=B(1,1)$. In this case, the positive isotropic roots of $\mathfrak{g}$ are $\{\alpha_1, \alpha_1+2\alpha_2\}$ and $2\Lambda\cap \mathbb{Z}\Phi=2\mathbb{Z}\alpha_1+\mathbb{Z}\alpha_2$. If $\lambda=\lambda_1\delta_1+\mu_1\varepsilon_1\in2\Lambda\cap Q$ is a dominant weight,  then  we have $\lambda_1\neq0$,  $\lambda_1-2,2\mu_1\in 2\mathbb{Z}_+$. Furthermore, $W\lambda=\{ \pm\lambda_1\delta_1\pm\mu_1\varepsilon_1\}$. Thus  $k_{\lambda}=\mathop{\sum}\limits_{w\in W}w(\bbk_{\lambda}-\bbk_{\lambda-2\alpha_1}-\bbk_{\lambda-2\alpha_1-4\alpha_2}+\bbk_{\lambda-4\alpha_1-4\alpha_2})\in(\mathrm{U}^0_{\mathrm{ev}})^{W}_{\mathrm{sup}}$.
\end{enumerate}
\end{example}

\subsection{Proof of Theorem A}

In order to prove the surjectivity of $\mathcal{HC}$, we need to investigate the Grothendieck rings $K(\mathfrak{g})$ of finite-dimensional representations of the basic classical Lie superalgebras $\mathfrak{g}$. In the  following proposition, we identify the algebra $(\mathrm{U}^0_{\mathrm{ev}})^{W}_{\mathrm{sup}}$ with $J_{\mathrm{ev}}(\mathfrak{g})$, which plays a crucial role on the surjectivity of $\mathcal{HC}$.
\begin{proposition}\label{J(U)}
$(\mathrm{U}^0_{\mathrm{ev}})^{W}_{\mathrm{sup}}=J_{\mathrm{ev}}(\mathfrak{g})$.
\end{proposition}
\begin{proof}
For $h=\mathop{\sum}\limits_{\mu} a_{\mu}\bbk_{\mu}\in(\mathrm{U}^0_{\mathrm{ev}})^{W}_{\mathrm{sup}} $, we define $\mathrm{Supp}(h)$ as the set of $\mu\in 2\Lambda\cap \mathbb{Z}\Phi$ for which $a_\mu$ is nonzero. For any $\alpha\in\Phi_{\mathrm{iso}}$, take a finite set $X$ such that $\mathrm{Supp}(h)\subset \mathop{\bigcup}\limits_{\nu\in X} (\nu+\mathbb{Z}_+\alpha)$. Furthermore, $\mathrm{Supp}(h)\subset \mathop{\bigcup}\limits_{\nu\in X} (\nu+2\mathbb{Z}_+\alpha)$ since there is an even root $\beta$ such that $\frac{2(\alpha,\beta)}{(\beta,\beta)}=1$. Then
\begin{align}
D_{\alpha}(h)=\mathop{\sum}\limits_{\mu} a_{\mu}(\mu,\alpha)\bbk_{\mu}=\mathop{\sum}\limits_{\nu\in X}\mathop{\sum}\limits_{k\in\mathbb{Z}_+} a_{\nu+2k\alpha}(\nu,\alpha)\bbk_{\nu+2k\alpha}
\end{align}
and
\begin{align*}
\mathop{\sum}\limits_{k\in\mathbb{Z}_+} a_{\nu+2k\alpha}(\nu,\alpha)\bbk_{\nu+2k\alpha}\in (\bbk_{\alpha}^2-1), \quad \text{ for all } \nu\in X
\end{align*}
because $\mathop{\sum}\limits_{k\in\mathbb{Z}_+} a_{\nu+k\alpha}=0$ for all $\nu\in X$ with $(\nu,\alpha)\neq 0$ and $\mathop{\sum}\limits_{k\in\mathbb{Z}_+} a_{\nu+2k\alpha}(\nu,\alpha)\bbk_{\nu+2k\alpha}=0$ for all $\nu\in X$ with $(\nu,\alpha)=0$.

On the other hand, take an element $h=\mathop{\sum}\limits_{\mu} a_{\mu}\bbk_{\mu}\in J_{\mathrm{ev}}(\mathfrak{g})$, then $$D_{\alpha}(h)=\mathop{\sum}\limits_{\mu} a_{\mu}(\mu,\alpha)\bbk_{\mu}=\mathop{\sum}\limits_{\nu\in X}\mathop{\sum}\limits_{k\in\mathbb{Z}_+} a_{\nu+k\alpha}(\nu,\alpha)\bbk_{\nu+k\alpha}\in(\bbk_{\alpha}-1),$$
for any $\alpha\in\Phi_{\mathrm{iso}}$. Therefore, $\mathop{\sum}\limits_{k\in\mathbb{Z}_+} a_{\nu+k\alpha}\bbk_{\nu+k\alpha}\in(\bbk_{\alpha}-1)$ for any $\nu\in X$ with $(\nu,\alpha)\neq 0$. This implies that  $\mathop{\sum}\limits_{\mu \in A^{\alpha}_{\nu}}a_{\mu}=\mathop{\sum}\limits_{k\in\mathbb{Z}_+} a_{\nu+k\alpha}=0$.
 \end{proof}
\begin{proposition}\label{Psi_R}
There is a linear map $\Psi_{\mathcal{R}}\colon k\otimes_{\mathbb{Z}}K_{\mathrm{ev}}(\mathrm{U}_q(\mathfrak{g}))\to \mathcal{Z}(\mathrm{U}_q(\mathfrak{g}))$ such that the diagram in the introduction commutes.
\end{proposition}
\begin{proof}
Define a map from $\Psi_{\mathcal{R}}\colon k\otimes_{\mathbb{Z}}K_{\mathrm{ev}}(\mathrm{U}_q(\mathfrak{g}))\to \mathcal{Z}(\mathrm{U}_q(\mathfrak{g}))$ with $\Psi_{\mathcal{R}}([M])=z_M$ where $z_M$ is defined in Lemma \ref{zlambda}.
We need to prove the well-defineness of $\Psi_{\mathcal{R}}$ and $\iota\circ\mathcal{HC}(z_M)=\mathrm{Sch}([M])$ for all $M$ in $\mathrm{U}$-{\bf mod} with all weights contained in $\Lambda\cap \frac{1}{2}\mathbb{Z}\Phi$.


Choose a homogeneous basis $\{m_1,\cdots,m_k,\cdots,m_l\}$ of $M$ such that $\{m_1,\cdots,m_k\}$ is a basis of $L$ and $\{\bar{m}_{k+1},\cdots,\bar{m}_l\}$ is a basis of $N$. Let $\{f_1,\cdots,f_l\}$ be the dual basis of $M$, then $\{f_1,\cdots,f_k\}$ and $\{\bar{f}_{k+1},\cdots,\bar{f}_l\}$ can be viewed as dual bases of $L$ and $N$, respectively. So $\{\pi\otimes m_1,\cdots,\pi\otimes m_l\}$ (resp. $\{\pi\otimes f_1,\cdots,\pi\otimes f_l\}$) is the basis (resp. dual bases) of $\mathrm{wt}(M)$, and $|\pi\otimes f_i|=|\pi\otimes m_i|=-|m_i|=-|f_i|$ for all $i$. Hence,
\begin{align*}
\langle u,z_{M}\rangle&=\mathop{\sum}\limits_{i=1}^l(-1)^{|m_i|}f_i(u\bbk_{2\rho}^{-1}m_i)\\
&=\mathop{\sum}\limits_{i=1}^k(-1)^{|m_i|}f_i(u\bbk_{2\rho}^{-1}m_i)+
\mathop{\sum}\limits_{i=k+1}^l(-1)^{|m_i|}f_i(u\bbk_{2\rho}^{-1}m_i)\\
&=\mathop{\sum}\limits_{i=1}^k(-1)^{|m_i|}f_i(u\bbk_{2\rho}^{-1}m_i)+
\mathop{\sum}\limits_{i=k+1}^l(-1)^{|\bar{m}_i|}\bar{f}_i(u\bbk_{2\rho}^{-1}\bar{m}_i)\\
&=\langle u,z_{L}\rangle+\langle u,z_{N}\rangle=\langle u,z_{L}+z_{N}\rangle;\\
\langle u,z_{M}\rangle&=\mathop{\sum}\limits_{i=1}^l(-1)^{|m_i|}f_i(u\bbk_{2\rho}^{-1}m_i)=-\mathop{\sum}\limits_{i=1}^l(-1)^{|\pi\otimes m_i|}(\pi\otimes f_i)\big(u\bbk_{2\rho}^{-1}(\pi\otimes m_i)\big)\\
&=-\langle u,z_{\Pi(M)}\rangle.
\end{align*}
Therefore, $z_{L}-z_M+z_N=0$ and $z_M+z_{\Pi(M)}=0$ according to Proposition \ref{propnondeg}.

Since $z_M$ is central, we have $z_M=\mathop{\sum}\limits_{\mu\geqslant0}z_{M,\mu}$ where $z_{M,\mu}\in\mathrm{U}^-_{-\mu}\mathrm{U}^0\mathrm{U}^+_{\mu}$. Write $z_{M,0}=\mathop{\sum}\limits_{\nu} a_{\nu}\bbk_{\nu}$. Then we have
\begin{equation*}
\langle \bbk_{\mu'},z_{M}\rangle=\langle \bbk_{\mu'},z_{M,0}\rangle=\mathop{\sum}\limits_{\nu} a_{\nu}\left(q^{1/2}\right)^{-(\nu,\mu')},
\end{equation*}
for all $\mu'\in \mathbb{Z}\Phi$.
On the other hand, this is the supertrace of $\bbk_{\mu'-2\rho}$ acting on $M$. This means it is equal to\begin{equation*}
\mathop{\sum}\limits_{\lambda'}\mathrm{sdim}M_{\lambda'}q^{(\lambda',\mu'-2\rho)}=\mathop{\sum}\limits_{\lambda'}\mathrm{sdim}M_{\lambda'}q^{-2(\lambda',\rho)}\left(q^{1/2}\right)^{(2\lambda',\mu')}.
\end{equation*} 
A comparison of these two formulas shows that
\begin{equation*}
z_{M,0}=\mathop{\sum}\limits_{\lambda'}\mathrm{sdim}M_{\lambda'}q^{(-2\lambda',\rho)}\bbk_{-2\lambda'}.
\end{equation*}
We have $z_{M,0}=\pi(z_{M})$, hence
\begin{equation}
\gamma_{-\rho}\circ\pi(z_{M})=\mathop{\sum}\limits_{\lambda'}\mathrm{sdim}M_{\lambda'}\bbk_{-2\lambda'},
\end{equation}
and $\iota\circ\mathcal{HC}(z_M)=\mathop{\sum}\limits_{\lambda'}\mathrm{sdim}M_{\lambda'}e^{\lambda'}=\mathrm{Sch}([M]). $
\end{proof}
\textbf{\textit{Proof of Theorem A:}}
\begin{equation*}
\xymatrix@C=25pt{
k\otimes_{\mathbb{Z}}K_{\mathrm{ev}}(\mathrm{U}_q(\mathfrak{g}))\ar@{-->}[d]_{\Psi_{\mathcal{R}}}&&k\otimes_{\mathbb{Z}}K_{\mathrm{ev}}(\mathfrak{g})\ar@{_{(}->}[ll]\ar@{-->}[d]_{\cong} \\
\mathcal{Z}(\mathrm{U}_q(\mathfrak{g}))\ar@{-->}[r]^(.42){\mathcal{HC}}
&(\mathrm{U}^0_{\mathrm{ev}})_{\mathrm{sup}}^W\ar@{=}[r]& k\otimes_{\mathbb{Z}}J_{\mathrm{ev}}(\mathfrak{g})
}
\end{equation*}
The injectivity of $\mathcal{HC}$ follows from \ref{injective}, so we only need to prove $\mathrm{Im}\mathcal{HC}=(\mathrm{U}^0_{\mathrm{ev}})_{\mathrm{sup}}^W$. Based on Proposition \ref{Psi_R}, the above diagram is commutative, so $\mathrm{Im}\mathcal{HC}=(\mathrm{U}^0_{\mathrm{ev}})_{\mathrm{sup}}^W$. \hfill $\square$\\

By using $\iota\circ\mathcal{HC}\circ\Psi_{\mathcal{R}}([M])=\mathrm{Sch}([M])$ for all $[M]\in K_{\mathrm{ev}}(\mathrm{U}_q(\mathfrak{g}))$, we get $\Psi_{\mathcal{R}}$ is injective. All morphisms in the diagram above are algebra isomorphisms as a result. Furthermore, for any $[M]\in K_{\mathrm{ev}}(\mathrm{U}_q(\mathfrak{g}))$, there exists $\mathop{\sum}\limits_ia_i[L(\lambda_i)]$ with $a_i\in k$ such that $\jmath(\mathop{\sum}\limits_ia_i[L(\lambda_i)])=[M]$, and these $\lambda_i$ are distinct. Let $X=\{\lambda_i| a_i\notin \mathbb{Z} \}$. Supposing that $X$ is nonempty and taking a maximal element $\lambda_t$ in $X$ for some $t$, we get $\mathrm{dim}M_{\lambda_t}=\mathop{\sum}\limits_ia_{i}\mathrm{dim}L(\lambda_i)_{\lambda_t}\in\mathbb{Z}$ and $\mathrm{dim}L(\lambda_i)_{\lambda_t}=\delta_{it}$. Thus $a_{t}=\mathrm{dim}M_{\lambda_t}$ is an integer, contradicting $\lambda_t\in X$. Therefore, $X$ is empty and $a_i\in\mathbb{Z}$ for all $i$. Thus, $K_{\mathrm{ev}}(\mathfrak{g})\hookrightarrow K_{\mathrm{ev}}(\mathrm{U}_q(\mathfrak{g}))$ is an isomorphism induced by $\jmath$. 

\begin{remark}\label{A(1,1)}
In Appendix \ref{appendixB}, we describe the $J_{\mathrm{ev}}(\mathfrak{g})$ in the sense of Sergeev and Veselov \cite{SerVes} and illustrate why $K_{\mathrm{ev}}(\mathfrak{g})\ncong J_{\mathrm{ev}}(\mathfrak{g})$ if $\mathfrak{g}=A(1,1)$ since $u-v=\bbk_1+\bbk_1^{-1}-\bbk_3-\bbk_3^{-1}\in (\mathrm{U}^0_{\mathrm{ev}})_{\mathrm{sup}}^W=J_{\mathrm{ev}}(\mathfrak{g})$ and $u-v\notin J(A(1,1))$. Therefore, $J(A(1,1))\subseteq\mathrm{Im}(\mathcal{HC})\subseteq J_{\mathrm{ev}}(\mathfrak{g})$. However, the image of $\mathcal{HC}$ for $\mathfrak{g}=A(1,1)$ has not yet determined.

\end{remark}

\section{Center of quantum superalgebras}\label{Centers}
\subsection{Quasi-R-matrix}
In Section \ref{HChomomorphism},  we established the $\mathcal{HC}$ for quantum superalgebras and proved that the center $\mathcal{Z}(\mathrm{U}_q(\mathfrak{g}))$ is isomorphic to $(\mathrm{U}^0_{\mathrm{ev}})^{W}_{\mathrm{sup}}$, the subalgebra of the ring of exponential super-invariants  $J_{\mathrm{ev}}(\mathfrak{g})$.  This section studies the structural theorem for the center.   Our approach to obtaining a structural theorem for quantum superalgebras takes advantage of the quasi-R-matrix,  which is inspired by  \cite{ZGB1,ZGB2}.  Recently,  based on main results \cite{LXZ},  Dai ang Zhang \cite{DZ} used the similar method to investigate explicit generators and relations for the center of the quantum group.  They proved that the center $\mathcal{Z}(\mathrm{U}_q(\mathfrak{g}))$ of quantum group $\mathrm{U}_q(\mathfrak{g})$ is isomorphic to the subring of Grothendieck algebra $K(\mathrm{U}_q(\mathfrak{g}))$.

For each $\mu\in Q$,  we take $u^{\mu}_1,u^{\mu}_2,\cdots,u^{\mu}_{r(\mu)}$ to be a basis of  $\mathrm{U}^+_{\mu}$. Since the skew-pairing between the $\mathrm{U}^+$ and $\mathrm{U}^-$ is non-degenerate, we can take the  dual basis $v^{\mu}_1,v^{\mu}_2,\cdots,v^{\mu}_{r(\mu)}$ of $\mathrm{U}^-_{-\mu}$,  with respect to 
$(v^{\mu}_i,u^{\mu}_j)=\delta_{ij}$, for all possible $i, j$. We have the following proposition.

\begin{proposition}\label{thetarelation}
Set $\Theta_{\mu}=\mathop{\sum}\limits_{i=1}^{r(\mu)}v_i^{\mu}\otimes u_i^{\mu}\in \mathrm{U}\otimes \mathrm{U}$. Then $\Theta_{\mu}$ does not depend on the choice of the basis $(u_i^{\mu})_i$ and
\begin{align}
(\bbe_i\otimes 1)\Theta_{\mu}+(\bbk_i\otimes \bbe_i)\Theta_{\mu-\alpha_i}&=\Theta_{\mu}(\bbe_i\otimes 1)+\Theta_{\mu-\alpha_i}(\bbk_i^{-1}\otimes \bbe_i),\label{theta1}\\
(1\otimes \bbf_i)\Theta_{\mu}+(\bbf_i\otimes \bbk_i^{-1})\Theta_{\mu-\alpha_i}&=\Theta_{\mu}(1\otimes \bbf_i)+\Theta_{\mu-\alpha_i}(\bbf_i\otimes\bbk_i),\label{theta2}\\
(\bbk_i\otimes\bbk_i)\Theta_{\mu}&=\Theta_{\mu}(\bbk_i\otimes \bbk_i).\label{theta3}
\end{align}
\end{proposition}

\begin{proof} 
It is easy to check $\Theta_{\mu}$ does not depend on the choice of the basis $(u_i^{\mu})_i$ and (\ref{theta3}). For (\ref{theta1}), we have
\begin{align*}
&(\bbe_i\otimes 1)\Theta_{\mu}-\Theta_{\mu}(\bbe_i\otimes 1)=\mathop{\sum}\limits_{j=1}^{r(\mu)}[\bbe_i,v_j^{\mu}]\otimes u_j^{\mu} \\
=&\mathop{\sum}\limits_{j=1}^{r(\mu)}(q_i-q_i^{-1})^{-1}\big((-1)^{|\bbe_i||r_i(v_j^{\mu})|}\bbk_ir_i(v_j^{\mu})-r'_i(v_j^{\mu})\bbk_i^{-1}\big)\otimes u_j^{\mu}\\
=&\mathop{\sum}\limits_{j=1}^{r(\mu)}\mathop{\sum}\limits_{k=1}^{r(\mu-\alpha_i)}(q_i-q_i^{-1})^{-1}\Big((-1)^{|\bbe_i||r_i(v_j^{\mu})|}\bbk_i\big(r_i(v_j^{\mu}),u_k^{\mu-\alpha_i}\big)v_k^{\mu-\alpha_i}-\big(r'_i(v_j^{\mu}),u_k^{\mu-\alpha_i}\big)v_k^{\mu-\alpha_i}\bbk_i^{-1}\Big)\otimes u_j^{\mu}\\
=&\mathop{\sum}\limits_{j=1}^{r(\mu)}\mathop{\sum}\limits_{k=1}^{r(\mu-\alpha_i)}\Big(-(-1)^{|\bbe_i||r_i(v_j^{\mu})|}\bbk_i(\bbf_i,\bbe_i)\big(r_i(v_j^{\mu}),u_k^{\mu-\alpha_i}\big)v_k^{\mu-\alpha_i}+(\bbf_i,\bbe_i)\big(r'_i(v_j^{\mu}),u_k^{\mu-\alpha_i}\big)v_k^{\mu-\alpha_i}\bbk_i^{-1}\Big)\otimes u_j^{\mu}\\
=&\mathop{\sum}\limits_{j=1}^{r(\mu)}\mathop{\sum}\limits_{k=1}^{r(\mu-\alpha_i)}\big(-(-1)^{|\bbe_i||r_i(v_j^{\mu})|}\bbk_i(v_j^{\mu},\bbe_iu_k^{\mu-\alpha_i})v_k^{\mu-\alpha_i}+(v_j^{\mu},u_k^{\mu-\alpha_i}\bbe_i)v_k^{\mu-\alpha_i}\bbk_i^{-1}\big)\otimes u_j^{\mu}\\
=&\mathop{\sum}\limits_{k=1}^{r(\mu-\alpha_i)}-(-1)^{|\bbe_i||r_i(v_j^{\mu})|}\bbk_iv_k^{\mu-\alpha_i}\otimes \bbe_iu_k^{\mu-\alpha_i}+v_k^{\mu-\alpha_i}\bbk_i^{-1}\otimes u_k^{\mu-\alpha_i}\bbe_i\\
=&-(\bbk_i\otimes \bbe_i)\Theta_{\mu-\alpha_i}+\Theta_{\mu-\alpha_i}(\bbk_i^{-1}\otimes \bbe_i).
\end{align*}
Thus, (\ref{theta1}) holds.  Because the proof for Equation (\ref{theta2}) is similar to that for Equation (\ref{theta1}), we omit it here.
\end{proof}

There is an algebra automorphism $\phi$ of $\mathrm{U}_q(\mathfrak{g})\otimes\mathrm{U}_q(\mathfrak{g})$ defined by\begin{align*}
&\phi(\bbk_i\otimes 1)=\bbk_i\otimes 1,\quad \phi(\bbe_i\otimes 1)=\bbe_i\otimes \bbk_i^{-1},\quad \phi(\bbf_i\otimes 1)=\bbf_i\otimes \bbk_i,\\
&\phi(1\otimes \bbk_i)=1\otimes \bbk_i,\quad \phi(1\otimes \bbe_i)=\bbk_i^{-1}\otimes\bbe_i\quad \phi(1\otimes\bbf_i)=\bbk_i\otimes\bbf_i,
\end{align*}
and $\phi$ can be extended to $\mathrm{U}_q(\mathfrak{g})\widehat{\otimes}\mathrm{U}_q(\mathfrak{g})$, which is a completion of the tensor product $\mathrm{U}_q(\mathfrak{g})\otimes \mathrm{U}_q(\mathfrak{g})$. Then the quasi-R-matrix is $\mathop{\sum}\limits_{\mu\geqslant0}\Theta_{\mu}\in\mathrm{U}_q(\mathfrak{g})\widehat{\otimes}\mathrm{U}_q(\mathfrak{g})$\footnote{More properties about quasi-R-matrix in a super setting can be deduction follows \cite[Chapter 4]{Lu}. For example, $\bar{\mathfrak{R}}=\mathfrak{R}^{-1}$, where the automorphism $\bar{~}$ of $\mathrm{U}\widehat{\otimes} \mathrm{U}$  is defined in \cite[Chapter 4]{Lu}.} and it is invertible. Its inverse is denoted by $\mathfrak{R}$. Then, by Proposition \ref{thetarelation}, we have 
\begin{align*}
\mathfrak{R}\Delta(u)=\phi\big(\Delta^{op}(u)\big)\mathfrak{R},\text{ and}\quad  \mathfrak{R}^{op}\Delta^{op}(u)=\phi\big(\Delta(u)\big)\mathfrak{R}^{op}.
\end{align*}
The universal R-matrix can be derived from the quasi-R-matrix,  which is significant because it can induce solutions of the quantum Yang-Baxter equation on any of its modules. This approach is prominent in the study of integrable systems,  knot invariants and so on.  The following proposition is essential for us to construct the explicit central elements,  named Casimir invariants,  which have been used to construct a family of Casimir invariants for quantum groups \cite{DZ}, quantum superalgebras $\mathrm{U}_q(\mathfrak{gl}_{m|n})$ and $\mathrm{U}_q(\mathfrak{osp}_{m|2n})$.  
\subsection{Constructing central elements using quasi-R-matrix}
\begin{proposition}\label{Gamma_M} \cite[Proposition 3.1]{DGL} Given an operator $\Gamma_M\in\mathrm{End}(M)\otimes \mathrm{U}_q(\mathfrak{g})$ satisfying
\begin{equation}
[\Gamma_M,\Delta(u)]=0\quad \text{for all}~ u\in \mathrm{U}_q(\mathfrak{g}),
\end{equation}
the elements 
\begin{equation}\label{C^k_M}
C_M^{(k)}:=\mathrm{Str}_1\big((\zeta\otimes 1)(\bbk_{2\rho}\otimes 1)(\Gamma_M)^k\big)
\end{equation}
are central in $\mathrm{U}_q(\mathfrak{g})$, where  $\mathrm{Str}_1(f\otimes u)=\mathrm{Str}(f)u$ for $f\in\mathrm{End}(M)$ and $u\in\mathrm{U}_q(\mathfrak{g})$.
\end{proposition}
\begin{proof}
We only need to prove $[C_M^{(k)}, \bbk_i]=[C_M^{(k)}, \bbe_i]=[C_M^{(k)}, \bbf_i]=0$  for all $i\in\mathbb{I}$. Assume $(\Gamma_M)^k=\mathop{\sum}\limits_{j}A_j\otimes B_j$, then 
\begin{align*}
0&=\mathrm{Str}_1\big((\bbk_{2\rho}\bbk_i^{-1}\otimes 1)[(\Gamma_M)^k,\Delta(\bbk_i)]\big)\\
&=\mathrm{Str}_1\Big((\bbk_{2\rho}\bbk_i^{-1}\otimes 1)\Big[\mathop{\sum}\limits_{j}A_j\otimes B_j,\bbk_i\otimes \bbk_i\Big]\Big)\\
&=\mathop{\sum}\limits_{j}\mathrm{Str}(\bbk_{2\rho}\bbk_i^{-1}A_j\bbk_i)B_j\bbk_i-\mathop{\sum}\limits_{j}\mathrm{Str}(\bbk_{2\rho}A_j)\bbk_iB_j\\
&=[C_M^{(k)},\bbk_i],
\end{align*}
where the last equation holds by $\mathrm{Str}([x, y])=0$ for all $x,y\in\mathrm{End}(M)$.  And, 
\begin{align*}
0=&\mathrm{Str}_1\big((\bbk_{2\rho}\otimes 1)[(\Gamma_M)^k,\Delta(\bbf_i)]\big)\\
=&\mathrm{Str}_1\Big((\bbk_{2\rho}\otimes 1)\Big[\mathop{\sum}\limits_{j}A_j\otimes B_j,\bbf_i\otimes \bbk_i^{-1}+1\otimes\bbf_i\Big]\Big)\\
=&\mathrm{Str}_1\Big((\bbk_{2\rho}\otimes 1)\mathop{\sum}\limits_{j}\big((-1)^{|B_j||\bbf_i|}A_j\bbf_i \otimes B_j\bbk_i^{-1}+A_j\otimes B_j\bbf_i\\
&-(-1)^{|\bbf_i|(|A_j|+|B_j|)}\bbf_iA_j\otimes\bbk_i^{-1} B_j-(-1)^{|\bbf_i||B_j|}A_j\otimes\bbf_i B_j\big)\Big)\\
=&[C_M^{(k)},\bbf_i],
\end{align*}
where the last equation follows from  $\Big[\mathop{\sum}\limits_{j}A_j\otimes B_j,\bbk_i\otimes \bbk_i\Big]=0$ and $\mathrm{Str}([x, y])=0$ for all $x,y\in\mathrm{End}(M)$. 
\end{proof}

Define by $\zeta\colon\mathrm{U}_q(\mathfrak{g})\rightarrow \mathrm{End}(M)$ the linear representation. Let $P^M_{\eta}\colon M\rightarrow M_{\eta}$ be the projection from $M$ to $M_{\eta}$ and define the following element in $\mathrm{End}(M)\otimes \mathrm{U} _q(\mathfrak{g})$ as
\begin{equation}\label{K_M}
\mathcal{K}_{M}=\mathop{\sum}\limits_{\eta\in\mathrm{wt}(M)}P^M_{\eta}\otimes\bbk_{2\eta}.
\end{equation}
Using the definition of $\phi$, we obtain\begin{equation}
\mathcal{K}_M(\zeta\otimes 1)\big(\phi^2(\Delta(u))\big)=(\zeta\otimes 1)(\Delta(u))\mathcal{K}_M,\quad \forall u\in\mathrm{U}_q(\mathfrak{g}).
\end{equation}
Define $R_M=(\zeta\otimes 1)(\mathfrak{R})$ and $R^{op}_M=(\zeta\otimes 1)(\mathfrak{R}^{op})$, we have 
\begin{align*}
&\mathcal{K}_M\phi(R_M^{op})R_M(\zeta\otimes 1)(\Delta(u))=\mathcal{K}_M(\zeta\otimes 1)\big(\phi(\mathfrak{R}^{op})\mathfrak{R}\Delta(u)\big)\\
=&\mathcal{K}_M(\zeta\otimes 1)\big(\phi^2(\Delta(u))\phi(\mathfrak{R}^{op})\mathfrak{R}\big)=\mathcal{K}_M(\zeta\otimes 1)\big(\phi^2(\Delta(u))\big)\phi(R^{op}_M)R_M \\
=&(\zeta\otimes 1)(\Delta(u)) \mathcal{K}_M\phi(R_M^{op})R_M,      \quad\quad\forall u\in\mathrm{U}_q(\mathfrak{g}).
\end{align*}
If we take 
\begin{equation}\label{Gamma_M}
\Gamma_M=\mathcal{K}_M\phi(R_M^{op})R_M,
\end{equation}
then $[\Gamma_{M},(\zeta\otimes 1)(\Delta(u))]=0$, for all $u\in\mathrm{U}_q(\mathfrak{g})$. 
\begin{example}
Let $\mathrm{U}=\mathrm{U}_q(A(1,0))$ and $\zeta\colon\mathrm{U}\rightarrow \mathrm{End}(M)=\mathrm{End}(L_q(\varepsilon_1))$ be the vector representation. Let $v_1$ be its highest weight vector with weight $\lambda_1$, and let $v_{2}=\bbf_1v_1, v_{3}=\bbf_2\bbf_1v_1$ and $\lambda_{2},\lambda_3$ be the corresponding weights associated with $v_{2},v_3$, respectively. $\{v_1,v_2,v_3\}$ is a basis of $M$. By using of (\ref{skew-pairing}) and (\ref{double}), $\{-(q_i-q_i)^{-1}\bbf_i\}$ and $\{\bbe_i\}$ are two basis-dual basis pairs of $\mathrm{U}^-_{-\alpha_i}$ and $\mathrm{U}^+_{\alpha_i}$ for $i=1,2$ and
\begin{equation*}
\{(q-q^{-1})\bbf_1\bbf_2,(q^{-1}-q)\bbf_2\bbf_1\} \text{ and } \{q\bbe_1\bbe_2-\bbe_2\bbe_1,\bbe_1\bbe_2-q\bbe_2\bbe_1\}
\end{equation*}
is a basis-dual basis pair of $\mathrm{U}^-_{-\alpha_1-\alpha_2}$ and $\mathrm{U}^+_{\alpha_1+\alpha_2}$ with respect to the Drinfeld double. We have $\mathfrak{R}=\overline{\mathop{\sum}\limits_{\mu\geqslant0}\Theta_{\mu}}$, which is a generalization of \cite[Corollary 4.1.3]{Lu}. Then
\begin{equation}\label{R_M}
R_M=(\zeta\otimes 1)(1\otimes 1+\mathop{\sum}\limits_{i=1}^{2}(q_i-q_i^{-1})\bbf_i\otimes \bbe_i-(q^{-1}-q)\bbf_2\bbf_1\otimes (\bbe_1\bbe_2-q^{-1}\bbe_2\bbe_1)-(q-q^{-1})\bbf_1\bbf_2\otimes (q^{-1}\bbe_1\bbe_2-\bbe_2\bbe_1))
\end{equation}
and
\begin{align}\label{R^op_M}
\phi(R^{\mathrm{op}}_M)=&(\zeta\otimes 1)(1\otimes 1+(q^{-1}-q)(\bbe_1\bbe_2-q^{-1}\bbe_2\bbe_1)\bbk_2\bbk_1\otimes \bbk_2^{-1}\bbk_1^{-1}\bbf_2\bbf_1\nonumber\\
+&\mathop{\sum}\limits_{i=1}^{2}(-1)^{\delta_{i2}}(q_i-q_i^{-1})\bbe_i\bbk_i\otimes \bbk_i^{-1}\bbf_i+(q-q^{-1})(q^{-1}\bbe_1\bbe_2-\bbe_2\bbe_1)\bbk_1\bbk_2\otimes \bbk_2^{-1}\bbk_1^{-1}\bbf_1\bbf_2   ).
\end{align}
because $\zeta(\mathrm{U}^-_{-\nu})=0$ if $\nu\neq \alpha_1,\alpha_2,\alpha_1+\alpha_2$.
Substitute (\ref{K_M}), (\ref{R_M}) and (\ref{R^op_M}) into (\ref{Gamma_M}) and (\ref{C^k_M}). As a result,
\begin{align*}
C_M^{(1)}=&\mathrm{Str}_1\big((\zeta\otimes 1)(\bbk_{2\rho}\otimes 1)\mathcal{K}_M\phi(R^{\mathrm{op}}_M)R_M\big)\\
=&\mathop{\sum}\limits_{i=1}^{3}(-1)^{|v_i|}q^{(2\rho,\lambda_i)}\bbk_{2\lambda_i}+\mathop{\sum}\limits_{i=1}^2 (q_i-q_i^{-1})^2(-1)^{|v_i|} q^{(\alpha_i,\lambda_{i+1})+(2\rho,\lambda_i)}\bbk_{2\lambda_{i}}\bbk_i^{-1}\bbf_i\bbe_i\\
+&(q-q^{-1})^2q^{(2\rho,\lambda_1)+(\alpha_1+\alpha_2,\lambda_3)}\bbk_{2\lambda_1}\bbk_2^{-1}\bbk_1^{-1}(\bbf_2\bbf_1-q^{-1}\bbf_1\bbf_2)(\bbe_1\bbe_2-q^{-1}\bbe_2\bbe_1)\\
=&\bbk_2^{-2}+q^{-2}\bbk_1^{-2}\bbk_2^{-2}-q^{-2}\bbk_1^{-2}\bbk_2^{-4}+(q-q^{-1})^2(q^{-1}\bbk_1^{-1}\bbk_2^{-2}\bbf_1\bbe_1+q^{-1}\bbk_1^{-2}\bbk_2^{-3}\bbf_2\bbe_2)\\
+&(q-q^{-1})^2q\bbk_1^{-1}\bbk_2^{-3}(\bbf_2\bbf_1-q^{-1}\bbf_1\bbf_2)(\bbe_1\bbe_2-q^{-1}\bbe_2\bbe_1),
\end{align*}
by using
\begin{align*}
&2\rho=\alpha_1-\alpha_2-(\alpha_1+\alpha_2)=-2\alpha_2;\\
&\lambda_1=\varepsilon_1=-\varepsilon_2+\delta_1=-\alpha_2;\\
&\lambda_2=\varepsilon_2=-\varepsilon_1+\delta_1=-\alpha_1-\alpha_2;\\
&\lambda_3=\delta_1=-\varepsilon_1-\varepsilon_2+2\delta_1=-\alpha_1-2\alpha_2.
\end{align*}
There is a $k$-algebra anti-automorphism $\tau$ of $\mathrm{U}$ defined by $\tau(\bbe_i)=\bbf_i,~\tau(\bbf_i)=\bbe_i,~\tau(\bbk^{\pm1}_i)=\bbk^{\pm1}_i$ for $i=1,2$. It is obvious that $C^{(1)}_M$ commutes with $\bbk_1$ and $\bbk_2$. One can check directly that $C^{(1)}_M$ commutes with $\bbe_1$ and $\bbe_2$. Because $C^{(1)}_M$ is $\tau$-invariant, $C^{(1)}_M$ commutes with $\bbf_1$ and $\bbf_2$. Therefore, $C^{(1)}_M\in\mathcal{Z}(\mathrm{U}_q(\mathfrak{g}))$.
\end{example}
\subsection{Proof of Theorem B}
In the previous subsection, we used the quasi-R-matrix to construct an explicit $\Gamma_M$ associated with a finite-dimensional $\mathrm{U}_q(\mathfrak{g})$-module $M$ satisfying Proposition \ref{Gamma_M}. Thus, we obtained a family of central elements of $\mathrm{U}_q(\mathfrak{g})$. Now, we are ready to prove Theorem $\mathrm{B}$. For convenience, we simplify $C_{L_q(\lambda)}$ for $C^{(1)}_{L_q(\lambda)}$.
\begin{theorem}
$\{C_{L_q(\lambda)} ~|~\lambda\in \Lambda\cap\frac{1}{2}\mathbb{Z}\Phi$ and $L(\lambda)$ finite-dimensional \} is a basis of $\mathcal{Z}(\mathrm{U}_q(\mathfrak{g}))$ if $\mathfrak{g}\neq A(1,1)$.
\end{theorem}
\begin{proof} Applying the $\mathcal{HC}$ to $C_{L_q(\lambda)^*}$ results in 
\begin{align*}
&\mathcal{HC}\left(C_{L_q(\lambda)^*}\right)=\mathcal{HC}\bigg(\mathrm{Str}_1\Big(\big(\zeta(\bbk_{2\rho})\otimes 1\big)\Gamma_{L_q(\lambda)^*}\Big)\bigg)=\gamma_{-\rho}\circ\pi\bigg(\mathrm{Str}_1\Big(\big(\zeta(\bbk_{2\rho})\otimes 1 \big)\mathcal{K}_{L_q(\lambda)^*}\Big)\bigg)\\
=&\mathop{\sum}\limits_{\eta\in\mathrm{wt}(L_q(\lambda)^*)}\gamma_{-\rho}\left(\mathrm{Str}(q^{(2\rho,\eta)}P^{L_q(\lambda)^*}_{\eta})\bbk_{2\eta}\right)=\mathop{\sum}\limits_{
\mu}\mathrm{sdim}L_q(\lambda)_{\mu}\bbk_{-2\mu}=\mathcal{HC}\left(z_{L_q(\lambda)}\right).
\end{align*}

According to Theorem A (i.e., the $\mathcal{HC}=\gamma_{-\rho}\circ\pi$ is an algebra isomorphism),  $z_{L_q(\lambda)}=C_{L_q(\lambda)^*}$. Furthermore, $\left\{\left.[L_q(\lambda)]\right |\lambda\in\Lambda\cap \frac{1}{2}\mathbb{Z}\Phi\text{ and } L_q(\lambda) \text{ is finite-dimensional}\right\}$ is a basis of $K_{\mathrm{ev}}(\mathrm{U}_q(\mathfrak{g}))$. Hence, $\left\{\left.C_{L_q(\lambda)^*} \right|\lambda\in \Lambda\cap\frac{1}{2}\mathbb{Z}\Phi ~\text{and}~ L_q(\lambda) ~\text{is finite-dimensional}\right\}$ is a basis of $\mathcal{Z}(\mathrm{U}_q(\mathfrak{g}))$. So is $\left\{\left.C_{L_q(\lambda)} \right|\lambda\in \Lambda\cap\frac{1}{2}\mathbb{Z}\Phi ~\text{and}~ L(\lambda) ~\text{is finite-dimensional}\right\}$.
\end{proof}
\begin{remark}
One can define a new quantum superalgebra $\tilde{\mathrm{U}}=\tilde{\mathrm{U}}_q(\mathfrak{g})$ associated with a simple Lie superalgebra $\mathfrak{g}$, except for $A(1,1)$, by replacing the cartan subalgebra of quantum superalgebra $\mathfrak{g}$ with the group ring $k\Gamma$ if $\mathbb{Z}\Phi\subseteq\Gamma\subseteq \Lambda$, $W\Gamma=\Gamma$ and $q^{(\gamma,\lambda)}\in k$ for all $\gamma\in\Gamma, \lambda\in\Lambda$.  Using the same procedure, we can establish the Harish-Chandra isomorphism between $\mathcal{Z}(\tilde{\mathrm{U}})$ and $(\tilde{\mathrm{U}}^0_{\mathrm{ev}})^W_{\mathrm{sup}}$,  where
\begin{align*}
(\tilde{\mathrm{U}}^0_{\mathrm{ev}})^{W}_{\mathrm{sup}}=\Bigg\{\mathop{\sum}\limits_{\mu\in 2\Lambda\cap \Gamma} a_{\mu}\bbk_{\mu}\in\mathrm{U}^0\Bigg|a_{w\mu}=a_{\mu}, ~\forall w\in W; \mathop{\sum}\limits_{\mu \in A^{\alpha}_{\nu}}a_{\mu}=0,~\forall\alpha\in{\Phi}_{\mathrm{iso}}~\text{with} ~(\nu,\alpha)\neq0\Bigg\}.
\end{align*}
In particular, $K(\mathfrak{g})\cong K_{\Lambda}(\tilde{\mathrm{U}})$, where $K_{\Lambda}(\tilde{\mathrm{U}})$ is the subring of $K(\tilde{\mathrm{U}})$ generated by all objects in $\tilde{\mathrm{U}}$-{\bf mod} whose weights are contained in $\Lambda$ if  $\Gamma=\Lambda$.
\end{remark}
\begin{remark}
Our approach to obtaining the Harish-Chandra type theorem for quantum superalgebras of type $\mathrm{A}$-$\mathrm{G}$ takes advantage of the Rosso form, which cannot be applied to quantum queer superalgebra $\mathrm{U}_q(\mathfrak{q}_n)$  \cite{Olsh} or quantum perplectic superalgebra $\mathrm{U}_q(\mathfrak{p}_n)$  \cite{AGG}.  One immediate problem is to establish the Harish-Chandra type theorems for these quantum superalgebras.  We hope to return to these questions in future.
\end{remark}
\appendix
\section{Dynkin diagrams in distinguished root systems}\label{appendix}
The Dynkin diagrams in the distinguished
root systems of a simple basic Lie superalgebra of type A-G are listed below, where $r$ is the number of nodes
and $s$ is the element of $\tau$.
Note that the form of Dynkin diagrams in the distinguished
root systems is quite uniform in the literature.

\medskip
{\bf$\boxed{A(m,n) \text{ case}:}$}
Let $\mathfrak{h}^*$ be a vector space spanned by $\{\varepsilon_i-\varepsilon_{i+1},\varepsilon_{m+1}-\delta_1 ,\delta_j-\delta_{j+1}| 1\leqslant i\leqslant m, 1\leqslant j\leqslant n \}$ satisfies 
$$(\varepsilon_1+\ldots+\varepsilon_{m+1})-(\delta_1+\ldots+\delta_{n+1})=0.$$ 
We equip the dual $\mathfrak{h}^*$ with a bilinear form $(\cdot,\cdot)$ such that  
\begin{align*}
(\varepsilon_i, \varepsilon_j)=\delta_{ij},\quad (\varepsilon_i,\delta_j)=(\delta_j, \varepsilon_i)=0, \quad (\delta_i, \delta_j)=-\delta_{ij}\quad \text{for all possible}~i,j.
\end{align*}
The distinguished fundamental system $\Pi=\{\alpha_1,\ \ldots,\ \alpha_{m+n+1}\}$ is given by 
$$\{\varepsilon_1-\varepsilon_2, \  \ldots,\ \varepsilon_m,-\varepsilon_{m+1}, \  \varepsilon_{m+1}-\delta_1, \  \delta_1-\delta_2,\ \ldots,\ \delta_{n}-\delta_{n+1}\}.$$
The Dynkin diagram associated with $\Pi$ is depicted as follows:

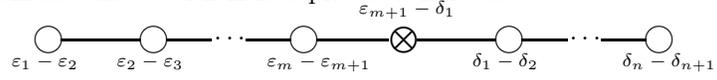
\begin{center}
\setlength{\unitlength}{0.35mm}
\begin{picture}(300, 20)(-5, 0)
\put(30, 10){\circle{10}}
\put(16, 0){\scriptsize $\varepsilon_1-\varepsilon_2$}
\put(35, 10){\line(1, 0){30}}
\put(70, 10){\circle{10}}
\put(56,0){\scriptsize $\varepsilon_2-\varepsilon_3$}
\put(75,10){\line(1,0){17}}
\put(93,10){$\ldots$}
\put(105,10){\line(1,0){17}}
\put(127,10){\circle{10}}
\put(113, 0){\scriptsize $\varepsilon_m-\varepsilon_{m+1}$}
\put(132,10){\line(1, 0){27}}
\put(165,10){\makebox(0,0)[c]{$\bigotimes$}}
\put(148,20){\scriptsize $\varepsilon_{m+1}-\delta_1$}
\put(170,10){\line(1, 0){30}}
\put(205,10){\circle{10}}
\put(191,0){\scriptsize $\delta_{1}-\delta_2$}
\put(210, 10){\line(1, 0){17}}
\put(228,10){$\ldots$}
\put(240, 10){\line(1, 0){17}}
\put(262, 10){\circle{10}}
\put(248,0){\scriptsize $\delta_{n}-\delta_{n+1}$.}
\end{picture}
\end{center}
In this case $r=m+n+1$, $s=m+1$. The distinguished positive system $\Phi^+=\Phi_{\bar{0}}^+\cup\Phi_{\bar{1}}^+$ corresponding to the distinguished Borel subalgebra for $A(m,n)$ is 
\begin{align*}
\{\varepsilon_i-\varepsilon_j, \delta_k-\delta_l| 1\leqslant i<j\leqslant m+1, 1\leqslant k<l\leqslant n+1\}
\cup\{\varepsilon_i-\delta_j| 1\leqslant i\leqslant m+1, 1\leqslant j\leqslant n+1\}.
\end{align*}
The Weyl group $W\cong\mathfrak{S}_{m+1}\times\mathfrak{S}_{n+1}$.

{\bf$\boxed{B(m,n) \text{ case:}}$}
Let $\mathfrak{h}^*$ be a vector space with basis $\{\varepsilon_i, \delta_j| 1\leqslant i\leqslant m, 1\leqslant j\leqslant n\}$. We equip the dual $\mathfrak{h}^*$ with a bilinear form $(\cdot,\cdot)$ such that  
\begin{align*}
(\varepsilon_i, \varepsilon_j)=\delta_{ij},\quad (\varepsilon_i,\delta_j)=(\delta_j, \varepsilon_i)=0, \quad (\delta_i, \delta_j)=-\delta_{ij}\quad\text{for all possible}~i,j.
\end{align*}
The distinguished fundamental system $\Pi=\{\alpha_1,\ \ldots,\ \alpha_{m+n}\}$ is given by 
$$\{\delta_1-\delta_2, \ \ldots,\ \delta_{n-1}-\delta_{n},\  \delta_{n}-\varepsilon_1,  \varepsilon_{1}-\varepsilon_2,\ \ldots,\ \varepsilon_{m-1}-\varepsilon_{m}, \ \varepsilon_m\}.$$
The Dynkin diagram associated with $\Pi$ is depicted as follows:

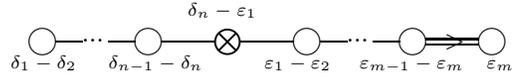
\begin{picture}(400, 30)(0, 0)
\put(110, 10){\circle{10}}
\put(98, 0){\scriptsize $\delta_1-\delta_2$}
\put(115, 10){\line(1, 0){10}}
\put(125,
10){...} \put(135, 10){\line(1, 0){10}}
\put(150, 10){\circle{10}}
\put(135, 0){\scriptsize $\delta_{n-1}-\delta_n$}
\put(155, 10){\line(1, 0){20}}
\put(180,10){\makebox(0,0)[c]{$\bigotimes$}}
\put(165, 20){\scriptsize $\delta_{n}-\varepsilon_1$}
\put(185, 10){\line(1, 0){20}}
\put(210, 10){\circle{10}}
\put(194, 0){\scriptsize $\varepsilon_1-\varepsilon_2$}
\put(215, 10){\line(1, 0){10}}
\put(225, 10){...}
\put(235, 10){\line(1, 0){10}}
\put(250, 10){\circle{10}}
\put(230, 0){\scriptsize $\varepsilon_{m-1}-\varepsilon_m$}
\put(255, 9){\line(1, 0){20}}
\put(255, 11){\line(1, 0){20}}
\put(262, 7.5){$>$}
\put(280, 10){\circle{10}}
\put(278, 0){\scriptsize $\varepsilon_m$.}
\end{picture}

In this case $r=m+n$,  $s=n+1$. The distinguished positive system $\Phi^+=\Phi^+_{\bar{0}}\cup\Phi_{\bar{1}}^+$ corresponding to the distinguished Borel subalgebra is 
\begin{align*}
\{\delta_{i}\pm\delta_j,\  2\delta_p,\  \varepsilon_k\pm\varepsilon_l, \  \varepsilon_q\}\cup\{\delta_p\pm\varepsilon_q,\  \delta_p\},
\end{align*}
where $1\leqslant i<j\leqslant n,  1\leqslant k<l\leqslant m,   1\leqslant p\leqslant n,  1\leqslant q\leqslant m$. The Weyl group $W\cong(\mathfrak{S}_{n}\ltimes\mathbb{Z}_2^n)\times(\mathfrak{S}_{m}\ltimes\mathbb{Z}_2^m)$.

{\bf $\boxed{B(0,n) \text{ case:}}$}
Let $\mathfrak{h}^*$ be a vector space with basis $\{\delta_i|1\leqslant i\leqslant n \}$. We equip the dual $\mathfrak{h}^*$ with a bilinear form $(\cdot,\cdot)$ such that  
\begin{align*}
(\delta_i, \delta_j)=-\delta_{ij}\quad\text{for all possible}~i,j.
\end{align*}
The distinguished fundamental system $\Pi=\{\alpha_1,\ \ldots,\ \alpha_{n}\}$ is given by 
$$\{\delta_1-\delta_2, \ \ldots,\ \delta_{n-1}-\delta_{n},\  \delta_{n}\}.$$
The Dynkin diagram associated with $\Pi$ is depicted as follows:

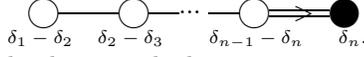
\begin{picture}(400, 30)(0, 0)
\setlength{\unitlength}{0.4mm}
\put(110, 10){\circle{10}}
\put(98, 0){\scriptsize $\delta_1-\delta_2$}
\put(115, 10){\line(1, 0){20}}
\put(140,10){\circle{10}}
\put(128, 0){\scriptsize $\delta_2-\delta_3$}
\put(145, 10){\line(1, 0){10}}
\put(155, 10){...}
\put(165, 10){\line(1, 0){10}}
\put(180, 10){\circle{10}}
\put(165, 0){\scriptsize $\delta_{n-1}-\delta_n$}
\put(185, 9){\line(1, 0){20}}
\put(185, 11){\line(1, 0){20}}
\put(193, 7.7){$>$}
\put(210, 10){\circle*{10}}
\put(208, 0){\scriptsize $\delta_n$.}
\end{picture}

 In this case, $r=s=n$. The distinguished positive system $\Phi^+=\Phi^+_{\bar{0}}\cup\Phi_{\bar{1}}^+$ corresponding to the distinguished Borel subalgebra is 
\begin{align*}
\{\delta_{i}\pm\delta_j,\  2\delta_p|1\leqslant i<j\leqslant n, 1\leqslant p\leqslant n\}\cup\{\delta_p|1\leqslant p\leqslant n\}.
\end{align*}
The Weyl group $W\cong(\mathfrak{S}_{n}\ltimes\mathbb{Z}_2^n)$.

{\bf $\boxed{C(n+1) \text{ case:}} $}
Let $\mathfrak{h}^*$ be a vector space with basis $\{\varepsilon,\delta_i|1\leqslant i\leqslant n \}$. 
We equip the dual $\mathfrak{h}^*$ with a bilinear form $(\cdot,\cdot)$ such that  
\begin{align*}
(\varepsilon,\varepsilon)=1,\quad(\varepsilon,\delta_i)=(\delta_i, \varepsilon)=0,\quad(\delta_i, \delta_j)=-\delta_{ij}\quad\text{for all possible}~i,j.
\end{align*}
The distinguished fundamental system $\Pi=\{\alpha_1,\ \ldots,\ \alpha_{n+1}\}$ is given by 
$$\{\varepsilon-\delta_1,\delta_1-\delta_2, \ \ldots,\ \delta_{n-1}-\delta_{n},\  2\delta_{n}\}.$$
The Dynkin diagram associated with $\Pi$ is depicted as follows:

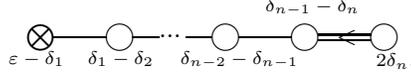
\begin{picture}(400, 30)(0, 0)
\put(110,10){\makebox(0,0)[c]{$\bigotimes$}}
\put(98, 0){\scriptsize $\varepsilon-\delta_1$}
\put(115, 10){\line(1, 0){20}}
\put(140, 10){\circle{10}}
\put(128, 0){\scriptsize $\delta_1-\delta_2$}
\put(145, 10){\line(1, 0){10}}
\put(155, 10){...}
\put(165, 10){\line(1, 0){10}}
\put(180, 10){\circle{10}}
\put(163, 0){\scriptsize $\delta_{n-2}-\delta_{n-1}$}
\put(185, 10){\line(1, 0){20}}
\put(210, 10){\circle{10}}
\put(195, 20){\scriptsize $\delta_{n-1}-\delta_n$}
\put(215, 9){\line(1, 0){20}}
\put(215, 11){\line(1, 0){20}}
\put(222, 7.5){$<$}
\put(240, 10){\circle{10}}
\put(237, -1){\scriptsize $2\delta_n$.}
\end{picture}

In this case $r=n+1,s=1$. The distinguished positive system $\Phi^+=\Phi^+_{\bar{0}}\cup\Phi_{\bar{1}}^+$ corresponding to the distinguished Borel subalgebra is 
\begin{align*}
\{\delta_{i}\pm\delta_j,\  2\delta_p|1\leqslant i<j\leqslant n, 1\leqslant p\leqslant n\}\cup\{\varepsilon\pm\delta_p|1\leqslant p\leqslant n\}.
\end{align*}
The Weyl group $W\cong(\mathfrak{S}_{n}\ltimes\mathbb{Z}_2^n)$.

{\bf $\boxed{D(m,n) \text{ case:}} $}
Let $\mathfrak{h}^*$ be a vector space with basis $\{\varepsilon_i, \delta_j| 1\leqslant i\leqslant m, 1\leqslant j\leqslant n\}$. We equip the dual $\mathfrak{h}^*$ with a bilinear form $(\cdot,\cdot)$ such that  
\begin{align*}
(\varepsilon_i, \varepsilon_j)=\delta_{ij},\quad (\varepsilon_i,\delta_j)=(\delta_j, \varepsilon_i)=0, \quad (\delta_i, \delta_j)=-\delta_{ij}\quad\text{for all possible}~i,j.
\end{align*}
The distinguished fundamental system $\Pi=\{\alpha_1,\ \ldots,\ \alpha_{m+n}\}$ is given by 
$$\{\delta_1-\delta_2, \ \ldots,\ \delta_{n-1}-\delta_{n},\  \delta_{n}-\varepsilon_1,  \varepsilon_{1}-\varepsilon_2,\ \ldots,\ \varepsilon_{m-1}-\varepsilon_{m}, \ \varepsilon_{m-1}+\varepsilon_m\}.$$
The Dynkin diagram associated with $\Pi$ is depicted as follows:
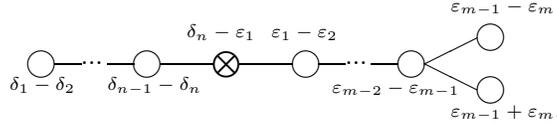
\begin{center}
\begin{picture}(400, 30)(0, 0)
\put(110, 10){\circle{10}}
\put(98, 0){\scriptsize $\delta_1-\delta_2$}
\put(115, 10){\line(1, 0){10}}
\put(125,
10){...}
\put(135, 10){\line(1, 0){10}}
\put(150, 10){\circle{10}}
\put(135, 0){\scriptsize $\delta_{n-1}-\delta_n$}
\put(155, 10){\line(1, 0){20}}
\put(180,10){\makebox(0,0)[c]{$\bigotimes$}}
\put(165, 20){\scriptsize $\delta_{n}-\varepsilon_1$}
\put(185, 10){\line(1, 0){20}}
\put(210, 10){\circle{10}}
\put(197, 20){\scriptsize $\varepsilon_1-\varepsilon_2$}
\put(215, 10){\line(1, 0){10}}
\put(225, 10){...}
\put(235, 10){\line(1, 0){10}}
\put(250, 10){\circle{10}}
\put(220, -1){\scriptsize $\varepsilon_{m-2}-\varepsilon_{m-1}$}
\put(255, 10){\line(2, 1){20}}
\put(255, 10){\line(2, -1){20}}
\put(280, 20){\circle{10}}
\put(265, 30){\scriptsize $\varepsilon_{m-1}-\varepsilon_m$}
\put(280, 0){\circle{10}}
\put(265, -10){\scriptsize $\varepsilon_{m-1}+\varepsilon_m$.}
\end{picture}
\end{center}
~\\
In this case $r=m+n$,  $s=n+1$. The distinguished positive system $\Phi^+=\Phi^+_{\bar{0}}\cap\Phi_{\bar{1}}^+$ corresponding to the distinguished Borel subalgebra is 
\begin{align*}
\{\delta_{i}\pm\delta_j,\  2\delta_p,\  \varepsilon_k\pm\varepsilon_l, \}\cup\{\delta_p\pm\varepsilon_q\},
\end{align*}
where $1\leqslant i<j\leqslant n,  1\leqslant k<l\leqslant m,   1\leqslant p\leqslant n,  1\leqslant q\leqslant m$. The Weyl group $W\cong(\mathfrak{S}_{n}\ltimes\mathbb{Z}_2^n)\times(\mathfrak{S}_{m}\ltimes\mathbb{Z}_2^{m-1} )$.
\\

{\bf$\boxed{D(2,1;\alpha) \text{ case}:}$}
Let $\mathfrak{h}^*$ be a vector space with basis  $\{\varepsilon_1, \varepsilon_2,\varepsilon_3\}$. We equip the dual $\mathfrak{h}^*$ with a bilinear form $(\cdot,\cdot)$ with 
$$(\varepsilon_1, \varepsilon_1)=-(1+\alpha),\quad (\varepsilon_2, \varepsilon_2)=1, \quad (\varepsilon_3, \varepsilon_3)=\alpha \quad\text{and}\quad (\varepsilon_i,\varepsilon_j)=0\quad\text{for all}~\quad 1\leqslant i\neq j\leqslant 3.$$
The distinguished fundamental system 
$$\Pi=\{\alpha_1=\varepsilon_1+\varepsilon_2+\varepsilon_3,  \alpha_2=-2\varepsilon_2,  \alpha_3=-2\varepsilon_3\}.$$
The Dynkin diagram associated with $\Pi$ is depicted as follows:
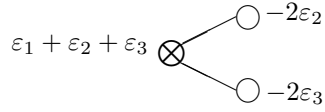
\begin{center}
	\setlength{\unitlength}{0.16in}
	 \begin{picture}(4,6)
	\put(4,1.8){\makebox(0,0)[c]{$\bigcirc$}}
	\put(4,4.2){\makebox(0,0)[c]{$\bigcirc$}}
	\put(1.5,3){\makebox(0,0)[c]{$\bigotimes$}}
	\put(3.6,1.8){\line(-2,1){1.8}}
	\put(3.6,4.2){\line(-2,-1){1.8}}
	\put(5.5,4.3){\makebox(0,0)[c]{$-2\varepsilon_2$}}
	\put(5.5,1.7){\makebox(0,0)[c]{ $-2\varepsilon_3$.}}
	\put(-1.5,3.3){\makebox(0,0)[c]{$\varepsilon_1+\varepsilon_2+\varepsilon_3$}}
	\end{picture}
\end{center}
In this case $r=3$, $s=1$. The distinguished positive system $\Phi^+=\Phi^+_{\bar{0}}\cap\Phi_{\bar{1}}^+$ corresponding to the distinguished Borel subalgebra is
\begin{align*}
\Phi^+_{\bar 0}=\{2\varepsilon_1,-2\varepsilon_2,-2\varepsilon_3\},\quad\Phi^+_{\bar 1}=\{\varepsilon_1\pm\varepsilon_2\pm\varepsilon_3\}.
\end{align*}
The Weyl group $W\cong\mathbb{Z}_2^3$. 

{\bf$\boxed{F(4) \text{ case}:}$}
Let $\mathfrak{h}^*$ be a vector space with basis $\{\delta, \varepsilon_1, \varepsilon_2, \varepsilon_3\}$.We equip the dual $\mathfrak{h}^*$ with a bilinear form $(\cdot,\cdot)$ such that  
$$(\delta, \delta)=-3,\quad (\varepsilon_i,\delta)=(\delta, \varepsilon_i)=0, \quad (\varepsilon_i, \varepsilon_j)=\delta_{ij}\quad\text{for all}~i .$$
The distinguished fundamental system 
$$\Pi=\left\{\alpha_1=\frac{1}{2}(\delta-\varepsilon_1-\varepsilon_2-\varepsilon_3), \quad\alpha_2=\varepsilon_3, \quad \alpha_3=\varepsilon_2-\varepsilon_3, \quad \alpha_4=\varepsilon_1-\varepsilon_2\right\}.$$
The Dynkin diagram associated with $\Pi$ is depicted as follows:

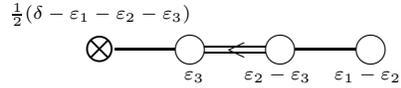
\setlength{\unitlength}{0.40mm}
\begin{picture}(400, 30)(0, 0)
\put(140,10){\makebox(0,0)[c]{$\bigotimes$}}
\put(110, 20){\scriptsize $\frac{1}{2}(\delta-\varepsilon_1-\varepsilon_2-\varepsilon_3)$}
\put(145, 10){\line(1, 0){20}}
\put(170, 10){\circle{10}}
\put(168, 0){\scriptsize $\varepsilon_3$}
\put(175, 9){\line(1, 0){20}}
\put(175, 11){\line(1, 0){20}}
\put(182, 7.5){$<$}
\put(200, 10){\circle{10}}
\put(188, 0){\scriptsize $\varepsilon_2-\varepsilon_3$}
\put(205, 10){\line(1, 0){20}}
\put(230, 10){\circle{10}}
\put(218, 0){\scriptsize $\varepsilon_1-\varepsilon_2$.}
\end{picture}

In this case $r=4$, $s=1$. The distinguished positive system $\Phi^+=\Phi^+_{\bar{0}}\cap\Phi_{\bar{1}}^+$ corresponding to the distinguished Borel subalgebra is 
\begin{align*}
\{ \delta,\ \varepsilon_p,\ \varepsilon_i\pm\varepsilon_j |1\leqslant i<j\leqslant 3,1\leqslant p\leqslant 3\}\cup\left\{\frac{1}{2}(\delta\pm\varepsilon_1\pm\varepsilon_2\pm\varepsilon_3)\right\},
\end{align*}
The Weyl group $W=\mathbb{Z}_2\times (\mathfrak{S}_3\ltimes \mathbb{Z}_2^3)$.

{\bf$\boxed{G(3) \text{ case}:}$}
Let $\mathfrak{h}^*$ be a vector space with basis $\{\delta, \varepsilon_1, \varepsilon_2\}$ and $\varepsilon_3=-\varepsilon_1-\varepsilon_2$. We equip the dual $\mathfrak{h}^*$ with a bilinear form $(\cdot,\cdot)$ such that  
$$(\delta, \delta)=-(\varepsilon_i, \varepsilon_i)=-2,\quad (\varepsilon_i,\delta)=(\delta, \varepsilon_i)=0, \quad (\varepsilon_i, \varepsilon_j)=-1,\quad\text{for all}~1\leqslant i\ne j\leqslant 3.$$
The distinguished fundamental system 
$$\Pi=\{\alpha_1=\delta+\varepsilon_3, \alpha_2=\varepsilon_1, \alpha_3=\varepsilon_2-\varepsilon_1\}.$$
The Dynkin diagram associated with $\Pi$ is depicted as follows:

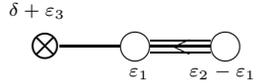
\setlength{\unitlength}{0.40mm}
\begin{picture}(400, 30)(0, 0)
\put(150,10){\makebox(0,0)[c]{$\bigotimes$}}
\put(138, 20){\scriptsize $\delta+\varepsilon_3$}
\put(155, 10){\line(1, 0){20}}
\put(180, 10){\circle{10}}
\put(178, 0){\scriptsize $\varepsilon_1$}
\put(192, 7.5){$<$}
\put(185, 12){\line(1, 0){20}}
\put(185, 10){\line(1, 0){20}}
\put(185, 8){\line(1, 0){20}}
\put(210, 10){\circle{10}}
\put(198, 0){\scriptsize $\varepsilon_2-\varepsilon_1$.}
\end{picture}

%
%
In this case $r=3$, $s=1$. The distinguished positive system $\Phi^+=\Phi^+_{\bar{0}}\cap\Phi_{\bar{1}}^+$ corresponding to the distinguished Borel subalgebra is 
\begin{align*}
\{ 2\delta,\ \varepsilon_1,\ \varepsilon_2,\ \varepsilon_2\pm\varepsilon_1,\ \varepsilon_1-\varepsilon_3,\ \varepsilon_2-\varepsilon_3 \}\cup\{\delta,\ \delta\pm\varepsilon_i|i=1,2,3\},
\end{align*}
The Weyl group $W=\mathbb{Z}_2\times D_6$, where $D_6$ is the dihedral group of order $12$. 
\section{Explicit description of the rings $J_{\mathrm{ev}}(\mathfrak{g})$}\label{appendixB}
Now we give the explicit description of the rings $J_{\mathrm{ev}}(\mathfrak{g})$ for quantum superalgebras, which is inspired by Sergeev and Veselov's description for Lie superalgebras  \cite[Section 7, 8]{SerVes}. Let $x_i=\bbk_{-\varepsilon_i/2}$ and $y_j=\bbk_{-\delta_j/2}$  formally. First we need to review the rings $J(\mathfrak{g})$ for $\mathfrak{g}$ is of type $A$. Let
\begin{align*}
P_0=\bigg\{ \mathop{\sum}\limits_{i=1}^{m+1}a_i\varepsilon_i+\mathop{\sum}\limits_{j=1}^{n+1}b_j\delta_j\bigg | a_i,b_j\in\mathbb{C}~\text{and}~a_i-a_{i+1},b_j-b_{j+1}\in\mathbb{Z}, ~ \forall i,j      \bigg\}\bigg/\mathbb{C}\gamma
\end{align*}
be the weights of $\mathfrak{sl}_{m+1|n+1}$, where $\gamma=\varepsilon_1+\cdots+\varepsilon_{m+1}-\delta_1-\cdots-\delta_{n+1}$ and $x_i=e^{\varepsilon_i}, y_j=e^{\delta_j}$ for all possible $i,j$ be the elements of the group ring of $\mathbb{C}[P_0]$. For convenience, we set $\mathbb{C}[x^{\pm},y^{\pm}]=\mathbb{C}[x_1^{\pm1},\cdots,x_{m+1}^{\pm1},y_1^{\pm1},\cdots,y_{n+1}^{\pm1}],~ \mathbb{Z}[x^{\pm},y^{\pm}]=\mathbb{Z}[x_1^{\pm1},\cdots,x_{m+1}^{\pm1},y_1^{\pm1},\cdots,y_{n+1}^{\pm1}]$ and then for $(m,n)\neq (1,1)$
\begin{align*}
J(\mathfrak{sl}_{m+1|n+1})&=\left\{  f\in\mathbb{Z}[P_0]^W\left | y_j\frac{\partial f}{\partial y_j}+x_i\frac{\partial f}{\partial x_i}\in(x_i-y_j)    \right.\right\}\\
&=\mathop{\bigoplus}\limits_{a\in\mathbb{C}/\mathbb{Z}}J(\mathfrak{sl}_{m+1|n+1})_{a},
\end{align*}
where
\begin{align*}
J(\mathfrak{sl}_{m+1|n+1})_{a}=(x_1\cdots x_{m+1})^a\mathop{\prod}\limits_{i,p}\left(1-\frac{x_i}{y_p}\right)\mathbb{Z}[x^{\pm1},y^{\pm1}]_0^{\mathfrak{S}_{m+1}\times \mathfrak{S}_{n+1}}
\end{align*}
if $a\notin \mathbb{Z}$;
\begin{align*}
J(\mathfrak{sl}_{m+1|n+1})_0=\left\{f\in\mathbb{Z}[x^{\pm1},y^{\pm1}]_0^{\mathfrak{S}_{m+1}\times \mathfrak{S}_{n+1}}\left| x_i\frac{\partial f}{\partial x_i}+y_j\frac{\partial f}{\partial y_j}\in (x_i-y_j) \right.\right\}
\end{align*}
and $\mathbb{Z}[x^{\pm1},y^{\pm1}]_0^{\mathfrak{S}_{m+1}\times \mathfrak{S}_{n+1}}$ is the quotient of the ring $\mathbb{Z}[x^{\pm1},y^{\pm1}]^{\mathfrak{S}_{m+1}\times \mathfrak{S}_{n+1}}$ by the ideal generated by $x_1\cdots x_{m+1}-y_1\cdots y_{n+1}$.

$J(A(n,n))=\mathop{\oplus}\limits_{i=0}^{n}J(A(n,n))_i$ for $n\neq 1$, where for $i\neq 0$
\begin{align*}
J(A(n,n))_i=\bigg\{ f=(x_1\cdots x_{n+1})^{\frac{i}{n+1}}\mathop{\prod}\limits_{j,p}^{n+1}\left(1-\frac{x_j}{y_p}\right)g\bigg | g\in\mathbb{Z}[x^{\pm1},y^{\pm1}]_0^{\mathfrak{S}_{n+1}\times \mathfrak{S}_{n+1}},~ \mathrm{deg}~g=-i    \bigg\}
\end{align*}
and $J(A(n,n))_0$ is the subring of $J(\mathfrak{sl}_{n+1|n+1})_0$ consisting of elements of degree $0$.

$J(A(1,1))=\{c+(u-v)^2g |c\in\mathbb{Z}, g\in\mathbb{Z}[u,v] \}$ where $u=\left(\frac{x_1}{x_2}\right)^{\frac{1}{2}}+\left(\frac{x_2}{x_1}\right)^{\frac{1}{2}},~ v=\left(\frac{y_1}{y_2}\right)^{\frac{1}{2}}+\left(\frac{y_2}{y_1}\right)^{\frac{1}{2}}$.

{\bf$\boxed{A(m,n), m\ne n \text{ case:}}$} Define

\begin{align*}
J^{m|n}=\left\{ f\in \mathbb{C}[x^{\pm1}, y^{\pm1}]^{\mathfrak{S}_{m+1}\times\mathfrak{S}_{n+1}}\left|x_i\frac{\partial f}{\partial x_i}+ y_j\frac{\partial f}{\partial y_j}\in(x_i-y_j)\right. \right\}
\end{align*}
and
\begin{align*}
J^{m|n}_{k}=\left\{\left. f\in J^{m|n}\right|\mathrm{deg}~f=k   \right\}.
\end{align*}
Thus, $J^{m|n}=\bigoplus\limits_{k\in\mathbb{Z}}J^{m|n}_{k}$.

For any element $\lambda\in\mathfrak{h}^*$, we write $\lambda=\mathop{\sum}\limits_{i=1}^{m+1}a_i\varepsilon_i+\mathop{\sum}\limits_{j=1}^{n+1}b_j\delta_j$, then we have 
\begin{align*}
&\mathbb{Z}\Phi=\bigg\{ \lambda\in\mathfrak{h}^* \bigg| a_i,b_j\in\mathbb{Z},~\forall i,j \text{ and }  \mathop{\sum}\limits_{i=1}^{m+1}a_i+\mathop{\sum}\limits_{j=1}^{n+1}b_j=0     \bigg\},
\end{align*}
and
\begin{align*}
&\Lambda=\bigg\{  \lambda\in\mathfrak{h}^* \bigg| a_i,b_j\in\mathbb{Q},~ a_i-a_{i+1},b_j-b_{j+1}\in\mathbb{Z},  ~\forall i\leqslant m,j\leqslant n   ~\text{and}~  \mathop{\sum}\limits_{i=1}^{m+1}a_i+\mathop{\sum}\limits_{j=1}^{n+1}b_j=0   \bigg\}.
\end{align*}
By direct computation, we know that 
$$
2\Lambda\cap \mathbb{Z}\Phi=\begin{cases}
2\mathbb{Z}\Phi+\mathbb{Z}\left(\mathop{\sum}\limits_{i=1}^{m+1}(-1)^{i+1}\varepsilon_i+\mathop{\sum}\limits_{j=1}^{n+1}(-1)^j\delta_j\right), & \text{ if } m=2k,n=2l,\\
2\mathbb{Z}\Phi+\mathbb{Z}\mathop{\sum}\limits_{j=1}^{n+1}(-1)^j\delta_j, &\text{ if } m=2k,n=2l+1,\\
2\mathbb{Z}\Phi+\mathbb{Z}\mathop{\sum}\limits_{i=1}^{m+1}(-1)^{i+1}\varepsilon_i, & \text{ if } m=2k+1,n=2l,\\
2\mathbb{Z}\Phi+\mathbb{Z}\mathop{\sum}\limits_{i=1}^{m+1}(-1)^{i+1}\varepsilon_i+\mathbb{Z}\mathop{\sum}\limits_{j=1}^{n+1}(-1)^j\delta_j, & \text{ if } m=2k+1,n=2l+1,
\end{cases}
$$
for some non-negative integers $k,l$.
Then the algebra 
$$J_{\mathrm{ev}}(\mathfrak{g})=\begin{cases}
J^{m|n}_0\oplus \mathop{\prod}\limits_{i}x_i^{\frac{1}{2}}\mathop{\prod}\limits_{j}y_j^{\frac{1}{2}}J^{m|n}_{-(k+l+1)},& \text{if} ~m=2k,n=2l,\\
J^{m|n}_0\oplus \mathop{\prod}\limits_{j}y_j^{\frac{1}{2}}J^{m|n}_{-(l+1)},&\text{if} ~m=2k,n=2l+1,\\
J^{m|n}_0\oplus \mathop{\prod}\limits_{i}x_i^{\frac{1}{2}}J^{m|n}_{-(k+1)}, &\text{if} ~m=2k+1,n=2l,\\

J^{m|n}_0\oplus \mathop{\prod}\limits_{i}x_i^{\frac{1}{2}}J^{m|n}_{-(k+1)}\oplus \mathop{\prod}\limits_{j}y_j^{\frac{1}{2}}J^{m|n}_{-(l+1)}\oplus \mathop{\prod}\limits_{i}x_i^{\frac{1}{2}}\mathop{\prod}\limits_{j}y_j^{\frac{1}{2}}J^{m|n}_{-(k+l+2)},   &\text{if} ~m=2k+1,n=2l+1.
\end{cases}
$$
for some non-negative integers $k,l$. So it can be viewed as a subalgebra of $k\otimes_{\mathbb{Z}}J(\mathfrak{g})$ by $\iota\colon J_{\mathrm{ev}}(\mathfrak{g})\to k\otimes_{\mathbb{Z}}J(\mathfrak{g})$ with $\bbk_i\mapsto e^{-\alpha_i/2}$ and its image is coincide with $k\otimes \mathrm{Sch}(K_{\mathrm{ev}}(\mathfrak{g}))$.

{\bf$\boxed{A(n,n)~~ (n\neq 1)\text{ case:}}$}
In this case,  we set 
\begin{align*}
J(n)_0&=\left\{ f\in \mathbb{C}[x^{\pm1}, y^{\pm1}]_{0,0}^{\mathfrak{S}_{n+1}\times\mathfrak{S}_{n+1}}\left|x_i\frac{\partial f}{\partial x_i}+ y_j\frac{\partial f}{\partial y_j}\in(x_i-y_j) \right.\right\}
\end{align*}
where $\mathbb{C}[x^{\pm1},y^{\pm1}]_{0,0}$ is the quotient of the ring $\mathbb{C}[x^{\pm1},y^{\pm1}]$ with degree $0$ by the ideal $I=\left\langle\frac{x_1\cdots x_{n+1}}{y_1\cdots y_{n+1}}-1\right\rangle$. Then we have 
\begin{align*}
J_{\mathrm{ev}}(\mathfrak{g})=\begin{cases}
J(n)_0   &~\text{if} ~ n ~\text{ is even},\\
J(n)_0\oplus \bigg\{\overrightarrow{x}^{\frac{1}{2}}\mathop{\prod}\limits_{j,p}\left(1-\frac{x_j}{y_p}\right)g+I\bigg|g\in \mathbb{C}[x^{\pm1},y^{\pm1}]^{W},~\mathrm{deg}~g=-\frac{n+1}{2}\bigg\}  &~\text{if} ~ n ~\text{ is odd},
\end{cases}
\end{align*}
where $\overrightarrow{x}=x_1x_2\cdots x_{n+1}$ and $W=\mathfrak{S}_{n+1}\times\mathfrak{S}_{n+1}$. It can be viewed as a subalgebra by $\iota\colon J_{\mathrm{ev}}(\mathfrak{g})\to k\otimes_{\mathbb{Z}}J(\mathfrak{g})$ with $\bbk_i\mapsto e^{-\alpha_i/2}$ and its image is coincide with $k\otimes \mathrm{Sch}(K_{\mathrm{ev}}(\mathfrak{g}))$.

{\bf$\boxed{A(1,1)  \text{ case:}}$}
We have $J_{\mathrm{ev}}(A(1,1))=\left\{c+(u-v)g\left| g\in\mathbb{C}[u,v]\right.\right\}$ where $u=\left(\frac{x_1}{x_2}\right)^{\frac{1}{2}}+\left(\frac{x_2}{x_1}\right)^{\frac{1}{2}},~ v=\left(\frac{y_1}{y_2}\right)^{\frac{1}{2}}+\left(\frac{y_2}{y_1}\right)^{\frac{1}{2}}$. And $u-v=\bbk_{1}+\bbk^{-1}_{1}-\bbk_{3}-\bbk_{3}^{-1}\in J_{\mathrm{ev}}(A(1,1))$, but $u-v\notin J(A(1,1))$.

{\bf$\boxed{B(m,n),  m, n>0 \text{ case:}}$} We set $\lambda=\mathop{\sum}\limits_{i=1}^{m}\lambda_i\varepsilon_i+\mathop{\sum}\limits_{j=1}^{n} \mu_j\delta_j\in\mathfrak{h}^*$, then in this case
\begin{align*}
\mathbb{Z}\Phi=\left\{\lambda\in\mathfrak{h}^*\left|\lambda_i, \mu_j\in\mathbb{Z}, ~\forall i,j               \right.\right\}~\text{ and }~
\Lambda=\left\{ \lambda\in\mathfrak{h}^*\left|\mu_j\in\mathbb{Z},~\forall j  \text{ and all }\lambda_i\in\mathbb{Z} \text{ or all } \lambda_i\in\mathbb{Z}+\frac{1}{2}              \right.\right\}.
\end{align*}
So $2\Lambda\cap \mathbb{Z}\Phi=2\Lambda$. Let $u_i=x_i+x_i^{-1}$ and $v_j=y_j+y_j^{-1}$ for all possible $i,j$, then we have $J_{\mathrm{ev}}(\mathfrak{g})=J(\mathfrak{g})_0\oplus J(\mathfrak{g})_{1/2}$, where
\begin{align*}
J(\mathfrak{g})_0=\left\{  f\in\mathbb{C}[u_1,\cdots,u_m,v_1,\cdots,v_n]^{\mathfrak{S}_m\times\mathfrak{S}_n}\left|u_i\frac{\partial f}{\partial u_i}+v_j\frac{\partial f}{\partial v_j}\in(u_i-v_j)                  \right.\right\},
\end{align*}
and
\begin{align*}
J(\mathfrak{g})_{1/2}=\bigg\{   \prod\limits_{i=1}^m(x_i^{1/2}+x_i^{-1/2})\mathop{\prod}\limits_{i=1}^m\prod\limits_{j=1}^n(u_i-v_j)g \bigg|g\in  \mathbb{C}[u_1,\cdots,u_m,v_1,\cdots,v_n]^{\mathfrak{S}_m\times\mathfrak{S}_n}                         \bigg\}.
\end{align*}

{\bf$\boxed{B(0,n) \text{ case:}}$} In this case 
$\Lambda=\mathbb{Z}\Phi=\Big\{\mathop{\sum}\limits_{j=1}^{n} \mu_j\delta_j\Big|\mu_j\in\mathbb{Z},~\forall j   \Big\},$
so $2\Lambda\cap \mathbb{Z}\Phi=2\Lambda$ and this algebra $J_{\mathrm{ev}}(\mathfrak{g})=\mathbb{C}[v_1,v_2,\cdots,v_n]^{\mathfrak{S}_n}$, where the notation $v_i$ are the same as above.

{\bf$\boxed{C(n+1) \text{ case:}}$} In this case
\begin{align*}
&\Lambda=\bigg\{ \lambda\varepsilon+\mathop{\sum}\limits_{j=1}^{n} \mu_j\delta_j\bigg|\lambda\in\mathbb{C},\mu_j\in\mathbb{Z},~\forall j    \bigg\}
\end{align*}
and
\begin{align*}
&\mathbb{Z}\Phi=\bigg\{ \lambda\varepsilon+\mathop{\sum}\limits_{j=1}^{n} \mu_j\delta_j\bigg|\lambda,\mu_j\in\mathbb{Z},~\forall j   ~\text{and}~    \lambda+\mathop{\sum}\limits_{j=1}^{n} \mu_j          ~\text{is even} \bigg\}.
\end{align*}
So $2\Lambda\cap \mathbb{Z}\Phi=\Big\{ \lambda\varepsilon+\mathop{\sum}\limits_{j=1}^{n} \mu_j\delta_j\Big|\lambda,\mu_j\in2\mathbb{Z},~\forall j  \Big\}$ and the algebra 
\begin{align*}
J_{\mathrm{ev}}(\mathfrak{g})=\left\{ f\in \mathbb{C}[x^{\pm1},y_1^{\pm1},\cdots,y_{n+1}^{\pm1}]^W\left|      y_j\frac{\partial f}{\partial y_j}+x\frac{\partial f}{\partial x}\in(x-y_j)  \right.\right\}.
\end{align*}

{\bf$\boxed{D(m,n), m>1,n>0 \text{ case:}}$} Let $\lambda=\mathop{\sum}\limits_{i=1}^{m}\lambda_i\varepsilon_i+\mathop{\sum}\limits_{j=1}^{n} \mu_j\delta_j\in\mathfrak{h}^*$ and $u_i, v_j$ are as above, then 
\begin{align*}
&\Lambda=\left\{ \lambda\in\mathfrak{h}^*\left|\mu_j\in\mathbb{Z},~\forall j \text{ and all } \lambda_i\in\mathbb{Z} ~\text{or all}~ \lambda_i\in\mathbb{Z}+\frac{1}{2}  \right.\right\}
\end{align*}
and
\begin{align*}
&\mathbb{Z}\Phi=\bigg\{ \lambda\in\mathfrak{h}^*\bigg|\lambda_i,\mu_j\in\mathbb{Z},~\forall i,j  ~\text{and}~  \mathop{\sum}\limits_{i=1}^{m}\lambda_i+\mathop{\sum}\limits_{j=1}^{n} \mu_j ~\text{is even}                    \bigg\}.
\end{align*}
So 
$$
2\Lambda\cap \mathbb{Z}\Phi=\begin{cases}
2\mathbb{Z}\Phi+\mathbb{Z}\left(\mathop{\sum}\limits_{i=1}^n \varepsilon_i\right)+2\mathbb{Z}\varepsilon_n, &\text{ if }m=2k,\\
2\mathbb{Z}\Phi+2\mathbb{Z}\varepsilon_n, &\text{ if }m=2k+1,
\end{cases}
$$
for some positive integer $k$.  Thus the  algebra $J_{\mathrm{ev}}(\mathfrak{g})$ is, respectively, equal to $J(\mathfrak{g})_0\oplus J(\mathfrak{g})_{1/2}$ for $m=2k$ and $J(\mathfrak{g})_0$ for $m=2k+1$, where 
\begin{align*}
J(\mathfrak{g})_0=\left\{  f\in\mathbb{C}[x^{\pm1}_1,\cdots,x^{\pm1}_m,y^{\pm1}_1,\cdots,y^{\pm1}_n]^{W}\left|x_i\frac{\partial f}{\partial x_i}+y_j\frac{\partial f}{\partial y_j}\in(x_i-y_j)                  \right.\right\},
\end{align*}
and
\begin{align*}
J(\mathfrak{g})_{1/2}=\bigg\{  \mathop{\prod}\limits_{i,j}(u_i-v_j)\Big((x_1x_2\cdots x_m)^{1/2}\mathbb{C}[x^{\pm1}_1,\cdots,x^{\pm1}_m,y^{\pm1}_1,\cdots,y^{\pm1}_n]\Big)^W                     \bigg\}.
\end{align*}

{\bf$\boxed{D(2,1,\alpha) \text{ case:}}$} In this case,
$$
\Lambda=\bigg\{ \mathop{\sum}\limits_{i=1}^3\lambda_i\varepsilon_i\bigg| \lambda_i\in\mathbb{Z}, ~\forall i         \bigg\},\text{ and }
\mathbb{Z}\Phi=\bigg\{   \mathop{\sum}\limits_{i=1}^3\lambda_i\varepsilon_i\bigg| \lambda_i\in\mathbb{Z} ~\text{and}~\lambda_i-\lambda_j\in2\mathbb{Z},     ~\forall i,j   \bigg\}.
$$
So $2\Lambda\cap \mathbb{Z}\Phi=2\Lambda$. Thus the algebra 
$$J_{\mathrm{ev}}(\mathfrak{g})=\begin{cases}
\big\{c+\Delta h | c\in\mathbb{C},h\in\mathbb{C}[u_1,u_2,u_3] \big\},&\text{ if } \alpha \text{ is not rational}, \\
\big\{g(w_{\alpha})+\Delta h| g\in\mathbb{C}[\omega],h\in\mathbb{C}[u_1,u_2,u_3] \big\}, &\text{ if }\alpha=p/q \text{ with } p\in\mathbb{Z},q\in\mathbb{N},
\end{cases}
$$
where \begin{align*}
\Delta=u_1^2+u_2^2+u_3^2-u_1u_2u_3-4,~ u_i=x_i+x_i^{-1}, \text{ for }i=1,2,3,
\end{align*} 
and \begin{align*}
\omega_{\alpha}=(x_1+x_1^{-1}-x_2x_3-x_2^{-1}x_3^{-1})\frac{(x_2^p-x_2^{-p})(x_3^{q}-x_3^{-q})}{(x_2-x_2^{-1})(x_3-x_3^{-1})}+x_2^px_3^{-q}+x_2^{-p}x_3^q.
\end{align*}

{\bf$\boxed{F(4) \text{ case:}}$} In this case,
\begin{align*}
&\Lambda=\bigg\{ \mu\delta+\mathop{\sum}\limits_{i=1}^3\lambda_i\varepsilon_i\bigg|\text{all}~\lambda_i\in\mathbb{Z} ~\text{or}~\text{all}~\lambda_i\in\mathbb{Z}+\frac{1}{2} ,~  2\mu\in\mathbb{Z}              \bigg\},
\end{align*}
and 
\begin{align*}
&\mathbb{Z}\Phi=\bigg\{ \mu\delta+\mathop{\sum}\limits_{i=1}^3\lambda_i\varepsilon_i\bigg|\text{all}~\lambda_i,~\mu\in\mathbb{Z} ~\text{or}~\text{all}~\lambda_i,~\mu\in\mathbb{Z}+\frac{1}{2}  \bigg\}.
\end{align*}
So $2\Lambda\cap \mathbb{Z}\Phi=2\Lambda$, and the algebra
\begin{align*}
J_{\mathrm{ev}}(\mathfrak{g})=\Big\{  g(\omega_1,\omega_2)+\Delta h\Big| h\in\mathbb{C}[x_1^{\pm2},x_2^{\pm2},x_3^{\pm2},(x_1x_2x_3)^{\pm1},y^{\pm1}]^W, ~g\in\mathbb{C}[\omega_1,\omega_2]     \Big\},
\end{align*}
where
\begin{align*}
\Delta=\left(y+y^{-1}-x_1x_2x_3-x_1^{-1}x_2^{-1}x_3^{-1}\right)\mathop{\prod}\limits_{i=1}^3\left(y+y^{-1}-\frac{x_1x_2x_3}{x_i^2}-\frac{x_i^2}{x_1x_2x_3}\right),
\end{align*}
and
\begin{align*}
\omega_k=\mathop{\sum}\limits_{1\leqslant i<j\leqslant 3}\left(x^{2k}_i+x^{-2k}_i+\frac{1}{2}\right)\left(x^{2k}_j+x^{-2k}_j+\frac{1}{2}\right)-\frac{3}{4}+y^{2k}+y^{-2k}-(y^k+y^{-k})\mathop{\prod}\limits_{i=1}^3\left(x_i^k+x_i^{-k}\right)
\end{align*}
with $k=1,2$, and $W=\mathbb{Z}_2\times (\mathfrak{S}_3\ltimes \mathbb{Z}_2^3)$.

{\bf$\boxed{G(3) \text{ case:}}$} In this case,
$\Lambda=\mathbb{Z}\Phi=\big\{ \lambda_1\varepsilon_1+\lambda_2\varepsilon_2+\mu\delta|\lambda_1,\lambda_2,\mu\in\mathbb{Z}    \big\}$.
So $2\Lambda\cap \mathbb{Z}\Phi=2\Lambda$, and the algebra
\begin{align*}
J_{\mathrm{ev}}(\mathfrak{g})=\bigg\{g(\omega)+\mathop{\prod}\limits_{i=1}^3(v-u_i)h\bigg|h\in\mathbb{C}[v,u_1,u_2,u_3]^{\mathfrak{S}_3},~g\in\mathbb{C}[\omega]   \bigg\},
\end{align*}
where 
$$\omega=v^2-v(u_1+u_2+u_3+1)+u_1u_2+u_1u_3+u_2u_3.$$
and the notations $u_i, v$ are the same as above.
\section*{Acknowledgements}
We would like to express our debt to Shun-Jen Cheng,  Hiroyuki Yamane, Hechun Zhang, and  Ruibin Zhang for many insightful discussions.  This paper was partially written up during the second author visit to Institute of Geometry and Physics,  USTC, in Summer 2021,  from which we gratefully acknowledge the support and excellent working environment where most of this work was completed.  Y.  Wang is partially supported by the National Natural Science Foundation of China (Nos. 11901146 and 12071026),  and the Fundamental Research Funds for the Central Universities JZ2021HGTB0124. Y. Ye is partially supported by the National Natural Science Foundation of China (No. 11971449).

\bibliographystyle{alpha}

\end{document}